\documentclass[10pt]{amsart}

\usepackage[utf8]{inputenc}
\usepackage[english]{babel}
\mathchardef\mhyphen="2D

\usepackage{amsthm, amsmath, amssymb, stmaryrd, mathrsfs, tikz, MnSymbol, bbm, mathtools, amsfonts}
\usepackage{enumerate}
\usepackage{relsize}
\usepackage{blindtext}

\usepackage{hyperref}
\hypersetup{
    colorlinks=true,
    linkcolor=blue,
    citecolor=red
}
\usepackage{cleveref}

\providecommand*\showkeyslabelformat[1]

\usepackage[notref,notcite]{showkeys}

\usetikzlibrary{decorations.pathreplacing,angles,quotes}

\usetikzlibrary{matrix, arrows}

\usepackage[letterpaper, margin=1.1in]{geometry}

\usepackage{lipsum}
\usepackage[titletoc,title]{appendix}

\usepackage{multicol}
  
\theoremstyle{definition}
\newtheorem{theorem}{Theorem}
\newtheorem{definition}[theorem]{Definition}

\newtheorem{lemma}[theorem]{Lemma}

\newtheorem{proposition}[theorem]{Proposition}

\newtheorem{example}[theorem]{Example}

\newtheorem{remark}[theorem]{Remark}
\newtheorem{remarks}[theorem]{Remarks}

\newtheorem{notation}[theorem]{Notation}
\newtheorem{conj}[theorem]{Conjecture}

\newtheorem*{con}{Conjecture}

\newtheorem*{coher+comp}{The $n$-cell composition coherence condition for $n$-categories}

\numberwithin{theorem}{section}

\def\bs{\ensuremath\boldsymbol}

\makeatletter
\newtheorem*{rep@theorem}{\rep@title}
\newcommand{\newreptheorem}[2]{%
\newenvironment{rep#1}[1]{%
 \def\rep@title{#2 \ref{##1}}%
 \begin{rep@theorem}}%
 {\end{rep@theorem}}}
\makeatother

\newreptheorem{theorem}{Theorem}

\title{Higher Categories and Slices of Globular Operads}

\author{Rhiannon Griffiths}

\begin{document}

\begin{abstract} 
In an unpublished preprint \cite{batanin}, Batanin conjectures that it is possible to take `slices' of a globular operad, thereby isolating the algebraic structure in each dimension. It was further hypothesised that the slices of a globular operad for some theory of higher category contain essential information about those higher categories, namely whether or not they are equivalent to the fully weak variety. In this paper, we use the theory of presentations for globular operads developed in \cite{Me} to provide a concrete definition of slices, and calculate the slices for several key theories of $n$-category.
\end{abstract}

\maketitle

\tableofcontents

\section{Introduction}

The biggest obstruction to working with weak $n$-categories is that their complexity increases so rapidly with $n$. Yet, it is weak higher categories that typically have the most interesting and useful applications. Grothendieck's Homotopy Hypothesis, for instance, states that homotopy $n$-types are modeled by $n$-groupoids, and by extension that topological spaces are modeled by $\infty$-groupoids. However, Simpson showed that \textit{strict} $n$-groupoids do not contain sufficient information to model spaces with nontrivial Whitehead products \cite[\S 4.4]{simpson}. For $n \leqslant 2$, weak $n$-categories and strict $n$-categories are equivalent, but the same is not true in dimension 3. Instead, things like the 3-dimensional Homotopy Hypothesis are stated in terms of Gray Categories -- a type of semi-strict 3-category for which composition of 2-cells  is strictly associative and unital but the interchange law does not hold. The proof that Gray categories are equivalent to fully weak 3-categories is extremely non-trivial \cite{NG}, and analogous proofs for higher dimensional categories would be even more so. However, in dimensions above 4, any direct work involving weak $n$-categories is impossible since there are no hands-on definitions of fully weak $n$-category for any $n > 4$ in the current literature. One way to overcome this difficulty is to find some theory of $n$-category that is as strict as it is possible to be while still being equivalent to the fully weak variety; the idea is to retain enough structure to be useful without being too complicated to work with directly. 

This work was inspired by a conjecture of Michael Batanin on the subject of globular operads. A globular operad is a kind of operad whose algebras bear a strong formal resemblance to higher categories. Significant progress has been made with this approach to $n$-categories, particularly by Batanin \cite{MB} and by Leinster \cite{TL}. In \cite{batanin}, Batanin described what he called the `slices' of a globular operad -- a series of symmetric operads capturing the internal structure of the algebras in each dimension. In particular, when $\bs{G}$ is the globular operad for some theory of $n$-category, the $k^{th}$ slice should be the symmetric operad corresponding to the algebraic structure formed by the $k$-cells of such an $n$-category whenever it contains exactly one $j$-cell, for all $j<k$. Note that this single $j$-cell must be the identity cell for all $0<j<k$. Thus, for \text{strict} higher categories, the first slice should be the operad for monoids, since the arrows of a one-object strict $n$-category form a monoid. Similarly, by the Eckmann-Hilton argument (discussed in Section \ref{pres for symmop section}), the $k^{th}$ slice for all $k>1$ should be the symmetric operad for commutative monoids. In the case of higher groupoids, slices are closely related to homotopy groups; the $k^{th}$ homotopy group of an $n$-groupoid $\mathcal{A}$ at an object $A$ is given by equipping the $k$-cells of $\mathcal{A}$ whose lower-dimensional source and target cells are all identities at $A$ with the equivalence relation\footnote{The fact that the $(k+1)$-cells are weakly invertible in a groupoid is precisely what makes this identification into an equivalence relation; see \cite{CS}.}  defined by, 
\begin{center}
$\alpha \sim \beta \iff$ there exists a $(k+1)$-cell $\alpha \longrightarrow \beta$ in $\mathcal{A}$.
\end{center}
Batanin's conjecture, stated in terms of formal objects known as computads, hypothesises that slices can be used to determine when semi-strict higher categories are equivalent to the fully weak variety:
\begin{con}\cite{batanin} Let $\bs{G}$ be a contractible $n$-globular operad.
\begin{enumerate}
\item If the $k^{th}$ slice of $\bs{G}$ is a kind of symmetric operad known as a \textit{plain} operad for all $k<n$, then the associated category of computads is a presheaf category.
\item If the category of computads associated to $\bs{G}$ is a presheaf category, then the algebras for $\bs{G}$ are equivalent to fully weak $n$-categories.
\end{enumerate}
\end{con}
Given any operadic theory of higher category, the associated computads may be understood as a formalisation of the string diagrams for those higher categories. Since objects in presheaf categories are diagrams in $\mathbf{Set}$, they can often be expressed geometrically; think for example of the triangles, tetrahedrons, and higher dimensional simplices forming a simplicial set. On the surface, this makes it seem as though categories of computads should all be presheaf categories, since they could all be expressed in the language of string diagrams. However, it was observed by Makkai and Zawadowski in \cite{M+Z} that the computads for strict $n$-categories are \textit{not} presheaves for any $n \geqslant 3$. Informally, this because for any 2-cells $\alpha:1_A \longrightarrow 1_A$ and $\beta:1_A \longrightarrow 1_A$ between matching identity 1-cells in a strict $n$-category, the horizontal composites $\alpha \ast \beta$ and $\beta \ast \alpha$, and the vertical composites $\alpha \cdot \beta$ and $\beta \cdot \alpha$ are all equal, since the collection of all such cells forms a commutative monoid. This means that the string diagrams below all represent the \textit{same} 2-cell. 
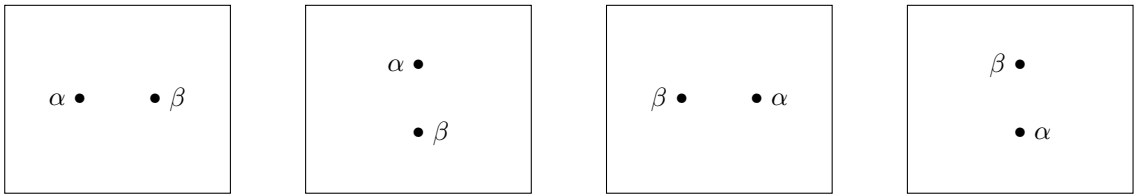
\vspace{1mm}
\begin{center}
\begin{tikzpicture}[node distance=5cm, auto]

\draw(0,0) rectangle (3,-2.5);

\node at (1,-1.25) {$\bullet$};
\node at (2,-1.25) {$\bullet$};

\node at (0.7,-1.25) {$\alpha$};
\node at (2.3,-1.25) {$\beta$};

\draw(4,0) rectangle (7,-2.5);

\node at (5.5,-0.8) {$\bullet$};
\node at (5.5,-1.7) {$\bullet$};

\node at (5.2,-0.8) {$\alpha$};
\node at (5.8,-1.7) {$\beta$};

\draw(8,0) rectangle (11,-2.5);

\node at (9,-1.25) {$\bullet$};
\node at (10,-1.25) {$\bullet$};

\node at (10.3,-1.25) {$\alpha$};
\node at (8.7,-1.25) {$\beta$};

\draw(12,0) rectangle (15,-2.5);

\node at (13.5,-0.8) {$\bullet$};
\node at (13.5,-1.7) {$\bullet$};

\node at (13.2,-0.8) {$\beta$};
\node at (13.8,-1.7) {$\alpha$};

\end{tikzpicture}
\end{center}
We can picture this as $\alpha$ and $\beta$ being allowed by the interchange law to move around each other in the empty space left by the identity axioms. Without giving all the details, the problem that arises here is that there is no way to place a consistent order on the 2-cells contained in string diagrams like this, and as a result, we cannot construct a presheaf category wherein the diagrams above are recognised as being equivalent. 
It is thought that string diagrams like this being equivalent is indicative of some homotopy degeneracy like the vanishing Whitehead products in the corresponding higher categories, meaning that those higher categories cannot contain sufficient information to model the fully weak variety.

Although Batanin gave an abstract definition of slices in \cite{batanin}, this definition did not provide the desired results when applied in practice. For instance, when applied to Gray categories, the resulting second slice is the symmetric operad for double monoids with shared unit \cite[Example 3.4]{batanin}. However, it is a straightforward exercise to show that the 2-cells of a Gray category with single 0-cell and a single identity 1-cell just form an ordinary monoid. At the time, it was not possible to formulate a concrete definition of slices due to the gaps in knowledge surrounding globular operads.

In this paper, we formulate a hands-on definition of slices by exploiting the theory of presentations for globular operads that we developed in \cite{Me}. While our definition is highly technical and requires building some additional machinery to state formally, it is fairly intuitive to apply in practice. This is demonstrated in the final section, where we calculate slices for two theories of semi-strict higher category; the first are $n$-categories with weak interchange laws, and the second are $n$-categories with weak identities. Many authors have speculated that one or both of these are equivalent to fully weak higher categories; see, for example, \cite{Crans1}, \cite{Crans2}, \cite{Dorn}, \cite{HeidemannReutterVicary}, \cite{ReutterVicary} for the former case, and \cite{AK}, \cite{CS} for the latter. 

\begin{reptheorem}{slices_for_weak_interchange}
For $k<n$, the $k^{th}$ slice of the $n$-globular operad for $n$-categories with weak interchange laws is the \textit{plain} operad for (many-pointed) $k$-tuple monoids with shared unit.
\end{reptheorem}

\begin{reptheorem}{slices_for_weak_units}
For $k<n$, the $k^{th}$ slice of the $n$-globular operad for $n$-categories with weak units is the \textit{non-plain} operad for (many-pointed) $k$-tuple semigroups, satisfying some additional axioms.
\end{reptheorem}

By Theorem \ref{slices_for_weak_interchange}, Batanin's conjecture says that $n$-categories with weak interchange laws do indeed model weak $n$-categories. On the other hand, Theorem \ref{slices_for_weak_units} shows that the slices of the globular operad for higher categories with weak identities are \textit{not} plain operads. To illustrate how the second part of Batanin's conjecture may still be relevant in this case, we show that the slices for higher categories with weak identity laws have a nice geometric representation that is strongly related to the corresponding string diagrams, suggesting that these diagrams do form presheaf categories. The string diagrams associated to categories with weak identity laws will be studied in more detail in future work \cite{RG3}. 

In \cite{RG}, Garner defines weak functors between operadic higher categories in terms of computads. In more follow up work \cite{RG2}, we use this approach to construct a weak homomorphism into any semi-strict higher category from a fully weak one, and investigate the relationship between semi-strict theories of higher category whose associated computads form presheaf categories, and fully weak higher categories. In particular, we consider Batanin's claim that this is precisely the condition needed to guarantee that a given variety of higher category contains sufficient information to model fully weak higher categories. We also use the definition of slices in this paper to show that if the slices of a globular operad for $n$-categories are plain operads for all $k<n$, then the arising computads are presheaves.

An important feature of operadic higher categories is that they are fully algebraic, meaning that composition and its associated coherence are given by specified operations satisfying equational axioms. Algebraic higher categories tend to have the most powerful applications to algebraic topology and homotopy theory. 
The best solution to the Homotopy Hypothesis, for example, would be to show that \textit{algebraic} $n$-groupoids are equivalent to homotopy $n$-types, since this would show that some completely algebraic structures are the same as purely topological ones. Variations on the Homotopy Hypothesis are still being actively researched. In an ongoing series of papers, Gurski, Johnson and Osorno have made significant progress towards proving the stable Homotopy Hypothesis in low dimensions; see \cite{jo} and \cite{gjo}. As another example, we have the Cobordism Hypothesis of Baez and Dolan \cite{baez and dolan}, which concerns the classification of extended topological quantum field theories (TQFTs). 

\subsection{Outline of the paper} In Section \ref{T-ops and algebras subsection}, we introduce $T$-operads and their algebras, with particular focus on plain operads. Section \ref{free T-operad section} introduces free $T$-operads, which will be of great use in later sections. We recall the more complex notion of globular operad in Section \ref{GlobOpSubSection}, and highlight their similarity to algebraic theories of higher category. Presentations for globular operads are introduced in Section \ref{presSubsec}, allowing us to build the operads for the higher categories we are interested in. We discuss how plain operads may be generalised to symmetric operads in Section \ref{symm vs plain op section}, and show how symmetric operads can be constructed using the generators and relations of a presentation in Section \ref{pres for symmop section}. In Section \ref{slices subsection}, we use presentations to formulate a concrete definition of slices for globular operads, and prove that the slices do not depend on the choice of presentation. In Section \ref{examples of slices subsection} we calculate the slices for several key theories of higher category.

\section{$T$-operads}\label{T-Op Section}

The operads in this paper are all either $T$-operads or symmetric operads. The latter of these may be viewed as a generalisation of the prototypal example of $T$-operads, known as \textit{plain} operads, detailed in Example \ref{plain op}. 

\subsection{$T$-operads and their algebras}\label{T-ops and algebras subsection}

In this section we provide some background on $T$-operads, which are defined for any cartesian monad $T$ on a finitely complete category $\mathcal{S}$. 

\begin{definition} Let $\mathcal{S}$ be a category with a terminal object $1$ and let $T : \mathcal{S} \longrightarrow \mathcal{S}$ be an endofunctor. The \textit{category $T \mhyphen \mathbf{Coll}$ of $T$-collections} is the slice category $\mathcal{S} / T1$. 
\end{definition}

Recall that a category is finitely complete if and only if it has pullbacks and a terminal object \cite[5.23]{SA}. With this in mind we give the following definition.

\begin{definition}\label{circ functor} Let $T=(T,\eta,\mu)$ be a monad on a finitely complete category $\mathcal{S}$. The functor $\circ:T \mhyphen \mathbf{Coll} \times T \mhyphen \mathbf{Coll} \longrightarrow T \mhyphen \mathbf{Coll}$ sends a pair $(c:C \longrightarrow T1, \ d:D \longrightarrow T1)$ of $T$-collections to the composite of the left hand diagonals in the diagram
\begin{center}
\begin{tikzpicture}[node distance=1.4cm, auto]

\node (A) {$C \circ D$};
\node (TA) [left of=A, below of=A] {$TC$};
\node (X) [left of=TA, below of=TA] {$T^21$};
\node (Y) [left of=X, below of=X] {$T1$};
\node (B) [right of=A, below of=A] {$D$};
\node (C) [node distance=2.8cm, below of=A] {$T1$};

\draw[->] (A) to node [swap] {} (TA);
\draw[->] (A) to node {} (B);
\draw[->] (TA) to node [swap] {$T!$} (C);
\draw[->] (B) to node {$d$} (C);
\draw[->] (TA) to node [swap] {$Tc$} (X);
\draw[->] (X) to node [swap] {$\mu_1$} (Y);

\end{tikzpicture}
\end{center}
where the upper square is a pullback square, and is defined on pairs of morphisms using the universal property of pullbacks. 
\end{definition}

\begin{definition}
A monad $T=(T,\eta, \mu)$ on a category $\mathcal{S}$ is \textit{cartesian} if
\begin{enumerate}[i)]
\item $\mathcal{S}$ has pullbacks;
\item $T$ preserves pullbacks; and
\item $\eta$ and $\mu$ are cartesian, i.e., the naturality squares of $\eta$ and $\mu$ are all pullbacks.
\end{enumerate}
\end{definition}

When $T=(T,\eta, \mu)$ is a cartesian monad on a category with a terminal object $1$, $(T \mhyphen \bf{Coll}, \circ)$ is a monoidal category with unit $\eta_1:1 \longrightarrow T1$. The construction of the left and right unit isomorphisms is straightforward, and uses the fact that $\eta$ is cartesian. The construction of the associativity isomorphisms is slighty more complex, and uses that $T$ preserves pullbacks and that $\mu$ is cartesian.

\begin{definition}\label{T-operad def} Let $T$ be a cartesian monad on a finitely complete category $\mathcal{S}$. The \textit{category $T \mhyphen \mathbf{Op}$ of $T$-operads} is the category of monoids in the monoidal category $(T \mhyphen \bf{Coll}, \circ)$.
\end{definition}

Unpacking the definition above, we see that a $T$-operad $C = (C,c,ids,comp)$ is a morphism $c:C \longrightarrow T1$ in $\mathcal{S}$ together with an identity map $ids:1 \longrightarrow C$ and a composition map $comp:C \circ C \longrightarrow C$ making the following diagrams commute
\begin{center}
\begin{tikzpicture}[node distance=2.8cm, auto]

\node (X) {$1$};
\node (Y) [right of=X] {$C$};
\node (A) [node distance=1.4cm, below of=X, right of=X] {$T1$};

\node (U) [node distance=2cm, right of=Y] {$C \circ C$};
\node (B) [node distance=1.4cm, below of=U, right of=U] {$TC$};
\node (C) [node distance=2.2cm, right of=B] {$T^21$};
\node (D) [node distance=2.2cm, right of=C] {$T1$};
\node (V) [node distance=1.4cm, above of=D, right of=D] {$C$};

\draw[->] (X) to node {$ids$} (Y);
\draw[->] (X) to node [swap] {$\eta_{1}$} (A);
\draw[->] (Y) to node {$c$} (A);

\draw[->] (U) to node {$comp$} (V);
\draw[->] (U) to node [swap] {} (B);
\draw[->] (B) to node [swap] {$Tc$} (C);
\draw[->] (C) to node [swap] {$\mu_{1}$} (D);
\draw[->] (V) to node {$c$} (D);

\end{tikzpicture}
\end{center}
and satisfying the associativity and identity axioms. Similarly, a morphism of $T$-operads is a morphism of the underlying collections preserving the identity and composition.

We are interested in the algebras for a $T$-operad, which may be defined as the algebras of the following associated monad.

\begin{definition} \label{operad-monad} Let $T$ be a cartesian monad on a complete category $\mathcal{S}$. A $T$-operad $C = (C, c, ids, comp)$ induces a monad $(-)_C$ on $\mathcal{S}$. The underlying endofunctor sends an object $X \in \mathcal{S}$ to the pullback object
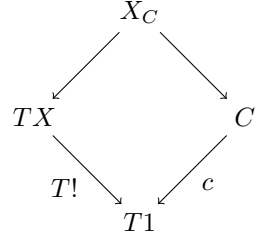
\begin{center}
\begin{tikzpicture}[node distance=1.4cm, auto]

\node (X) {$X_C$};
\node (TX) [left of=X, below of=X] {$TX$};
\node (C) [right of=X, below of=X] {$C$};

\draw[->] (X) to node [swap] {} (TX);
\draw[->] (X) to node {} (C);

\node (TC') [node distance=2.8cm, below of=X] {$T1$};

\draw[->] (TX) to node [swap] {$T!$} (TC');
\draw[->] (C) to node {$c$} (TC');

\end{tikzpicture}
\end{center}
and is defined on morphisms using the universal property of pullbacks. The same property is used to construct the unit $X \longrightarrow X_C$ and multiplication $(X_C)_C \longrightarrow X_C$ of the monad at $X$. The monad axioms are satisfied by the identity and associativity axioms for $C$. 
\end{definition}

\begin{definition}\label{operad-algebras}
Let $C$ be a $T$-operad. The \textit{category $C \mhyphen \bf{Alg}$ of algebras for $C$}, or $C$-algebras, is the category of algebras for the monad $(-)_C$.
\end{definition}

As a first example, we cover a variety of $T$-operad that will be used extensively throughout this paper; namely $T$-operads for which $T$ is the free monoid monad on the category $\mathbf{Set}$ of sets. Before proceeding, let us recall some definitions from universal algebra.

An \textit{algebraic theory}\label{reg/fin alg theories} on $\mathbf{Set}$ is a collection of abstract operations satisfying equational axioms. The \textit{arity} of an operation is the number of inputs of the operation. An algebraic theory is called \textit{finitary} if all of its operations have finite arity. The theory of groups, for example, is a finitary algebraic theory with an operation of arity 2 (a binary operation), an operation of arity 1 (a unary operation) and an operation of arity 0 (a nullary operation). These operations give group multiplication, inverses, and the identity element, respectively, and the equational axioms they satisfy encode the group axioms. A \textit{model} of an algebraic theory is a set equipped with operations corresponding to those of the theory, and which satisfy the same equational axioms. Thus, a model for the theory of groups on a set $X$ is a precisely a group with underlying set $X$. 
A more detailed overview can be found in texts such as \cite[\S II.1]{SB+HPS}, although the terminology differs slightly.

An equation made up of variables and finitary operation symbols is called \textit{regular} if the same variables appear without repetition on each side. So the equations
$$x \cdot (y \cdot z) = (x \cdot y) \cdot z  \qquad \quad x \cdot 1 = x = 1 \cdot x \qquad \quad x \cdot y =y \cdot x$$
are all regular\footnote{Note that the `1' here does not denote a variable, but an element specified by a nullary operation.}, but the equations
$$x \cdot x^{-1} = 1 \qquad \quad x \cdot (y + z) = (x \cdot y) + (x \cdot z)$$
are not\footnote{Universal algebraists use the term `regular equation' for an equation in which the same variables appear on each side, with repetitions allowed; see, for example, \cite{BJ+EN}}. An equation made up of variables and finitary operation symbols is called \textit{strongly} regular if the same variables appear \textit{in the same order}, without repetition on each side. So of the three regular equations above, only the equations 
$$x \cdot (y \cdot z) = (x \cdot y) \cdot z  \qquad \quad x \cdot 1 = x = 1 \cdot x$$
are strongly regular; see \cite[\S 3]{AC+PJ} or \cite[2.2.5]{TL}. An algebraic theory is called regular if it can be presented by operations and regular equations, and strongly regular if it can be presented by operations and strongly regular equations. From the examples here we can intuit that neither the theory of groups nor the theory of rings is regular, but the theory of commutative monoids is regular, and the theory of monoids is strongly regular.

\begin{example}\label{plain op}
Plain operads, also called non-symmetric operads, are $T$-operads for which $T$ is the free monoid monad on $\mathbf{Set}$; see \cite[1.4(ii)]{TL1} for a proof that this monad is cartesian. Each plain operad is equivalent to a strongly regular finitary algebraic theory on $\mathbf{Set}$ \cite[C.1.1]{TL}, and the algebras are equivalent to models of that theory. To see this, observe that in this context $T1$ is isomorphic to the set $\mathbb{N}$ of natural numbers, so the underlying collection of a plain operad $C = (C,c,ids,comp)$ is a function $c:C \longrightarrow \mathbb{N}$. For each $m \in \mathbb{N}$, we think of the set $c^{-1}(m)$ as a set of abstract operations of arity $m$. The identity map $ids:1 \longrightarrow C$ picks out a unary operation which we denote by $\text{id}$ and refer to as the identity. The elements of $C \circ C$ are tuples $(\rho_1,...,\rho_m, \, \rho)$ where $\rho, \rho_i \in C$ and $c(\rho) = m$, so $\rho$ has arity $m$. The composition map $comp:C \circ C \longrightarrow C$ takes an element $(\rho_1,...,\rho_m, \, \rho) \in C \circ C$ to an element $\rho \circ (\rho_1,...,\rho_m) \in C$ satisfying 
$$c \, (\rho \circ (\rho_1,...,\rho_m)) = c(\rho_1) + ... + c(\rho_m).$$
We think of $\rho \circ (\rho_1,...,\rho_m)$ as the abstract operation of arity $c(\rho_1) + ... + c(\rho_m)$ given by first using the $\rho_i$s and then applying $\rho$ to the result. 

We visualise these operations using tree diagrams. In particular, an operation $\rho$ of arity $m$ is represented by drawing $m$ strings entering $\rho$ from above, and a single string exiting $\rho$ from below. So when $m = 3$ we have,
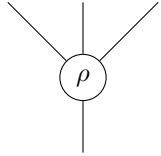
\begin{center}
\begin{tikzpicture}[node distance=5cm, auto]

\node[circle,draw] (A) at (0,0) {$\rho$};

\draw(A) to (0,-1);
\draw(A) to (1,1);
\draw(A) to (0,1);
\draw(A) to (-1,1);

\end{tikzpicture}
\end{center}
and given operations $\rho_1$, $\rho_2$ and $\rho_3$ of arities 4, 0, and 2, respectively, their composite $\rho \circ (\rho_1, \rho_2, \rho_3)$ is an operation of arity $4+0+2=6$ illustrated by the tree
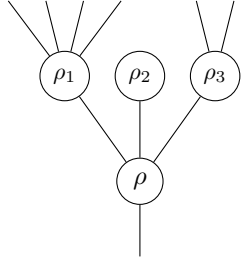
\vspace{1mm}
\begin{center}
\begin{tikzpicture}[node distance=5cm, auto]

\node[circle,draw] (A) at (0,0) {$\rho$};
\node[circle,draw, scale=0.9] (X) at (-1,1.4) {$\rho_1$};
\node[circle,draw, scale=0.9] (Y) at (0,1.4) {$\rho_2$};
\node[circle,draw, scale=0.9] (Z) at (1,1.4) {$\rho_3$};

\draw(A) to (0,-1);
\draw(A) to (X);
\draw(A) to (Y);
\draw(A) to (Z);

\draw(X) to (-1.75,2.4);
\draw(X) to (-1.25,2.4);
\draw(X) to (-0.75,2.4);
\draw(X) to (-0.25,2.4);

\draw(Z) to (0.75,2.4);
\draw(Z) to (1.25,2.4);

\end{tikzpicture}
\end{center}
of height 2. The operad axioms, expressed by the equalities below, ensure that composition in $C$ is appropriately associative and unital, 
$$\rho \circ \big(\rho_1 \circ (\rho_{11},...,\rho_{1k_1}),...,\rho_m \circ (\rho_{m1},...,\rho_{mk_m})\big) = \big(\rho \circ (\rho_1,...,\rho_m)\big) \circ (\rho_{11},...,\rho_{1k_1},...,\rho_{n1},...,\rho_{mk_m}) $$
$$ \rho \circ (\text{id},...,\text{id}) = \rho = \text{id} \circ (\rho) .$$
The associativity axiom means that we can draw tree diagrams of arbitrary height without ambiguity. For example, the tree
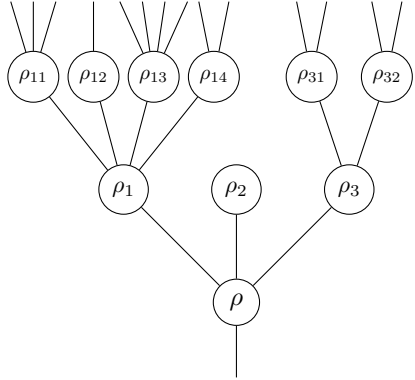
\begin{center}
\begin{tikzpicture}[node distance=5cm, auto]

\node[circle,draw] (A) at (0,0) {$\rho$};
\node[circle,draw, scale=0.9] (X) at (-1.5,1.5) {$\rho_1$};
\node[circle,draw, scale=0.9] (Y) at (0,1.5) {$\rho_2$};
\node[circle,draw, scale=0.9] (Z) at (1.5,1.5) {$\rho_3$};

\draw(A) to (0,-1);
\draw(A) to (X);
\draw(A) to (Y);
\draw(A) to (Z);

\node[circle,draw, scale=0.8] (Q) at (-2.7,3) {$\rho_{11}$};
\node[circle,draw, scale=0.8] (W) at (-1.9,3) {$\rho_{12}$};
\node[circle,draw, scale=0.8] (E) at (-1.1,3) {$\rho_{13}$};
\node[circle,draw, scale=0.8] (R) at (-0.3,3) {$\rho_{14}$};

\draw(X) to (Q);
\draw(X) to (W);
\draw(X) to (E);
\draw(X) to (R);

\node[circle,draw, scale=0.8] (T) at (1,3) {$\rho_{31}$};
\node[circle,draw, scale=0.8] (U) at (2,3) {$\rho_{32}$};

\draw(Z) to (T);
\draw(Z) to (U);

\draw(Q) to (-2.4,4);
\draw(Q) to (-2.7,4);
\draw(Q) to (-3,4);

\draw(W) to (-1.9,4);

\draw(E) to (-0.65,4);
\draw(E) to (-0.95,4);
\draw(E) to (-1.25,4);
\draw(E) to (-1.55,4);

\draw(R) to (-0.5,4);
\draw(R) to (-0.1,4);

\draw(T) to (0.8,4);
\draw(T) to (1.2,4);

\draw(U) to (1.8,4);
\draw(U) to (2.2,4);

\end{tikzpicture}
\end{center}
of height 3 represents a unique element of $C$, since the two ways of composing this diagram are equal. 
If we represent the identity element $\text{id} \in C$ as the tree diagram of height 0, i.e., as a single undecorated vertical line, then the identity axiom gives us intuitive equalities of diagrams like the one below.
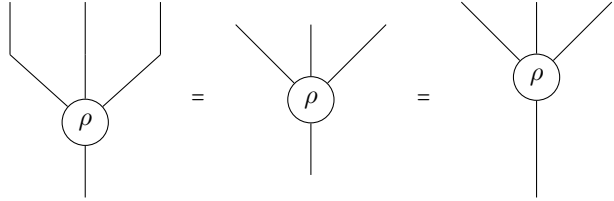
\vspace{1mm}
\begin{center}
\begin{tikzpicture}[node distance=5cm, auto]

\node[circle,draw] (A) at (0,0) {$\rho$};

\draw(A) to (0,-1);
\draw(A) to (1,0.9);
\draw(A) to (0,0.9);
\draw(A) to (-1,0.9);

\draw(1,0.9) to (1,1.6);
\draw(0,0.9) to (0,1.6);
\draw(-1,0.9) to (-1,1.6);

\node[circle,draw] (B) at (3,0.3) {$\rho$};

\draw(B) to (3,-0.7);
\draw(B) to (4,1.3);
\draw(B) to (3,1.3);
\draw(B) to (2,1.3);

\node[circle,draw] (C) at (6,0.6) {$\rho$};

\draw(C) to (6,-1);
\draw(C) to (7,1.6);
\draw(C) to (6,1.6);
\draw(C) to (5,1.6);

\node(X) at (1.5,0.3) {$=$};
\node(Y) at (4.5,0.3) {$=$};

\end{tikzpicture}
\end{center}

An algebra for a plain operad $C$ on a set $X$ is a function $\theta:X_C \longrightarrow X$ satisfying unit and multiplication axioms. The elements of $X_C$ are tuples $(x_1,...,x_m, \, \rho)$ where $c(\rho) = m$ and $x_i \in X$ for all $i$. We think of  
$$ \rho \, (x_1,...,x_m) := \theta(x_1,...,x_m, \, \rho)  \in X$$ 
as the result of applying the operation $\rho$ to the ordered list $(x_1,...,x_m)$ of elements of $X$, which can be represented by decorating the tree diagram for $\rho$,
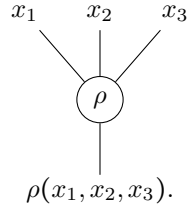
\begin{center}
\begin{tikzpicture}[node distance=5cm, auto]

\node[circle,draw] (A) at (0,0) {$\rho$};

\node (a) at (-1,1.15) {$x_1$};
\node (b) at (0,1.15) {$x_2$};
\node (c) at (1,1.15) {$x_3$};

\node (x) at (0,-1.2) {$\rho(x_1,x_2,x_3)$.};

\draw(A) to (0,-1);
\draw(A) to (c);
\draw(A) to (b);
\draw(A) to (a);

\end{tikzpicture}
\end{center}
The unit axiom says that $\text{id}$ acts as the identity operation, so for any $x \in X$ we have $ \text{id}(x) = x$.
The multiplication axiom, concisely expressed by the equation 
$$ \big( \rho \circ (\rho_1,...,\rho_m) \big) (x_1,...,x_{k_1 + ... + k_m}) = \rho \, \big(\rho_1(x_1,...,x_{k_1}),...,\rho_m(x_{k_1 + ... + k_{m-1} +1},...,x_{k_1 + ... + k_m}) \big),$$
tells us that our operations behave in a coherent way with respect to composition.
Moreover, a morphism of algebras for $C$ is a morphism $f:X \longrightarrow Y$ of the underlying sets preserving each of the operations, that is, for any $(x_1,...,x_m, \rho) \in X_C$ we have
$f ( \rho \, (x_1,...,x_m) ) = \rho \, ( f(x_1),...,f(x_m)).$

For a concrete example, let us construct the plain operad for monoids. Given an ordered list of $m$ elements in a monoid there is exactly one way to multiply them. The corresponding plain operad should therefore have underlying collection $1_{\mathbb{N}}:\mathbb{N} \longrightarrow \mathbb{N}$, since this means that it contains exactly one operation of arity $m$ for each $m \in \mathbb{N}$. This collection has a unique operad structure, and it is easy to check that an algebra for this operad on a set $X$ is precisely a monoid with underlying set $X$, and that a morphism of algebras is precisely a monoid morphism.
\end{example} 

\begin{lemma}\label{initial-terminal}
Let  $T = (T, \eta, \mu)$ be a cartesian monad on a complete category $\mathcal{S}$. 
\begin{enumerate}[i)]
\item The terminal $T$-operad is given by equipping the $T$-collection $1_{T1}:T1 \longrightarrow T1$ with its unique operad structure. The monad induced on $\mathcal{S}$ is isomorphic to $T$, so the associated category of algebras is isomorphic to the category $T \mhyphen \mathbf{Alg}$ of algebras for $T$.
\item Similarly, the initial $T$-operad is given by equipping the $T$-collection $\eta_1:1 \longrightarrow T1$ with its unique operad structure. The monad induced on $\mathcal{S}$ is isomorphic to the identity monad, so the associated category of algebras is isomorphic to $\mathcal{S}$.
\end{enumerate}
\end{lemma}

\begin{remark} As we will see in Section \ref{symmop section}, plain operads are a special case of symmetric operads, which are equivalent to (not necessarily strongly) regular finitary algebraic theories. In general, symmetric operads cannot be expressed as $T$-operads.
\end{remark}

Just as plain operads describe strongly regular finitary algebraic theories on $\mathbf{Set}$, $T$-operads over a category $\mathcal{S}$ describe certain kinds of algebraic theories on $\mathcal{S}$. Given a $T$-operad with underlying collection $c:C \longrightarrow T1$, we think of $C$ as the \textit{object of abstract operations}, with arities specified by $c$. These arities are not necessarily natural numbers, they could instead be shapes or spaces, or any number of other mathematical structures, depending on the choice of $T$. In Section \ref{GlobOp Section} we cover globular operads, where the arities are given by shapes of pasting diagrams in globular sets.

\subsection{Free $T$-operads}\label{free T-operad section}

In this section we construct a left adjoint to the canonical forgetful functor $T \mhyphen \mathbf{Op} \longrightarrow T \mhyphen \mathbf{Coll}$ for cartesian monads $T$ satisfying some extra conditions.  
The following theorem appears in greater generality in \cite{TL} as Theorem 6.5.2, but since we use some details of the proof elsewhere, we provide an outline here.

Recall that a monad is \textit{finitary} if the underlying endofunctor preserves filtered colimits.

\begin{theorem}\label{T-op adjunction}
Let $T=(T, \mu, \eta)$ be a finitary cartesian monad on a presheaf category $\mathcal{S}$. There exists a free $T$-operad functor $T \mhyphen \mathbf{Coll} \longrightarrow T \mhyphen \mathbf{Op}$ that is left adjoint to the canonical forgetful functor.
\end{theorem}

\begin{proof} Let $c:C \longrightarrow T1$ be a $T$-collection. For each $k \in \mathbb{N}$, we construct a new $T$-collection inductively by defining $c_{k+1} : C^{{(k+1)}} \longrightarrow T1$ to be the collection $\langle \eta_1, q_k \rangle : 1 + (C^{(k)} \circ C) \longrightarrow T1$, where $q_k$ is the composite of the left hand diagonals in the diagram 
\begin{center}
\begin{tikzpicture}[node distance=1.4cm, auto]

\node (A) {$C^{(k)} \circ C$};
\node (TA) [left of=A, below of=A] {$TC^{(k)}$};
\node (X) [left of=TA, below of=TA] {$T^21$};
\node (Y) [left of=X, below of=X] {$T1$};
\node (B) [right of=A, below of=A] {$C$};
\node (C) [node distance=2.8cm, below of=A] {$T1$};

\draw[->] (A) to node [swap] {} (TA);
\draw[->] (A) to node {} (B);
\draw[->] (TA) to node [swap] {$T!$} (C);
\draw[->] (B) to node {$c$} (C);
\draw[->] (TA) to node [swap] {$Tc_k$} (X);
\draw[->] (X) to node [swap] {$\mu$} (Y);

\end{tikzpicture}
\end{center}
and $c_0 : C^{(0)} \longrightarrow T1$ is the unit component $\eta_1 : 1 \longrightarrow T1$.
Next, let $i_0:C^{(0)} = 1 \longrightarrow 1 + C \cong 1 + (1 \circ C) = C^{(1)}$ be the coprojection, and for each $k > 0$ define a morphism $i_k:C^{(k)} \longrightarrow C^{(k+1)}$ by
\begin{center}
\begin{tikzpicture}[node distance=6.8cm, auto]

\node (X) {$1 + (C^{(k)} \circ C) = C^{(k+1)}.$};
\node (A) [left of=X] {$C^{(k)} = 1 + (C^{(k-1)} \circ C)$};

\draw[->] (A) to node {$1 + (i_{k-1} \circ 1)$} (X);

\end{tikzpicture}
\end{center}
Let $C^{\infty}$ be the colimit of the sequential diagram $C^{(0)} \longrightarrow C^{(1)} \longrightarrow C^{(2)} \longrightarrow \hdots$ of $i_k$s. It can be shown by induction on $k$ that the triangle
\begin{center}
\begin{tikzpicture}[node distance=1.4cm, auto]

\node (A) {$C^{(k)}$};
\node (B) [right of=A, below of=A] {$T1$};
\node (C) [right of=B, above of=B] {$C^{(k+1)}$};

\draw[->] (A) to node [swap] {$c_k$} (B);
\draw[->] (A) to node {$i_k$} (C);
\draw[->] (C) to node {$c_{k+1}$} (B);

\end{tikzpicture}
\end{center}
commutes for each $k \in \mathbb{N}$, so there is an induced $T$-collection $c_{\infty}:C^{\infty} \longrightarrow T1$. This collection can be equipped with the structure of a $T$-operad by defining the identity map $ids:1 = C^{(0)} \longrightarrow C^{\infty}$ to be the coprojection, and building the composition map $comp: C^{\infty} \circ C^{\infty} \longrightarrow C^{\infty}$ using canonical maps $comp_{(k,m)}:C^{(k)} \circ C^{(m)} \longrightarrow C^{(k+m)}$ defined by induction on $m$ for a fixed $k$. It is the construction of the composition map that requires $T$ to be finitary and $\mathcal{S}$ to be a presheaf category. Now that we know the value of the free functor on objects of $T \mhyphen \mathbf{Coll}$, it is straightforward to extend this definition to morphisms, and to show that our functor is left adjoint to the forgetful one.
\end{proof}

The free monoid monad on $\mathbf{Set}$ is finitary \cite[2.24]{A+S}, and $\mathbf{Set}$ is certainly a presheaf category, so Theorem \ref{T-op adjunction} applies to plain operads. The example below provides an intuitive description of free plain operads, which will be of great use in Sections \ref{symmop section} and \ref{slices section}.

\begin{example}\label{free plain operads} 
 Let $c:C \longrightarrow \mathbb{N}$ be a plain collection. Recall from Example \ref{plain op} that an element $\rho$ of $C$ may be represented as a tree diagram of height 1. The set $C^{(k)}$ can then be interpreted as the set of all tree diagrams of height $\leqslant k$ built from the elements of $C$, and the functions $i_k:C^{(k)} \longrightarrow C^{(k+1)}$ are the canonical inclusions. To see this, let us consider the construction of $C^{(k)}$ for low values of $k$:

\begin{list}{•}{}

\item The set $C^{(0)} := 1 = \{ \text{id} \}$ is the terminal set whose single element we refer to as the identity. We represent this element as the unique tree diagram of height 0, i.e., as a single vertical line. 

\item There is an isomorphism $C^{(1)} := 1 + (1 \circ C) \cong 1 + C$, so $C^{(1)} \cong \{ \rho \ | \ \rho = \text{id} \ \text{or} \ \rho \in C \}$. The elements of $C^{(1)}$ therefore correspond to tree diagrams of height $\leqslant 1$. 

\item The elements of $(1 + C) \circ C$ are tuples $(\rho_1,...,\rho_m, \, \rho)$ where $\rho \in C$, $c(\rho) = m$, and for each $i$ either $\rho_i = \text{id}$ or $\rho_i \in C$. We will denote these tuples using composite notation,
$$(\rho_1,...,\rho_m, \, \rho) = \rho \circ (\rho_1,...,\rho_m).$$
If $\rho_i = \text{id}$ for all $i$, then we denote the element simply by $\rho$, so
$$\rho \circ (\text{id},...,\text{id}) = \rho.$$
Note that if $\rho_i \neq \text{id}$ for at least one $i$, then the tree diagram associated to $\rho \circ (\rho_1,...,\rho_m)$ has height 2. Now since $C^{(2)}:=1+(C^{(1)} \circ C) \cong 1 + ((1 + C ) \circ C)$, we see that $C^{(2)}$ is isomorphic to the set of tree diagrams of height $\leqslant 2$.

\end{list}

The pattern above continues, so the underlying set $\text{colim}_k \, C^{(k)} = C^{\infty}$ of the free plain operad on $c:C \longrightarrow \mathbb{N}$ may be thought of as the set of all tree diagrams, of any height, built from the elements of $C$. 
Each of these diagrams represents a unique free operadic composite of the elements of $C$. 
Note that using the associativity and identity axioms for operads, we can show by induction on $k$ that every non-identity element of $C^{\infty}$ with associated tree diagram of height $k$ may be expressed in the form 
$$Q \circ (\rho_1,...,\rho_m),$$ 
where $Q$ is a free operadic composite of the elements of $C$ with an associated tree diagram of height $k-1$, $\rho_i = \text{id}$ or $\rho_i \in C$, and $\rho_i \neq \text{id}$ for at least one $i$. The underlying collection map $c_{\infty}:C^{\infty} \longrightarrow \mathbb{N}$ is given by \label{underlying collection of free plain op} 
$$c_{\infty}( Q \circ (\rho_1,...,\rho_m)) = c(\rho_1) + ... + c(\rho_m),$$
where we define $c(\text{id}) = 1$. 
\end{example}

We end this section with a specific example of a free plain operad.

\begin{example}\label{free op for magmas} Recall that a \textit{magma} is a set equipped with a binary operation satisfying no axioms. The plain operad for magmas is the free plain operad on the collection $j : J \longrightarrow \mathbb{N}$, where $J = \{ b \}$ is a single element set and $j(b) = 2$. The elements of $J^{\infty}$ correspond to tree diagrams built from $b$; the identity element $\text{id}$ corresponds to the tree diagram of height 0, the unique tree diagram
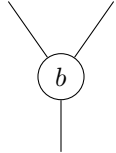
\vspace{1mm}
\begin{center}
\begin{tikzpicture}[node distance=5cm, auto]

\node[circle,draw] (A) at (0,0) {$b$};

\draw(A) to (0,-1);
\draw(A) to (0.7,1);
\draw(A) to (-0.7,1);

\end{tikzpicture}
\end{center}
of height 1 corresponds to $b$, the tree diagrams
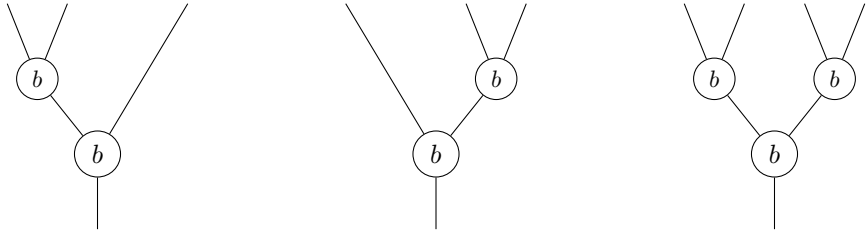
\vspace{1mm}
\begin{center}
\begin{tikzpicture}[node distance=5cm, auto]

\node[circle,draw] (A) at (0,0) {$b$};
\node[circle,draw, scale=0.9] (X) at (-0.8,1) {$b$};

\draw(A) to (0,-1);
\draw(A) to (X);
\draw(A) to (1.2,2);

\draw(X) to (-1.2,2);
\draw(X) to (-0.4,2);

\node[circle,draw] (B) at (4.5,0) {$b$};
\node[circle,draw, scale=0.9] (V) at (5.3,1) {$b$};

\draw(B) to (4.5,-1);
\draw(B) to (3.3,2);
\draw(B) to (V);

\draw(V) to (4.9,2);
\draw(V) to (5.7,2);

\node[circle,draw] (C) at (9,0) {$b$};
\node[circle,draw, scale=0.9] (U) at (8.2,1) {$b$};
\node[circle,draw, scale=0.9] (V) at (9.8,1) {$b$};

\draw(C) to (9,-1);
\draw(C) to (U);
\draw(C) to (V);

\draw(U) to (7.8,2);
\draw(U) to (8.6,2);

\draw(V) to (9.4,2);
\draw(V) to (10.2,2);

\end{tikzpicture}
\end{center}
of height 2 correspond to the free composites $b \circ (b, \text{id})$, $b \circ (\text{id}, b)$ and $b \circ (b, b)$, and so on. 
An algebra $\theta:X_{J^{\infty}} \longrightarrow X$ for $J^{\infty}$ is precisely a magma with underlying set $X$. The single element $b \in J$ provides a binary operation on the elements of $X$; given $x_1, x_2 \in X$ we write $\theta(x_1, x_2, b) = b(x_1,x_2) = x_1 \ast x_2$. The elements of $J^{\infty}$ obtained by operadic composition of multiple copies of $b$ provide operations derived from this basic one, for instance, 
$$\big( b \circ (b, \text{id}) \big) (x_1, x_2, x_3) = b \circ \big( b(x_1,x_2), \text{id}(x_3) \big) = (x_1 \ast x_2) \ast x_3.$$
\end{example}

\section{Globular operads}\label{GlobOp Section}

We now turn our attention to globular operads, whose algebras bear a strong formal resemblance to algebraic higher categories. The material in this section is covered in greater detail in \cite{Me}.

\subsection{Globular operads and their algebras}\label{GlobOpSubSection} Globular operads are a kind of $T$-operad arising from a monad on the category of globular sets. Just as the underlying data of a monoid is a set, the underlying graph data of a higher category is a globular set.

\begin{definition}\label{GSet is presheaf} The \textit{category $\mathbf{GSet}$ of globular sets} is the presheaf category $[\mathbb{G}^{\textrm{op}}, \, \mathbf{Set}]$, where $\mathbb{G}$ is generated by the graph
\begin{center}
\begin{tikzpicture}[node distance=2cm, auto]

\node (A) {$0$};
\node (B) [left of=A] {$1$};
\node (C) [left of=B] {$2$};
\node (D) [left of=C] {$...$};

\draw[transform canvas={yshift=0.5ex},->] (A) to node [swap] {$s$} (B);
\draw[transform canvas={yshift=-0.5ex},->] (A) to node  {$t$} (B);

\draw[transform canvas={yshift=0.5ex},->] (B) to node [swap] {$s$} (C);
\draw[transform canvas={yshift=-0.5ex},->] (B) to node {$t$} (C);

\draw[transform canvas={yshift=0.5ex},->] (C) to node [swap] {$s$} (D);
\draw[transform canvas={yshift=-0.5ex},->] (C) to node {$t$} (D);

\end{tikzpicture}
\end{center}
subject to the equations $ss=ts$ and $st=tt$. Given a globular set $\bs{G}:\mathbb{G}^{\textrm{op}} \longrightarrow \mathbf{Set}$, we refer to the elements of $G_n := \bs{G}(n)$ as the $n$-cells of $\bs{G}$, and to the functions $s, t:G_n \longrightarrow G_{n-1}$ as the source and target maps, respectively. 
\end{definition}

\begin{definition} 
Let $U:\mathbf{Str} \, \omega \mhyphen \mathbf{Cat} \longrightarrow \mathbf{GSet}$ be the canonical forgetful functor from the category of strict $\omega$-categories.
\begin{enumerate}
\item The \textit{free strict $\omega$-category functor} $(-)^*:\mathbf{GSet} \longrightarrow \mathbf{Str} \, \omega \mhyphen \mathbf{Cat}$ is the left adjoint to $U$.  
Given a globular set $\bs{G}$, the free strict $\omega$-category $\bs{G}^*$ is the one whose $n$-cells are $n$-pasting diagrams in $\bs{G}$ and whose composition is concatenation along matching boundary cells.
\item The \textit{free strict $\omega$-category monad} on $\mathbf{GSet}$, which by abuse of notation we also denote by $(-)^* = ((-)^*,\eta, \mu )$, is the monad arising from the adjunction $(-)^* \dashv U : \ \mathbf{Str} \, \omega \mhyphen \mathbf{Cat} \longrightarrow \mathbf{GSet}$. 
\end{enumerate}
\end{definition}

\begin{definition} \label{degenerate pd}
A \textit{degenerate} $n$-pasting diagram in a globular set $\bs{G}$ is an $n$-pasting diagram which does not contain any $n$-cells of $\bs{G}$.
\end{definition}

\begin{remark}
The degenerate $n$-pasting diagrams in $\bs{G}$ are the identity $n$-cells of the $\omega$-category $\bs{G}^*$.
\end{remark}

It is shown in \cite[F.2.2]{TL} that the forgetful functor $U:\mathbf{Str} \, \omega \mhyphen \mathbf{Cat} \longrightarrow \mathbf{GSet}$ is monadic, so the categories $\mathbf{Str} \, \omega \mhyphen \mathbf{Cat}$ and $(-)^* \mhyphen \mathbf{Alg}$ are equivalent. In fact, a $(-)^{\ast}$-algebra structure on a globular set $\bs{G}$ is precisely a strict $\omega$-category with underlying globular set $\bs{G}$, and a morphism of $(-)^{\ast}$-algebras is precisely a strict $\omega$-functor. It is also shown in \cite[F.2.2]{TL} that the monad $(-)^*$ is cartesian, and since $\mathbf{GSet}$ is a presheaf category, it is certainly complete, allowing for the following definition.

\begin{definition}
The \textit{category $\mathbf{GOp}$ of globular operads} is the category of $T$-operads for which $T$ is the free strict $\omega$-category monad on $\mathbf{GSet}$.
\end{definition}

To see how theories of higher category may be understood as globular operads, let us first unpack the definition above.

\begin{notation}\label{TerminalGlobSet} We denote by $\mathbf{1}$ the terminal globular set given by the following diagram in $\mathbf{Set}$.
\vspace{1mm}
\begin{center}
\begin{tikzpicture}[node distance=2cm, auto]

\node (A) {$\{ \filledstar \}$};
\node (B) [left of=A] {$\{ \filledstar \}$};
\node (C) [left of=B] {$\{ \filledstar \}$};
\node (D) [left of=C] {$...$};

\draw[transform canvas={yshift=0.5ex},->] (B) to node {} (A);
\draw[transform canvas={yshift=-0.5ex},->] (B) to node [swap] {} (A);

\draw[transform canvas={yshift=0.5ex},->] (C) to node {} (B);
\draw[transform canvas={yshift=-0.5ex},->] (C) to node [swap] {} (B);

\draw[transform canvas={yshift=0.5ex},->] (D) to node {} (C);
\draw[transform canvas={yshift=-0.5ex},->] (D) to node [swap] {} (C);

\end{tikzpicture}
\end{center}
\end{notation}

The $n$-cells of $\bs{1}^*$ are $n$-pasting diagrams in $\bs{1}$; observe that there is no need to label the individual cells within these pasting diagrams since each $k$-cell represents the single $k$-cell of $\bs{1}$. A typical 2-cell $\pi$ of $\bs{1}^*$ is illustrated below.
\begin{center}
\begin{tikzpicture}[node distance=1.5cm, auto]

\node (A) {$\cdot$};
\node (B) [node distance=2.5cm, right of=A] {$\cdot$};
\node (C) [node distance=2.5cm, right of=B] {$\cdot$};
\node (D) [node distance=2.5cm, right of=C] {$\cdot$};
\node (D') [node distance=0.7cm, left of=A] {$\pi \ =$};

\draw[->, bend left=85] (A) to node {} (B);
\draw[->, bend left=25] (A) to node {} (B);
\draw[->, bend right=25] (A) to node [swap] {} (B);
\draw[->, bend right=85] (A) to node [swap] {} (B);

\node (XB) [node distance=1.25cm, right of=A] {$\Downarrow $};
\node (XA) [node distance=0.6cm, above of=XB] {$\Downarrow $};
\node (XC) [node distance=0.6cm, below of=XB] {$\Downarrow $};

\draw[->] (B) to node {} (C);

\draw[->, bend left=60] (C) to node {} (D);
\draw[->] (C) to node {} (D);
\draw[->, bend right=60] (C) to node [swap] {} (D);

\node (W) [node distance=1.25cm, right of=C] {};
\node (YA) [node distance=0.38cm, above of=W] {$\Downarrow $};
\node (YC) [node distance=0.38cm, below of=W] {$\Downarrow $};

\end{tikzpicture}
\end{center}

\begin{notation}\label{boundary notation}
Since the source and target $(n{-}1)$-cells of each $n$-cell of $\bs{1}^*$ are equal, we write ${\partial}$ rather than $s$ or $t$ when we want to refer to the source or target of a cell. 
\end{notation}

\begin{definition} 
The \textit{category $\mathbf{GColl}$ of globular collections} is the slice category $\mathbf{GSet} / \bs{1}^*$. 
\end{definition}

A globular operad $\bs{G} = (\bs{G}, g, ids, comp)$ is a globular collection $g:\bs{G} \longrightarrow \bs{1}^*$ together with morphisms $ids:\bs{1} \longrightarrow \bs{G}$ and $comp:\bs{G} \circ \bs{G} \longrightarrow \bs{G}$ of collections, satisfying identity and associativity axioms. We think of each $n$-cell $\Lambda$ of $\bs{G}$ as an abstract operation composing $n$-pasting diagrams of shape $g(\Lambda)$ into single $n$-cells. For example, if $\chi$ is a 2-cell of $\bs{G}$ and $g(\chi) = \tau$,
\begin{center}\label{tau}
\begin{tikzpicture}[node distance=2.5cm, auto]

\node (B) {$\cdot$};
\node (C) [right of=B] {$\cdot$};
\node (D) [node distance=2.3cm, right of=C] {$\cdot$};
\node (D') [node distance=4mm, right of=D] {$= \tau$}; 

\draw[->] (C) to node {} (D);

\draw[->, bend left=60] (B) to node {} (C);
\draw[->] (B) to node {} (C);
\draw[->, bend right=60] (B) to node [swap] {} (C);

\node (W) [node distance=1.25cm, right of=B] {};
\node (YA) [node distance=0.38cm, above of=W] {$\Downarrow $};
\node (YC) [node distance=0.38cm, below of=W] {$\Downarrow $};

\node (X) [node distance=3.5cm, left of=B] {$Y$};
\node (Y) [left of=X] {$X$};
\draw[->, bend left=40] (Y) to node {$x$} (X);
\draw[->, bend right=40] (Y) to node [swap] {$x'$} (X);
\node (Z) [node distance=1.25cm, right of=Y] {$\Downarrow \chi$};

\node (P) [node distance=0.5cm, right of=X] {};
\node (Q) [node distance=0.5cm, left of=B] {};
\draw[|->, dashed] (P) to node {$g$} (Q);

\end{tikzpicture}
\end{center}
then we think of $\chi$ as an abstract operation composing 2-pasting diagrams of shape $\tau$ into single 2-cells. The map $ids:\bs{1} \longrightarrow \bs{G}$ picks out an $n$-cell of $\bs{G}$ for each $n \in \mathbb{N}$. We denote this cell by $\text{id}_n$ and refer to it as the identity $n$-cell. The image of $\text{id}_n$ under $g$ is the $n$-pasting diagram in $\bs{1}$ made up of a single $n$-cell.
\begin{center}
\begin{tikzpicture}[node distance=2.5cm, auto]

\node (C) {$\cdot$};
\node (D) [node distance=2.5cm, right of=C] {$\cdot$};

\draw[->, bend left=45] (C) to node {} (D);
\draw[->, bend right=45] (C) to node [swap] {} (D);

\node (W) [node distance=1.25cm, right of=C] {$\Downarrow$};

\node (X) [node distance=3.5cm, left of=C] {$\text{id}_0$};
\node (Y) [left of=X] {$\text{id}_0$};
\draw[->, bend left=40] (Y) to node {$\text{id}_1$} (X);
\draw[->, bend right=40] (Y) to node [swap] {$\text{id}_1$} (X);
\node (Z) [node distance=1.25cm, right of=Y] {$\Downarrow \text{id}_2$};

\node (P) [node distance=0.5cm, right of=X] {};
\node (Q) [node distance=0.5cm, left of=C] {};
\draw[|->, dashed] (P) to node {$g$} (Q);

\end{tikzpicture}
\end{center}

The $n$-cells of $\bs{G} \circ \bs{G}$ are pairs consisting of an $n$-cell $\Lambda$ of $\bs{G}$ together with an $n$-pasting diagram in $\bs{G}$ of shape $g(\Lambda)$. A typical 2-cell $((\nu, \nu', v), \chi )$ of $\bs{G} \circ \bs{G}$ is
\vspace{-4.5mm}
\begin{center}
\[ \left(
\begin{tikzpicture}[node distance=2cm, auto, baseline=-1.5mm]

\node (A) {$U$};
\node (B) [node distance=2.5cm, right of=A] {$V$};
\node (C) [node distance=2.5cm, right of=B] {$W$};
\node (p) [node distance=0.44cm, right of=C] {};
\node () [node distance=0.1cm, below of=p] {$,$};
\node (C') [node distance=1cm, right of=C] {$X$};
\node (D) [node distance=2.5cm, right of=C'] {$Y$};

\draw[->] (B) to node {$v$} (C);

\draw[->, bend left=60] (A) to node {$u$} (B);
\draw[->] (A) to node {} (B);
\draw[->, bend right=60] (A) to node [swap] {$u''$} (B);

\node (W) [node distance=1.25cm, right of=A] {};
\node (YA) [node distance=0.4cm, above of=W] {$\Downarrow  \nu $};
\node (YC) [node distance=0.4cm, below of=W] {$\Downarrow  \nu' $};

\draw[->, bend left=40] (C') to node {$x$} (D);
\draw[->, bend right=40] (C') to node [swap] {$x'$} (D);

\node (XB) [node distance=1.3cm, right of=C'] {$\Downarrow  \chi $};

\end{tikzpicture}
\right) \]
\end{center}
where the left hand side is a 2-pasting diagram in $\bs{G}$ and $\chi$ is the 2-cell above  satisfying $g(\chi)=\tau$. We use this example to describe the composition map. Let $g(\nu) = \tau_1$, $g(\nu') = \tau_2$, and $g(v) = \tau_3$,
\begin{center}
\resizebox{5.9in}{!}{
\begin{tikzpicture}[node distance=2.5cm, auto]

\node (A) {$U$};
\node (B) [node distance=2.5cm, right of=A] {$V$};
\node (C) [node distance=2.5cm, right of=B] {$W$};
\draw[->] (B) to node {$v$} (C);
\draw[->, bend left=60] (A) to node {$u$} (B);
\draw[->] (A) to node {} (B);
\draw[->, bend right=60] (A) to node [swap] {$u''$} (B);
\node (X) [node distance=1.25cm, right of=A] {};
\node (Y) [node distance=0.4cm, above of=X] {$\Downarrow \nu $};
\node (Z) [node distance=0.4cm, below of=X] {$\Downarrow \nu' $};

\node (D) [node distance=2cm, below of=B] {$\cdot$};
\node (E) [node distance=2.3cm, right of=D] {$\cdot$};
\node (F) [node distance=2.3cm, right of=E] {$\cdot$};
\node (G) [node distance=2.3cm, right of=F] {$\cdot$};
\node (G') [node distance=0.5cm, right of=G] {$= \tau_3$};
\draw[->] (F) to node {} (G);
\draw[->] (E) to node {} (F);
\draw[->] (D) to node {} (E);
\node (a) [node distance=1.25cm, right of=B] {};
\node (a') [node distance=0.2cm, below of=a] {};
\node (b) [node distance=1.2cm, left of=F] {};
\node (b') [node distance=0.15cm, above of=b] {};
\draw[|->, dashed] (a') to node {$g$} (b');

\node (H) [node distance=1.5cm, left of=A, above of=A] {$\cdot$};
\node (I) [node distance=2.4cm, left of=H] {$\cdot$};
\node (I') [node distance=0.5cm, left of=I] {$\tau_1 = $};
\draw[->, bend left=40] (I) to node {} (H);
\draw[->, bend right=40] (I) to node {} (H);
\node (X') [node distance=1.2cm, right of=I] {$\Downarrow$};;
\node (y) [node distance=0.5cm,above of=A] {};
\draw[|->, dashed] (y) to node [swap] {$g$} (H);

\node (K) [node distance=1.5cm, left of=A, below of=A] {$\cdot$};
\node (L) [node distance=2.4cm, left of=K] {$\cdot$};
\node (L') [node distance=0.5cm, left of=L] {$\tau_2 =$};
\draw[->, bend left=60] (L) to node {} (K);
\draw[->] (L) to node {} (K);
\draw[->, bend right=60] (L) to node {} (K);
\node (T) [node distance=1.2cm, right of=L] {};
\node (X'') [node distance=0.3cm, above of=T] {$\Downarrow$};
\node (Y'') [node distance=0.35cm, below of=T] {$\Downarrow$};
\node (z) [node distance=0.5cm, below of=A] {};
\draw[|->, dashed] (z) to node {$g$} (K);

\end{tikzpicture}\label{tau_i}
}
\end{center}
and denote by $ \tau \circ (\tau_1, \tau_2, \tau_3)$ the 2-cell of $\bs{1}^*$ obtained by replacing the three cells of $\bs{1}$ making up $\tau$ with $\tau_1, \tau_2$ and $\tau_3$, respectively. Then $comp \, ((\nu, \nu', v), \chi) = \chi \circ (\nu, \nu', v)$
is a 2-cell of $\bs{G}$ satisfying $g(\chi \circ (\nu, \nu', v)) = \tau \circ (\tau_1, \tau_2, \tau_3)$.  
\begin{center}
\resizebox{1\textwidth}{!}{
\begin{tikzpicture}[node distance=1.9cm, auto]

\node (B) {$\cdot$};
\node (C) [node distance=2.5cm, right of=B] {$\cdot$};
\node (D) [right of=C] {$\cdot$};
\node (E) [right of=D] {$\cdot$};
\node (F) [right of=E] {$\cdot$};

\draw[->, bend left=25] (B) to node {} (C);
\draw[->, bend right=25] (B) to node {} (C);
\draw[->, bend left=80] (B) to node {} (C);
\draw[->, bend right=80] (B) to node {} (C);
\draw[->] (C) to node {} (D);
\draw[->] (D) to node {} (E);
\draw[->] (E) to node {} (F);

\node (X) [node distance=1.25cm, right of=B] {$\Downarrow$};
\node (Y) [node distance=0.6cm, above of=X] {$\Downarrow$};
\node (Z) [node distance=0.6cm, below of=X] {$\Downarrow$};

\node (R) [node distance=3.5cm, left of=B] {$Y \circ (W)$};
\node (Q) [node distance=4cm, left of=R] {$X \circ (U)$};
\draw[->, bend left=40] (Q) to node {$x \circ (u,v)$} (R);
\draw[->, bend right=40] (Q) to node [swap] {$x' \circ (u'', v)$} (R);
\node [node distance=2cm, right of=Q] {$\Downarrow \chi \circ (\nu, \nu', v)$};

\node (r) [node distance=1cm, right of=R] {};
\node (a) [node distance=0.6cm, left of=B] {};
\draw[|->, dashed] (r) to node {$g$} (a);

\end{tikzpicture}
}
\end{center}
More generally, the composition map sends an $n$-cell $((\Lambda_1,...,\Lambda_m), \Lambda )$ of $\bs{G} \circ \bs{G}$ to an $n$-cell $\Lambda \circ (\Lambda_1,...,\Lambda_m)$ of $\bs{G}$ for which 
$$g( \Lambda \circ (\Lambda_1,...,\Lambda_m)) = g(\Lambda) \circ (g(\Lambda_1),...,g(\Lambda_m))$$ 
is the $n$-pasting diagram in $\bs{1}$ obtained by 
replacing the individual cells making up $g(\Lambda)$ with the $g(\Lambda_i)$s.
We think of $\Lambda \circ (\Lambda_1,...,\Lambda_m)$ as the abstract operation composing $n$-pasting diagrams of shape $g( \Lambda \circ (\Lambda_1,...,\Lambda_m) )$ given by first using the $\Lambda_i$s to compose smaller components and then applying $\Lambda$ to the result. 

Consider now an algebra for a globular operad $\bs{G}$, that is, an algebra for the monad $(-)_{\bs{G}}$ on $\mathbf{GSet}$; see Definition \ref{operad-monad}. Given a globular set $\bs{A}$, the $n$-cells of $\bs{A_G}$ are pairs consisting of an $n$-cell $\Lambda$ of $\bs{G}$ together with an $n$-pasting diagram in $\bs{A}$ of shape $g(\Lambda)$. 
A typical 2-cell $( (\alpha, \alpha', b), \chi)$ of $\bs{A_G}$ is
\vspace{-4mm}
\begin{center}
\[ \left( 
\begin{tikzpicture}[node distance=2cm, auto, baseline=-1.5mm]

\node (A) {$A$};
\node (B) [node distance=2.5cm, right of=A] {$B$};
\node (C) [node distance=2.4cm, right of=B] {$C$};
\node (P) [node distance=0.45cm, right of=C] {};
\node () [node distance=0.1cm, below of=P] {\large ,};
\node (C') [node distance=1cm, right of=C] {$X$};
\node (D) [node distance=2.5cm, right of=C'] {$Y$};

\draw[->, bend left=60] (A) to node {$a$} (B);
\draw[->] (A) to node {} (B);
\draw[->, bend right=60] (A) to node [swap] {$a''$} (B);

\draw[->] (B) to node {$b$} (C);

\node (W) [node distance=1.25cm, right of=A] {};
\node (YA) [node distance=0.4cm, above of=W] {$\Downarrow  \alpha$};
\node (YC) [node distance=0.4cm, below of=W] {$\Downarrow  \alpha'$};

\draw[->, bend left=40] (C') to node {$x$} (D);
\draw[->, bend right=40] (C') to node [swap] {$x'$} (D);

\node (XB) [node distance=1.3cm, right of=C'] {$\Downarrow  \chi$};

\end{tikzpicture}
\right) \]
\end{center}
where the left hand side is a 2-pasting diagram in $\bs{A}$ and $\chi$ is the 2-cell of $\bs{G}$ described earlier satisfying $g(\chi) = \tau$. 
An algebra for $\bs{G}$ on $\bs{A}$ is then a morphism $\theta:\bs{A_G} \longrightarrow \bs{A}$ of globular sets satisfying unit and multiplication axioms. Given an $n$-cell $( (\alpha_1,...,\alpha_m),  \, \Lambda )$ of $\bs{A_G}$ we write
$$\Lambda(\alpha_1,...,\alpha_m) := \theta ( (\alpha_1,...,\alpha_m), \, \Lambda ) \in A_n,$$
and think of $\Lambda(\alpha_1,...,\alpha_m)$ as a composition of the $n$-pasting diagram $(\alpha_1,...,\alpha_m)$ in $\bs{A}$ into a single $n$-cell of $\bs{A}$. So $\theta$ sends the 2-cell $((\alpha, \alpha', b), \chi)$ of $\bs{A_G}$ above to a 2-cell in $\bs{A}$ of the form
\begin{center}
\begin{tikzpicture}[node distance=2cm, auto]
\node (R) {$X(A)$};
\node (Q) [node distance=3.7cm, right of=R] {$Y(C)$.};
\draw[->, bend left=37] (R) to node {$x(a,b)$} (Q);
\draw[->, bend right=37] (R) to node [swap] {$x'(a'',b)$} (Q);
\node (U) [node distance=1.85cm, right of=R] {$\Downarrow \chi(\alpha, \alpha',b)$};
\end{tikzpicture}
\end{center}
The unit axiom tell us that for any $n$-cell $\alpha$ of $\bs{A}$, we have $\text{id}_n(\alpha) = \alpha$, so the identities provide the `do nothing' way of composing single cells. Meanwhile, the multiplication axiom, expressed by the equality below, says that composition of pasting diagrams in $\bs{A}$ behaves in a coherent way with respect to operadic composition in $\bs{G}$,
$$\Lambda \, \big( \Lambda_1 (\alpha_{11},...,\alpha_{1k_1}),..., \, \Lambda_m(\alpha_{m1},...,\alpha_{mk_m}) \big) = \big( \Lambda \circ (\Lambda_1,...,\Lambda_m) \big) (\alpha_{11},...,\alpha_{1k_1},...,\alpha_{m1},...,\alpha_{mk_m}).$$
Moreover, a morphism $F:(\bs{A},\theta) \longrightarrow (\bs{B},\sigma)$ of algebras for $\bs{G}$ is a morphism $F:\bs{A} \longrightarrow \bs{B}$ of the underlying globular sets strictly preserving the composition of pasting diagrams defined by the algebra structures.

\begin{example}\label{GlobOp T for strict w cats}
By Lemma \ref{initial-terminal}, the terminal globular operad $\bs{T}$ is given by equipping the collection $1_{\bs{1}^*}:\bs{1}^* \longrightarrow \bs{1}^*$ with its unique operad structure, and the monad $(-)_{\bs{T}}$ induced on $\mathbf{GSet}$ is isomorphic to the free strict $\omega$-category monad $(-)^*$. The category $\bs{T} \mhyphen \bf{Alg}$ of algebras for $\bs{T}$ is then equivalent to the category $\bf{Str} \, \omega \mhyphen \bf{Cat}$ of strict $\omega$-categories. 
We refer to $\bs{T}$ as \textit{the globular operad for strict $\omega$-categories}.
\end{example}

Observe that both an $\omega$-category structure and a $\bs{G}$-algebra structure on a globular set $\bs{A}$ are essentially ways of composing pasting diagrams in $\bs{A}$ a coherent way. However, not every globular operad gives rise to something that we could call a theory of $\omega$-category. Take for example the initial globular operad given by equipping the globular collection $\eta_{\bs{1}}:\bs{1} \longrightarrow \bs{1}^*$ with its unique operad structure. By Lemma \ref{initial-terminal}, the category of algebras for this operad is isomorphic to $\mathbf{GSet}$ rather than to some category of $\omega$-categories and the strict functors between them. However, asking for our globular operads to be contractible (defined below) ensures that their algebras do satisfy the required conditions on composition and coherence for $\omega$-categories. We refer the reader to \cite[\S 5]{Me} for a full explanation of these conditions and their relationship to contractibility, where it is also determined precisely when a globular operad is equivalent to some theory of higher category.

\begin{definition}
Let $\bs{G}$ be a globular set. For $n \geqslant 1$ we say that a pair $(x,x')$ of $n$-cells of $\bs{G}$ are \textit{parallel} if they share the same source and target, so $s(x)=s(x')$ and $t(x)=t(x')$. All pairs of 0-cells are considered parallel.
\end{definition}

\begin{definition}\cite[\S 9.1]{TL} A \textit{contraction function} on a globular collection $g:\bs{G} \longrightarrow \bs{1}^*$ is a function assigning to each triple $(x, x', \tau)$, where $(x,x')$ is a parallel pair of $n$-cells of $\bs{G}$ satisfying $g(x) = g(x') = {\partial}\tau$ (see Notation \ref{boundary notation}), an $(n{+}1)$-cell $\chi:x \longrightarrow x'$ of $\bs{G}$ satisfying $g(\chi)=\tau$. 
We refer to the cells in the image of a contraction function as \textit{contraction cells}.
A globular operad $\bs{G}$ is \textit{contractible} if there exists a contraction function on its underlying globular collection.
\end{definition}

\begin{example} There is a unique contraction on the globular operad $\bs{T}$ for strict $\omega$-categories; see Example \ref{GlobOp T for strict w cats}. When equipped with this contraction, every $n$-cell of $\bs{T}$ for $n>0$ is a contraction cell.
\end{example}

We end this section by restricting everything here to $n$-dimensions, allowing us to apply the same framework to $n$-categories.

\begin{definition}\label{GSet_n}
The \textit{category $\mathbf{GSet}_{\bs{n}}$ of $n$-globular sets} is the presheaf category $[\mathbb{G}_n^{\textrm{op}}, \, \mathbf{Set}]$, where $\mathbb{G}_n$ is generated by the graph
\begin{center}
\begin{tikzpicture}[node distance=2cm, auto]

\node (A) {$0$};
\node (B) [left of=A] {$1$};
\node (C) [left of=B] {$...$};
\node (D) [left of=C] {$n$};

\draw[transform canvas={yshift=0.5ex},->] (A) to node [swap] {$s$} (B);
\draw[transform canvas={yshift=-0.5ex},->] (A) to node {$t$} (B);

\draw[transform canvas={yshift=0.5ex},->] (B) to node [swap] {$s$} (C);
\draw[transform canvas={yshift=-0.5ex},->] (B) to node {$t$} (C);

\draw[transform canvas={yshift=0.5ex},->] (C) to node [swap] {$s$} (D);
\draw[transform canvas={yshift=-0.5ex},->] (C) to node {$t$} (D);

\end{tikzpicture}
\end{center}
subject to the equations $ss=ts$ and $st=tt$.
\end{definition}

The free strict $n$-category monad on $\mathbf{GSet}_{\bs{n}}$, which by abuse of notation we also denote by $(-)^*$, is defined by truncating the free strict $\omega$-category monad on $\mathbf{GSet}$ to $n$ dimensions.

\begin{definition} The \textit{category $\mathbf{GOp_{\bs{n}}}$ of $n$-globular operads} is the category of $T$-operads for which $T$ is the free strict $n$-category monad on $\mathbf{GSet}_{\bs{n}}$.
\end{definition}

\begin{example}\label{GlobOp T for strict n cats}
It follows from Lemma \ref{initial-terminal} that the monad induced on $\mathbf{GSet}_{\bs{n}}$ by the terminal $n$-globular operad $\bs{T_n}$ is isomorphic to the free strict $n$-category monad, so the category $\bs{T_n} \mhyphen \bf{Alg}$ of algebras for $\bs{T_n}$ is equivalent to the category $\mathbf{Str} \, n \mhyphen \mathbf{Cat}$ of strict $n$-categories. 
We refer to $\bs{T_n}$ as \textit{the globular operad for strict $n$-categories}.
\end{example}

\begin{definition} The category $\mathbf{GColl}_{\bs{n}}$ of \textit{$n$-globular collections} is the slice category $\mathbf{GSet}_{\bs{n}} / \bs{1}^*$, where $\bs{1}^*$ denotes the terminal $n$-globular operad.
\end{definition}

A contraction function on an $n$-globular collection is defined analogously to a contraction function on a globular collection. However, there is an extra requirement that an $n$-globular operad must meet in order to be contractible. This is to account for the fact that the coherence condition for $n$-categories is stronger than the coherence condition for $\omega$-categories; in $n$-categories the $n$-cells require special attention. Specifically, it is thought that for any $n$-pasting diagram in an $n$-category together with a way to compose its $(n-1)$-dimensional boundary, the axioms should ensure that there is a \textit{unique} choice of composite $n$-cell that is consistent with the composition of the boundary.

\begin{definition}\label{contractible n-glob ops}
An $n$-globular operad $\bs{G_n}$ is \textit{contractible} if
\begin{enumerate}[i)]
\item there exists a contraction function on the underlying collection $g:\bs{G_n} \longrightarrow \bs{1}^*$, and
\item any parallel pair $(\chi,\chi')$ of $n$-cells of $\bs{G_n}$ satisfying $g(\chi) = g(\chi')$ are equal.
\end{enumerate}
\end{definition} 

\begin{example} There is a unique contraction on the $n$-globular operad $\bs{T_n}$ for strict $n$-categories. When equipped with this contraction, every $k$-cell of $\bs{T_n}$ is a contraction cell.
\end{example}

\begin{lemma}\label{contractible n-cells} A contractible $n$-globular operad $\bs{G_n}$ is completely determined by its $k$-cells for all $k<n$; for each triple $(x, x', \tau)$ where $(x,x')$ is a parallel pair of $(n{-}1)$-cells of $\bs{G_n}$ satisfying $g(x) = g(x') = {\partial}\tau$, there exists a \textit{unique} $n$-cell $\chi:x \longrightarrow x'$ of $\bs{G_n}$ satisfying $g(\chi) = \tau$. Similarly, given any $n$-cell $((\Lambda_1,...,\Lambda_m), \Lambda )$ of $\bs{G_n} \circ \bs{G_n}$ there is a unique choice of composite $n$-cell $\Lambda \circ (\Lambda_1,...,\Lambda_m)$.
\end{lemma}

\subsection{Presentations for globular operads}\label{presSubsec}

Algebraic structures like groups are often defined via presentations, meaning that we specify a set of generators such that every element of our group can be expressed as a free product of the generators and their inverses, together with a set of equational relations among these generators. In this section we demonstrate how to construct a globular operad by building a presentation for it. These presentations are more complex than those for structures whose underlying data is just a single set. To build presentations for globular operads, we need to specify generators and relations in each dimension, beginning in dimension 0 and building the presentation recursively. We refer the reader to \cite[\S 6]{Me} for the formal definition of presentations for globular operads, but as it is rather technical and lengthy, and the idea is clear from examples, we will not provide the definition here.

\begin{notation}\label{0-cells in higher cats}
Since we are only interested in globular operads equivalent to some theory of higher category, and there are no non-trivial operations on 0-cells in higher categories, our globular operads will always have a single 0-cell - the identity $\text{id}_0$. This means that the presentations will contain no 0-cell generators, and therefore no 0-cell relations. Since there is no ambiguity, we will leave the single identity 0-cell of these operads unlabeled. For example, a 1-cell $x$ will be represented by 
\begin{center}
\begin{tikzpicture}[node distance=2.4cm, auto]

\node (A) {$\cdot$};
\node (B) [right of =A] {$\cdot$};

\draw[->] (A) to node {$x$} (B);

\node (x) [right of=B] {rather than};

\node (A') [right of=x] {$\text{id}_0$};
\node (B') [right of =A'] {$\text{id}_0$.};

\draw[->] (A') to node {$x$} (B');

\end{tikzpicture}
\end{center}
\end{notation}

\begin{remark}
In each of our examples we exploit the fact that if an $n$-globular operad $\bs{G_n}$ is contractible then there is no need to specify any $n$-cell generators or relations in a presentation for $\bs{G_n}$, since by Lemma \ref{contractible n-cells}, a contractible $n$-globular operad is completely determined by the $k$-cells for all $k<n$.
\end{remark}

Our first example is a presentation for the 3-globular operad for strict 3-categories.

\begin{example}\label{pres for T_3}
The globular operad $\bs{T_3}$ for strict 3-categories is the \textit{contractible} 3-globular operad with

\begin{list}{$\bullet$}{}

\item a single 0-cell, the identity $\text{id}_0$; 

\item 1-cells consisting of the operadic composites of 1-cells $i_1$ and $h_1$ whose images under the underlying collection map are as follows

\begin{center}
\begin{tikzpicture}[node distance=2cm, auto]

\node (A) {$\cdot$};
\node (B) [right of =A] {$\cdot$};

\draw[->] (A) to node {$i_1$} (B);

\node (P) [node distance=3cm, right of =B] {$\cdot$};

\node (b) [node distance=0.5cm, right of=B] {};
\node (p) [node distance=0.5cm, left of=P] {};
\draw[|->, dashed] (b) to node {} (p);

\node (A') [node distance=8mm, below of=A] {$\cdot$};
\node (B') [right of =A'] {$\cdot$};

\draw[->] (A') to node {$h_1$} (B');

\node (P') [node distance=3cm, right of =B'] {$\cdot$};
\node (Q') [right of =P'] {$\cdot$};
\node (R') [right of =Q'] {$\cdot$};

\draw[->] (P') to node {} (Q');
\draw[->] (Q') to node {} (R');

\node (b') [node distance=0.5cm, right of=B'] {};
\node (p') [node distance=0.5cm, left of=P'] {};
\draw[|->, dashed] (b') to node {} (p');

\end{tikzpicture}
\end{center}
subject to the equalities below,
\begin{enumerate}[i)]
\item $h_1 \circ (\text{id}_1, \, i_1) = \text{id}_1$
\item $h_1 \circ (i_1, \, \text{id}_1) = \text{id}_1$
\item $h_1 \circ (\text{id}_1, \, h_1) = h_1 \circ (h_1, \, \text{id}_1)$;
\end{enumerate}

\item 2-cells consisting of the operadic composites of 2-cells $i_2$, $h_2$ and $v_2$ whose images under the underlying collection map are as follows

\begin{center}
\begin{tikzpicture}[node distance=2cm, auto]

\node (A) {$\cdot$};
\node (B) [node distance=2.5cm, right of =A] {$\cdot$};

\draw[->, bend left=40] (A) to node {$\text{id}_1$} (B);
\draw[->, bend right=40] (A) to node [swap] {$\text{id}_1$} (B);

\node () [node distance=1.25cm, right of =A] {$\Downarrow i_2$};

\node (P) [node distance=3cm, right of =B] {$\cdot$};
\node (Q) [right of =P] {$\cdot$};

\draw[->] (P) to node {} (Q);

\node (b) [node distance=0.5cm, right of=B] {};
\node (p) [node distance=0.5cm, left of=P] {};
\draw[|->, dashed] (b) to node {} (p);

\node (A') [node distance=2.2cm, below of=A] {$\cdot$};
\node (B') [node distance=2.5cm, right of =A'] {$\cdot$};

\draw[->, bend left=40] (A') to node {$h_1$} (B');
\draw[->, bend right=40] (A') to node [swap] {$h_1$} (B');

\node () [node distance=1.25cm, right of =A'] {$\Downarrow h_2$};

\node (P') [node distance=3cm, right of =B'] {$\cdot$};
\node (Q') [right of =P'] {$\cdot$};
\node (R') [right of =Q'] {$\cdot$};

\draw[->, bend left=40] (P') to node {} (Q');
\draw[->, bend right=40] (P') to node {} (Q');
\draw[->, bend left=40] (Q') to node {} (R');
\draw[->, bend right=40] (Q') to node {} (R');

\node () [node distance=1cm, right of =P'] {$\Downarrow$};
\node () [node distance=1cm, right of =Q'] {$\Downarrow$};

\node (b') [node distance=0.5cm, right of=B'] {};
\node (p') [node distance=0.5cm, left of=P'] {};
\draw[|->, dashed] (b') to node {} (p');




\node (A'') [node distance=2.2cm, below of=A'] {$\cdot$};
\node (B'') [node distance=2.5cm, right of =A''] {$\cdot$};

\draw[->, bend left=40] (A'') to node {$\text{id}_1$} (B'');
\draw[->, bend right=40] (A'') to node [swap] {$\text{id}_1$} (B'');

\node () [node distance=1.25cm, right of =A''] {$\Downarrow v_2$};

\node (P'') [node distance=3cm, right of =B''] {$\cdot$};
\node (Q'') [right of =P''] {$\cdot$};

\draw[->, bend left=60] (P'') to node {} (Q'');
\draw[->] (P'') to node {} (Q'');
\draw[->, bend right=60] (P'') to node {} (Q'');

\node (x) [node distance=1cm, right of =P''] {};
\node () [node distance=0.3cm, above of =x] {$\Downarrow$};
\node () [node distance=0.35cm, below of =x] {$\Downarrow$};

\node (b'') [node distance=0.5cm, right of=B''] {};
\node (p'') [node distance=0.5cm, left of=P''] {};
\draw[|->, dashed] (b'') to node {} (p'');

\end{tikzpicture}
\end{center}
subject to the equations below.
\begin{enumerate}[i')]
\begin{minipage}{0.5\linewidth}
\item[\vspace{\fill}]
\item $h_2 \circ (\text{id}_2, \, i_2 \circ (i_1)) = \text{id}_2$
\item $h_2 \circ (i_2 \circ (i_1), \, \text{id}_2) = \text{id}_2$
\item $v_2 \circ (\text{id}_2, \, i_2) = \text{id}_2$
\item $v_2 \circ (i_2, \, \text{id}_2) = \text{id}_2$
\end{minipage}
\begin{minipage}{0.5\linewidth}
\item[\vspace{\fill}]
\item $h_2 \circ (\text{id}_2, \, h_2) = h_2 \circ (h_2, \, \text{id}_2)$
\item $v_2 \circ (\text{id}_2, \, v_2) = v_2 \circ (v_2, \, \text{id}_2)$
\item $h_2 \circ (v_2, \, v_2) = v_2 \circ (h_2, \, h_2)$
\item $h_2 \circ (i_2, \, i_2) = i_2 \circ (h_1)$
\end{minipage}
\end{enumerate}
\vspace{1mm}
\end{list}

In this example $i_1$ and $h_1$ are the 1-cell generators and the three equations they satisfy are the 1-cell relations. Similarly, $i_2$, $h_2$ and $v_2$ are the 2-cell generators and the eight equations they satisfy are the 2-cell relations. These equations mean that $\bs{T_3}$ contains exactly one $k$-cell for each $k$-cell in $\bs{1}^{\ast}$. It follows that the underlying collection map of $\bs{T_3}$ is the identity map on $\bs{1}^*$, so the definition here agrees with our earlier one up to isomorphism; see Lemma \ref{initial-terminal} and Example \ref{GlobOp T for strict n cats}.

An algebra for $\theta: \bs{A_{T_3}} \longrightarrow \bs{A}$ for $\bs{T_3}$ on a 3-globular set $\bs{A}$ equips $\bs{A}$ with the structure of a strict 3-category. The generators $i_1$ and $h_1$ provide 1-cell identities and binary composition of 1-cells, respectively: given 1-pasting diagrams
\begin{center}
\begin{tikzpicture}[node distance=2cm, auto]

\node (A) {$A$};
\node (B) [right of=A] {$A$};
\node (C) [right of=B] {$B$};
\node (D) [right of=C] {$C$};

\draw[->] (B) to node {$a$} (C);
\draw[->] (C) to node {$b$} (D);

\end{tikzpicture}
\end{center}
in $\bs{A}$ we write $\theta((A), i_1) = i_1(A) = 1_A$ and $\theta((a,b), h_1) = h_1(a,b) = ab$,
\begin{center}
\begin{tikzpicture}[node distance=2.2cm, auto]

\node (A) {$A$};
\node (B) [right of=A] {$A$};
\node (C) [right of=B] {$A$};
\node (D) [right of=C] {$C.$};

\draw[->] (A) to node {$1_A$} (B);

\draw[->] (C) to node {$ab$} (D);

\end{tikzpicture}
\end{center}
The relations ensure that the two identity axioms and the associativity axiom for 1-cell composition, respectively, are satisfied. For example, equation ii) yields the equality,
$$1_Aa = h_1 (i_1(A), \, \text{id}_1(a)) = (h_1 \circ (i_1,  \text{id}_1)) (a) = \text{id}_1(a) = a.$$
Similarly, the 2-cell generators $i_2$, $h_2$ and $v_2$ provide 2-cell identities, horizontal composition of 2-cells and vertical composition of 2-cells in $\bs{A}$, respectively. Equalities i') - iv') yield the four unit axioms, equalities v') and vi') yield the two associativity axioms, equality vii') yields the interchange law, and equality viii') yields the axiom stating that the horizontal composite of two identity cells is another identity cell. 
\end{example}

Observe that the generators in the presentation for $\bs{T_3}$ above correspond to the basic composition operations and coherence cells present in a strict 3-category, and the relations correspond to the axioms. Presentations of globular operads for other varieties of higher category can be build analogously. The next example is a presentation for the 3-globular operad for tricategories, or fully weak 3-categories. This presentation will require more generators than the one above, since there more kinds of coherence cells in weak higher categories than in strict ones, but will require no relations, since there are no axioms for composition of $k$-cells in tricategories for any $k<3$.

\begin{example}\label{W_3} 
The 3-globular operad $\bs{W_3}$ for tricategories is the  \textit{contractible} 3-globular operad with

\begin{list}{•}{}

\item a single 0-cell, the identity $\text{id}_0$; 

\item 1-cells consisting of the free operadic composites of 1-cells $i_1$ and $h_1$ whose images under the underlying collection map are as in the presentation for $\bs{T_3}$ above; and

\item 2-cells consisting of the free operadic composites of 2-cells $i_2$, $h_2$, $v_2$,  $l_2$, $l'_2$, $r_2$, $r'_2$, $a_2$ and $a'_2$ whose images under the underlying collection map are as follows

\begin{center}
\begin{tikzpicture}[node distance=2cm, auto]

\node (A) {$\cdot$};
\node (B) [node distance=2.5cm, right of =A] {$\cdot$};

\draw[->, bend left=40] (A) to node {$\text{id}_1$} (B);
\draw[->, bend right=40] (A) to node [swap] {$\text{id}_1$} (B);

\node () [node distance=1.25cm, right of =A] {$\Downarrow i_2$};

\node (P) [node distance=3cm, right of =B] {$\cdot$};
\node (Q) [right of =P] {$\cdot$};

\draw[->] (P) to node {} (Q);

\node (b) [node distance=0.5cm, right of=B] {};
\node (p) [node distance=0.5cm, left of=P] {};
\draw[|->, dashed] (b) to node {} (p);

\node (A') [node distance=2.3cm, below of=A] {$\cdot$};
\node (B') [node distance=2.5cm, right of =A'] {$\cdot$};

\draw[->, bend left=40] (A') to node {$h_1$} (B');
\draw[->, bend right=40] (A') to node [swap] {$h_1$} (B');

\node () [node distance=1.25cm, right of =A'] {$\Downarrow h_2$};

\node (P') [node distance=3cm, right of =B'] {$\cdot$};
\node (Q') [right of =P'] {$\cdot$};
\node (R') [right of =Q'] {$\cdot$};

\draw[->, bend left=40] (P') to node {} (Q');
\draw[->, bend right=40] (P') to node {} (Q');
\draw[->, bend left=40] (Q') to node {} (R');
\draw[->, bend right=40] (Q') to node {} (R');

\node () [node distance=1cm, right of =P'] {$\Downarrow$};
\node () [node distance=1cm, right of =Q'] {$\Downarrow$};

\node (b') [node distance=0.5cm, right of=B'] {};
\node (p') [node distance=0.5cm, left of=P'] {};
\draw[|->, dashed] (b') to node {} (p');

\end{tikzpicture}
\end{center}


\begin{center}
\begin{tikzpicture}[node distance=2cm, auto]

\node (A'') {$\cdot$};
\node (B'') [node distance=2.5cm, right of =A''] {$\cdot$};

\draw[->, bend left=40] (A'') to node {$\text{id}_1$} (B'');
\draw[->, bend right=40] (A'') to node [swap] {$\text{id}_1$} (B'');

\node () [node distance=1.25cm, right of =A''] {$\Downarrow v_2$};

\node (P'') [node distance=3cm, right of =B''] {$\cdot$};
\node (Q'') [right of =P''] {$\cdot$};

\draw[->, bend left=60] (P'') to node {} (Q'');
\draw[->] (P'') to node {} (Q'');
\draw[->, bend right=60] (P'') to node {} (Q'');

\node (x) [node distance=1cm, right of =P''] {};
\node () [node distance=0.3cm, above of =x] {$\Downarrow$};
\node () [node distance=0.35cm, below of =x] {$\Downarrow$};

\node (b'') [node distance=0.5cm, right of=B''] {};
\node (p'') [node distance=0.5cm, left of=P''] {};
\draw[|->, dashed] (b'') to node {} (p'');

\node (X) [node distance=2.4cm, below of=A''] {$\cdot$};
\node (Y) [node distance=2.5cm, right of =X] {$\cdot$};

\draw[->, bend left=40] (X) to node {$h_1 \circ (\text{id}_1, i_1)$} (Y);
\draw[->, bend right=40] (X) to node [swap] {$\text{id}_1$} (Y);

\node () [node distance=1.25cm, right of =X] {$\Downarrow l_2$};

\node (U) [node distance=3cm, right of =Y] {$\cdot$};
\node (V) [node distance=2cm, right of =U] {$\cdot$};

\draw[->] (U) to node {} (V);

\node (y) [node distance=0.5cm, right of=Y] {};
\node (u) [node distance=0.5cm, left of=U] {};
\draw[|->, dashed] (y) to node {} (u);

\node (X') [node distance=2.3cm, below of=X] {$\cdot$};
\node (Y') [node distance=2.5cm, right of =X'] {$\cdot$};

\draw[->, bend left=40] (X') to node {$\text{id}_1$} (Y');
\draw[->, bend right=40] (X') to node [swap] {$h_1 \circ (\text{id}_1, i_1)$} (Y');

\node () [node distance=1.25cm, right of =X'] {$\Downarrow l'_2$};

\node (U') [node distance=3cm, right of =Y'] {$\cdot$};
\node (V') [node distance=2cm, right of =U'] {$\cdot$};

\draw[->] (U') to node {} (V');

\node (y') [node distance=0.5cm, right of=Y'] {};
\node (u') [node distance=0.5cm, left of=U'] {};
\draw[|->, dashed] (y') to node {} (u');

\node (X'') [node distance=2.5cm, below of=X'] {$\cdot$};
\node (Y'') [node distance=2.5cm, right of =X''] {$\cdot$};

\draw[->, bend left=40] (X'') to node {$h_1 \circ (i_1, \text{id}_1)$} (Y'');
\draw[->, bend right=40] (X'') to node [swap] {$\text{id}_1$} (Y'');

\node () [node distance=1.25cm, right of =X''] {$\Downarrow r_2$};

\node (U'') [node distance=3cm, right of =Y''] {$\cdot$};
\node (V'') [node distance=2cm, right of =U''] {$\cdot$};

\draw[->] (U'') to node {} (V'');

\node (y'') [node distance=0.5cm, right of=Y''] {};
\node (u'') [node distance=0.5cm, left of=U''] {};
\draw[|->, dashed] (y'') to node {} (u'');

\node (X''') [node distance=2.3cm, below of=X''] {$\cdot$};
\node (Y''') [node distance=2.5cm, right of =X'''] {$\cdot$};

\draw[->, bend left=40] (X''') to node {$\text{id}_1$} (Y''');
\draw[->, bend right=40] (X''') to node [swap] {$h_1 \circ (i_1, \text{id}_1)$} (Y''');

\node () [node distance=1.25cm, right of =X'''] {$\Downarrow r'_2$};

\node (U''') [node distance=3cm, right of =Y'''] {$\cdot$};
\node (V''') [node distance=2cm, right of =U'''] {$\cdot$};

\draw[->] (U''') to node {} (V''');

\node (y''') [node distance=0.5cm, right of=Y'''] {};
\node (u''') [node distance=0.5cm, left of=U'''] {};
\draw[|->, dashed] (y''') to node {} (u''');

\node (N) [node distance=2.5cm, below of=X'''] {$\cdot$};
\node (M) [node distance=2.5cm, right of =N] {$\cdot$};

\draw[->, bend left=40] (N) to node {$h_1 \circ (\text{id}_1, h_1)$} (M);
\draw[->, bend right=40] (N) to node [swap] {$h_1 \circ (h_1, \text{id}_1)$} (M);

\node () [node distance=1.25cm, right of =N] {$\Downarrow a_2$};

\node (H) [node distance=3cm, right of =M] {$\cdot$};
\node (I) [node distance=2cm, right of =H] {$\cdot$};
\node (J) [node distance=2cm, right of =I] {$\cdot$};
\node (K) [node distance=2cm, right of =J] {$\cdot$};

\draw[->] (H) to node {} (I);
\draw[->] (I) to node {} (J);
\draw[->] (J) to node {} (K);

\node (m) [node distance=0.5cm, right of=M] {};
\node (h) [node distance=0.5cm, left of=H] {};
\draw[|->, dashed] (m) to node {} (h);

\node (N') [node distance=2.5cm, below of=N] {$\cdot$};
\node (M') [node distance=2.5cm, right of =N'] {$\cdot$};

\draw[->, bend left=40] (N') to node {$h_1 \circ (h_1, \text{id}_1)$} (M');
\draw[->, bend right=40] (N') to node [swap] {$h_1 \circ (\text{id}_1, h_1)$} (M');

\node () [node distance=1.25cm, right of =N'] {$\Downarrow a'_2$};

\node (H') [node distance=3cm, right of =M'] {$\cdot$};
\node (I') [node distance=2cm, right of =H'] {$\cdot$};
\node (J') [node distance=2cm, right of =I'] {$\cdot$};
\node (K') [node distance=2cm, right of =J'] {$\cdot$};

\draw[->] (H') to node {} (I');
\draw[->] (I') to node {} (J');
\draw[->] (J') to node {} (K');

\node (m') [node distance=0.5cm, right of=M'] {};
\node (h') [node distance=0.5cm, left of=H'] {};
\draw[|->, dashed] (m') to node {} (h');

\end{tikzpicture}
\end{center}
\end{list}

An algebra $\theta:\bs{A_{W_3}} \longrightarrow \bs{A}$ for $\bs{W_3}$ is a tricategory in the sense of \cite{NG} with underlying 3-globular set $\bs{A}$. The 2-cell generators for $\bs{W_3}$ not appearing in the presentation for $\bs{T_3}$, namely $l_2$, $r_2$, $a_2$,  $l'_2$, $r'_2$, and $a'_2$, provide the left unit, right unit, and associativity coherence 2-cells, and their (weak) inverses, respectively. Note that since a morphism of algebras for an $n$-globular operad is a morphism of the underlying $n$-globular sets strictly preserving composition of pasting diagrams, a morphism of algebras for $\bs{W_3}$ is a \textit{strict} 3-functor between tricategories. 
\end{example}

\begin{remark}
The 3-globular operad $\bs{W_3}$ satisfies the definition of contractibility in dimension 3 by construction; there exists a unique 3-cell $\chi:x \longrightarrow x'$ of $\bs{W_3}$ satisfying $w(\chi) = \tau$, where $w:\bs{W_3} \longrightarrow \bs{1}^*$ is the underlying collection map, for every parallel pair $(x,x')$ of 2-cells satisfying $w(x)=w(x') = \partial(\tau)$. However, in order for Example \ref{W_3} to make sense, we need to show that $\bs{W_3}$ is actually contractible, that is, we need to show that condition i) in Definition \ref{contractible n-glob ops} is also satisfied. This is done in \cite[\S 9.1]{Me}. 
\end{remark}

\section{Symmetric operads}\label{symmop section}

Given a higher category $\mathcal{A}$ with exactly one $l$-cell for all $l<k$, the $k$-cells of $\mathcal{A}$ form their own algebraic structure. For example, the 1-cells of a strict higher category with a single 0-cell form a monoid; the binary operation is given by 1-cell composition, and the identity element is given by the single identity 1-cell. Similarly, the 2-cells of a strict higher category with a single object $A$ and a single 1-cell, which is necessarily the identity $1_A$, form a commutative monoid. This follows directly from the Eckmann-Hilton argument, discussed in Example \ref{Eck-Hil Pres} and Remark \ref{Eck-Hil remark}. 

As mentioned in the introduction, a preprint of Michael Batanin \cite{batanin} conjectures that it should possible to take `slices' of globular operads, thereby isolating their algebraic structure in each dimension. Intuitively, the $k^{th}$ slice of a globular operad for some theory of higher category should be the operad describing the kind of algebras formed by the $k$-cells of those higher categories whenever they contain only one $l$-cell for all $l< k$. It follows that these slices cannot always be plain operads. For example, the theory of commutative monoids is not strongly regular, and therefore cannot be described by a plain operad; see Section \ref{T-ops and algebras subsection}. Instead, we need symmetric operads, which are able to capture all regular algebraic theories, rather than just the strongly regular ones.

\subsection{Plain vs Symmetric operads}\label{symm vs plain op section}

In this section we define symmetric operads and their algebras by generalising a reformulated definition of plain operads. 

\begin{definition} An endofunctor $A:\mathbf{Set} \longrightarrow \mathbf{Set}$ is \textit{strongly analytic} if it isomorphic to a functor of the form
\begin{center}
$A(X) = \displaystyle\sum_{n \geqslant 0} X^n \times A[n]$
\end{center}
where $A[n]$ is a set for each $n \geqslant 0$, and we define $X^0 = \{ \star \}$ to be a single element set. A \textit{morphism $A \longrightarrow A'$ of strongly analytic functors} consists of a function $A[n] \longrightarrow A'[n]$ for each $n \in \mathbb{N}$.
\end{definition}

\begin{example}\label{identity is strongly analytic} The identity functor on $\mathbf{Set}$ is the strongly analytic functor $I$ for which $I[1] = \{ \star \}$ and $I[n] = \emptyset$ for all $n \neq 1$.
\end{example}

Strongly analytic functors are closed under composition; given any two such functors $A$ and $B$ we have, 
\[ BA[n] = \displaystyle\sum_{m \geqslant 0} \ \displaystyle\sum_{k_1 + ... + k_m = n} B[m] \times  A[k_1] \times ... \times A[k_m].\]
Furthermore, a morphism $A \longrightarrow A'$ of strongly analytic functors defines a natural transformation $A \Rightarrow A'$, and the horizontal composite of two such transformations $A \Rightarrow A'$ and $B \Rightarrow B'$ is another morphism of strongly analytic functors. The category of strongly analytic functors is therefore a monoidal category with respect to composition. 

\begin{definition}\label{PlainOpAlt} The \textit{category $\mathbf{Plain Op}$ of plain operads} is the category of monoids in the monoidal category of strongly analytic functors.
\end{definition}

The definition of plain operads stated here is equivalent to the one in Section \ref{T-Op Section}; see Definition \ref{T-operad def} and Example \ref{plain op}. Given a plain operad $A = (A, ids, comp)$, we think of $A[n]$ as a set of abstract operations of arity $n$. From Example \ref{identity is strongly analytic}, we see that the identity map $ids:I \Rightarrow A$ is just a function $\{ \ast \} \longrightarrow A[1]$. We refer to the element $\text{id} \in A[1]$ specified by this function as the \textit{identity} element. Let $A^2$ denote the composite endofunctor $AA$, so the composition map $comp:A^2 \Rightarrow A$ consists of a function 
\begin{center}
\begin{tikzpicture}[node distance=4.5cm, auto]

\node (A) {$A[m] \times A[k_1] \times ... \times A[k_m]$};
\node (B) [right of=A] {$A[k_1 + ... + k_m]$};

\draw[->] (A) to node {} (B);

\node (X) [node distance=6mm, below of=A] {};
\node (X') [node distance=1cm, right of=X] {$(\rho, \rho_1,...,\rho_m)$};
\node (Y) [node distance=6mm, below of=B]  {$\rho \circ (\rho_1,...,\rho_m)$};

\draw[|->, dashed] (X') to node {} (Y);

\end{tikzpicture}
\end{center}
for each tuple $(m,k_1,...,k_m)$ of natural numbers. The monoid axioms mean that this composition satisfies the familiar associativity and identity axioms.

\begin{example}\label{Plain op for monoids} 
Recall the plain operad for monoids constructed at the end of Example \ref{plain op}. In the language of Definition \ref{PlainOpAlt}, this is the plain operad given by equipping the terminal analytic functor $T$, defined by $T[n] = \{ \star \}$ for all $n \geqslant 0$, with its unique monoid structure.
\end{example}

We now move to the more general notion of symmetric operad.

\begin{definition} An endofunctor $A:\mathbf{Set} \longrightarrow \mathbf{Set}$ is \textit{analytic} if it is isomorphic to a functor of the form
$$A(X) = \displaystyle\sum_{n \geqslant 0} X^n \times_{S_n} A[n]$$
where each $A[n]$ is set equipped with a left group action $S_n \times A[n] \longrightarrow A[n]$ and $X^n \times_{S_n} A[n] $ is the coequaliser of the diagram 
\begin{center}
\begin{tikzpicture}[node distance=3.5cm, auto]

\node (A) {$X^n \times S_n \times A[n] $};
\node (B) [right of=A] {$X^n \times A[n]$};

\draw[transform canvas={yshift=0.7ex},->] (A) to node {} (B);
\draw[transform canvas={yshift=-0.3ex},->] (A) to node [swap] {} (B);

\end{tikzpicture}
\end{center}
in $\mathbf{Set}$. The second parallel function here is the one arising from the canonical right group action $X^n \times S_n  \longrightarrow X^n$ sending a tuple $(x_1,...,x_n,t)$ to $(x_1,...,x_n) \cdot t = (x_{t^{-1}(1)},...,x_{t^{-1}(n)})$. So  $X^n \times_{S_n} A[n]$ is the product set  $X^n \times A[n]$ under the equivalence relation $(x_1,...,x_n, \, t \cdot \rho ) \sim ( x_{t^{-1}(1)},...,x_{t^{-1}(n)}, \, \rho )$.
\end{definition}

\begin{definition} 
A \textit{morphism $A \longrightarrow A'$ of analytic functors} consists of a function $A[n] \longrightarrow A'[n] $ for each $n \in \mathbb{N}$
preserving the group actions, that is, satisfying $t \cdot f(\rho) = f(t \cdot \rho)$ for each $t \in S_n$ and $\rho \in A[n]$.
\end{definition}

As in the strongly analytic case, a morphism $A \longrightarrow A'$ of analytic functors defines a natural transformation $A \Rightarrow A'$.

\begin{lemma}\cite[\S 2.1]{AJ}
Analytic functors are closed under composition, so the category of analytic functors is a monoidal category with respect to composition.
\end{lemma}

\begin{definition} The \textit{category $\mathbf{Symm Op}$ of symmetric operads} is the category of monoids in the monoidal category of analytic functors.
\end{definition}

Given a symmetric operad $A= (A, ids, comp)$, we continue to think of $A[n]$ as a set of abstract operations of arity $n$, but now the composition maps $A[m] \times A[k_1] \times ... \times A[k_m] \longrightarrow A[k_1 + ... + k_m]$ satisfy two additional equivariance axioms. First, for any composite $\rho \circ (\rho_1,...,\rho_m)$ and $t_i \in S_{k_i}$ we have
\begin{equation}\label{SymmOpEq1}
\rho \circ (t_1 \cdot \rho_1,...,t_m \cdot \rho_m) = t_1 \oplus ... \oplus t_m \cdot ( \rho \circ (\rho_1,...,\rho_m)) \tag{$\filledsquare$},
\end{equation}
where $t_1 \oplus ... \oplus t_m$ denotes the element of $S_{k_1 + ... + k_m}$ made up of disjoint cycles corresponding to the $t_i$s. Second, for any $t \in S_m$ we have 
\begin{equation}\label{SymmOpEq2}
(t \cdot \rho) \circ (\rho_{t(1)},...,\rho_{t(m)}) = t \cdot ( \rho \circ (\rho_1,...,\rho_m))  \tag{$\filledtriangleup$}
\end{equation}
where by abuse of notation the $t$ on the right hand side denotes the element of $S_{k_1 + ... + k_m}$ given by the block permutation corresponding to the element $t \in S_m$.
The structure provided by the identity map, as well as the equational associativity and identity laws induced by the monoid axioms, are the same as in plain operads.  

Observe that, as a monoid in the category of endofunctors on $\mathbf{Set}$, a symmetric operad is a kind of monad on $\mathbf{Set}$. 

\begin{definition} The \textit{category $A \mhyphen \mathbf{Alg}$ of algebras} for a symmetric operad $A$ is the category of monad algebras.
\end{definition}

An algebra for a symmetric operad $A$ on a set $X$ is then a function
\begin{center}
\begin{tikzpicture}[node distance=3.4cm, auto]

\node (A) {};
\node (B) [right of=A] {$X$};
\node (a) [node distance=1mm, below of=A] {$\underset{m \geqslant 0}{\mathlarger{\sum}} X^m \times_{S_m} A[m]$};
\node (a'') [node distance=1.3cm, right of=A] {};

\node () [node distance=2cm, left of=A] {$A(X)=$};

\draw[->] (a'') to node {} (B);

\node (X) [node distance=7mm, below of=A] { \hspace{4mm} $(x_1,...,x_m, \rho)$};
\node (Y) [node distance=4cm, right of=X]  {$\rho(x_1,...,x_m)$};

\draw[|->, dashed] (X) to node {} (Y);

\end{tikzpicture}
\end{center}
satisfying the same multiplication and unit axioms as the algebras for plain operads do. We note that this function imposes an equality 
$$(t \cdot \rho)(x_1,...,x_m) = \rho(x_{t^{-1}(1)},...,x_{t^{-1}(m)})$$
for any $t \in S_m$. We think of $\rho(x_1,...,x_m) \in X$ as the result of applying $\rho$ to the ordered list $(x_1,...,x_m)$ of elements of $X$, so this equality tells us that $t \cdot \rho \in A[m]$ should be thought of as the operation of arity $m$ given by first rearranging elements using $t$, and then applying $\rho$. We may visualise this by extending the tree diagram notation used to describe plain operads to the context of symmetric operads. For example, when $m=3$ and $t = (1 \ 2 \ 3) \in S_3$, the equation above can be represented as, 
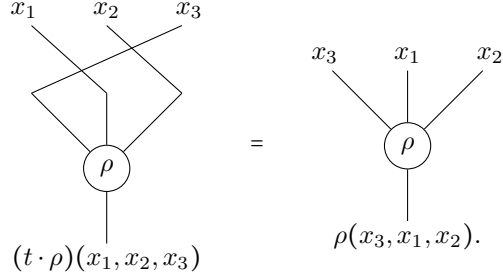
\begin{center}
\begin{tikzpicture}[node distance=5cm, auto]

\node[circle,draw] (A) at (0,0) {$\rho$};

\draw(A) to (0,-1);
\draw(A) to (1,1);
\draw(A) to (0,1);
\draw(A) to (-1,1);

\node () at (-1.1,2.1) {$x_1$};
\node () at (0,2.1) {$x_2$};
\node () at (1.1,2.1) {$x_3$};

\draw(1,1) to (0,1.9);
\draw(0,1) to (-1,1.9);
\draw(-1,1) to (1,1.9);

\node () at (0,-1.2) {$(t \cdot \rho)(x_1,x_2,x_3)$};

\node[circle,draw] (B) at (4,0.3) {$\rho$};

\draw(B) to (4,-0.7);
\draw(B) to (5,1.3);
\draw(B) to (4,1.3);
\draw(B) to (3,1.3);

\node () at (2.9,1.5) {$x_3$};
\node () at (4,1.5) {$x_1$};
\node () at (5.1,1.5) {$x_2$};

\node () at (4,-0.9) {$\rho(x_3,x_1,x_2)$.};

\node(X) at (2,0.3) {$=$};

\end{tikzpicture}
\end{center}
The equivariance axioms \ref{SymmOpEq1} and \ref{SymmOpEq2} can also be expressed intuitively in the language of tree diagrams.

\begin{example}\label{SymmOp for sets} The identity monad on $\mathbf{Set}$ is the initial symmetric operad whose underlying analytic functor is the identity functor $I$; see Example \ref{identity is strongly analytic}.
\end{example}

The extra structure given by the group actions means that symmetric operads are able to encode all regular finitary algebraic theories, rather than just the \textit{strongly} regular ones, as demonstrated by our next example. 

\begin{example}\label{symm op for comm monoids} The free commutative monoid monad on $\mathbf{Set}$ is the terminal symmetric operad, which we denote by $C$. The underlying analytic functor is given by $C[n] = \{ \star \}$ for all $n \geqslant 0$. To see this, observe that
$$C(X) = \displaystyle\sum_{n \geqslant 0} X^n \times_{S_n} \{ \star \} \cong \displaystyle\sum_{n \geqslant 0} X^n /{\sim}$$
where $X^n /{\sim}$ is just the quotient set given by $(x_1,...,x_n) \sim (x_{t^{-1}(1)},...,x_{t^{-1}(n)})$ for any $t \in S_n$. It is now a straightforward exercise to show that an algebra for $C$ on a set $X$ is precisely a commutative monoid with underlying set $X$.
\end{example}

There is a canonical forgetful functor $\mathbf{Symm Op} \longrightarrow \mathbf{Plain Op}$ sending a symmetric operad to a plain operad by forgetting the group actions. For example, the value of the forgetful functor at the symmetric operad $C$ for commutative monoids is the plain operad $T$ for monoids, since $C[n] = T[n] = \{ \star \}$ for all $n \in \mathbb{N}$; see Examples \ref{Plain op for monoids} and \ref{symm op for comm monoids}. This functor is part of a free-forgetful adjunction,
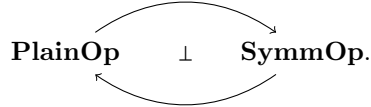
\begin{center}
\begin{tikzpicture}[node distance=3.2cm, auto]

\node (A) {$\mathbf{Plain Op}$};
\node (B) [right of=A] {$\mathbf{Symm Op}$.};

\draw[->, bend left=35] (A) to node {} (B);
\draw[->, bend left=35] (B) to node {} (A);

\node (W) [node distance=1.6cm, right of=A] {$\perp$};

\end{tikzpicture}
\end{center}
The free symmetric operad $A_S$ on a plain operad $A$ is given by $A_S[n] = S_n \times A[n]$. Note that for any set $X$ we have
$$A_S(X) = \displaystyle\sum_{n \geqslant 0} X^n \times_{S_n} A_S[n] = \displaystyle\sum_{n \geqslant 0} X^n \times_{S_n} \big(S_n \times A[n]\big) \cong \displaystyle\sum_{n \geqslant 0} X^n \times A[n] = A(X)$$
so the free functor is really just an inclusion functor. We refer to the image of a plain operad $A$ under the inclusion functor as the \textit{symmetrisation} of $A$.

\subsection{Presentations for Symmetric Operads}\label{pres for symmop section}

Our goal is to be able to take the $k$-dimensional data of a globular operad and use it to identify a corresponding symmetric operad. Since we construct our globular operads by building presentations, we will need to do the same for symmetric operads. For this, we use the fact that presentations for algebraic structures can be defined in terms of free-forgetful adjunctions and coequalisers\footnote{This is how the theory of presentations for globular operads that we use here is defined formally in \cite{Me}}. For example, a group presentation is a set $J$ whose elements we call \textit{generators}, a set $R$ whose elements we call \textit{relations} and a pair of functions 
\begin{center}
\begin{tikzpicture}[node distance=2.5cm, auto]

\node (A) {$UF(J)$};
\node (B) [left of=A] {$R$};

\draw[transform canvas={yshift=0.6ex},->] (B) to node {$e$} (A);
\draw[transform canvas={yshift=-0.4ex},->] (B) to node [swap] {$q$} (A);

\end{tikzpicture}
\end{center}
where $F$ and $U$ are the adjoint free-forgetful functors $F \dashv U:\mathbf{Grp} \longrightarrow \mathbf{Set}$.
We say that $(J,R,e,q)$ is a presentation for a group $G$ if the coequaliser of the diagram
\begin{center}
\begin{tikzpicture}[node distance=2.8cm, auto]

\node (A) {$F(J)$};
\node (B) [left of=A] {$F(R)$};

\draw[transform canvas={yshift=0.7ex},->] (B) to node {$e$} (A);
\draw[transform canvas={yshift=-0.3ex},->] (B) to node [swap] {$q$} (A);

\end{tikzpicture}
\end{center}
is isomorphic to $G$. We now describe the analogous adjunction for symmetric operads.

Recall the free functor $T \mhyphen \mathbf{Coll} \longrightarrow T \mhyphen \mathbf{Op}$ that is left adjoint to the canonical forgetful functor for finitary cartesian monads $T$ on presheaf categories defined in Theorem \ref{T-op adjunction}. When $T$ is the free monoid monad on $\mathbf{Set}$, this gives the adjunction
\begin{center}\label{free plain collection adjunction}
\begin{tikzpicture}[node distance=3cm, auto]

\node (A) {$\mathbf{Set} / \mathbb{N}$};
\node (B) [right of=A] {$\mathbf{Plain Op}$};

\draw[->, bend left=35] (A) to node {} (B);
\draw[->, bend left=35] (B) to node {} (A);

\node (W) [node distance=1.5cm, right of=A] {$\perp$};

\end{tikzpicture}
\end{center}
for which the right adjoint is the functor sending a plain operad $A$ to its underlying collection $a:\sum A[n] \longrightarrow \mathbb{N}$, where $a^{-1}(n) = A[n]$ for each $n \in \mathbb{N}$. As shown in Example \ref{free plain operads}, the left adjoint sends a collection $c:C \longrightarrow \mathbb{N}$ to the plain operad $C^{\infty}$ whose elements are the free operadic composites of the elements of $C$. Recall that every non-identity element of $C^{\infty}$ can be expressed in the form $Q \circ (\rho_1,...,\rho_m)$ where $Q$ is some smaller operadic composite, $\rho_i = \text{id}$ or $\rho_i \in C$, and $\rho_i \neq \text{id}$ for at least one $i$. 
The element $Q \circ (\rho_1,...,\rho_m)$ has arity $c(\rho_1)+...+c(\rho_m)$, where we define $c(\text{id}) = 1$, so 
$$Q \circ (\rho_1,...,\rho_m) \in C^{\infty}\big[c(\rho_1)+...+c(\rho_m)\big].$$

\begin{definition}\label{free symmetric operad adjunction}
The \textit{free-forgetful adjunction} $F \dashv U : \mathbf{Symm Op} \longrightarrow \mathbf{Set}/\mathbb{N}$ is given by composing the adjunction between $\mathbf{PlainOp}$ and $\mathbf{SymmOp}$ described at the end of previous section with the adjunction between $\mathbf{Set}/\mathbb{N}$ and $\mathbf{PlainOp}$ above. The value of the free functor at a collection $c:C \longrightarrow \mathbb{N}$ is the symmetric operad $F(C)$ given by 
$$F(C)[n] = S_n \times C^{\infty}[n].$$
\end{definition} 

We use this adjunction to define presentations for symmetric operads. 

\begin{definition} A \textit{presentation $P=(J,R,e,q)$ for a symmetric operad} consists of a pair of collections $j:J \longrightarrow \mathbb{N}$ and $r:R \longrightarrow \mathbb{N}$ together with a pair of parallel morphisms 
\begin{center}
\begin{tikzpicture}[node distance=2.6cm, auto]

\node (A) {$R$};
\node (B) [right of=A] {$UF(J)$};
\node (C) [node distance=1.9cm, right of=B] {};
\node () [node distance=1mm, below of=C] {$= \underset{n \geqslant 0}{\mathlarger{\sum}} S_n \times J^{\infty}[n]$};

\draw[transform canvas={yshift=0.6ex},->] (A) to node {$e$} (B);
\draw[transform canvas={yshift=-0.4ex},->] (A) to node [swap] {$q$} (B);

\end{tikzpicture}
\end{center}
in $\mathbf{Set}/\mathbb{N}$, or equivalently, functions $e_n, q_n : r^{-1}(n) \longrightarrow S_n \times J^{\infty}[n]$ for each $n \geqslant 0$. We refer to the elements of $J$ as the \textit{generators} and to the elements of $R$ as the \textit{relations}.
\end{definition}

It is shown in \cite[\S 1.3]{IW} that the category $\mathbf{SymmOp}$ of symmetric operads is complete and cocomplete, allowing for the following definition.

\begin{definition} Let $A$ be a symmetric operad. A \textit{presentation for} $A$ is a presentation $P=(J,R,e,q)$ for which the coequaliser of the following diagram in $\mathbf{SymmOp}$ is isomorphic to $A$
\begin{center}
\begin{tikzpicture}[node distance=2.8cm, auto]

\node (A) {$F(R)$};
\node (B) [right of=A] {$F(J)$.};

\draw[transform canvas={yshift=0.75ex},->] (A) to node {$e$} (B);
\draw[transform canvas={yshift=-0.25ex},->] (A) to node [swap] {$q$} (B);

\end{tikzpicture}
\end{center}
\end{definition}

\begin{example}\label{Presentation for I}
The free symmetric operad on the empty collection is the symmetric operad $I$ for sets; see Example \ref{SymmOp for sets}. This follows immediately from the fact that left adjoints preserve colimits (and in particular, initial objects), but is also easy to show directly. We may therefore build a presentation for $I$ by setting $J = R = \emptyset$, so both $e$ and $q$ must be the unique map from the empty collection,
\begin{center}
\begin{tikzpicture}[node distance=3.2cm, auto]

\node (A) [transform canvas={yshift=0.2ex}] {$\emptyset$};
\node (B) [right of=A] {$UF(\emptyset) = U(I)$.};

\draw[transform canvas={yshift=0.7ex},->] (A) to node {} (B);
\draw[transform canvas={yshift=-0.3ex},->] (A) to node [swap] {} (B);

\end{tikzpicture}
\end{center}
The diagram in $\mathbf{Set} / \mathbb{N}$ above is equivalently the diagram
\begin{center}
\begin{tikzpicture}[node distance=2cm, auto]

\node (A) {$I$};
\node (B) [right of=A] {$I$};
\node (C) [node distance=7.5mm, left of=A] {$F(\emptyset) = $};

\draw[transform canvas={yshift=0.5ex},->] (A) to node {} (B);
\draw[transform canvas={yshift=-0.5ex},->] (A) to node [swap] {} (B);

\end{tikzpicture}
\end{center}
in $\mathbf{SymmOp}$ under adjunction. The coequaliser of this diagram is trivially $I$, so $(\emptyset, \emptyset, !, !)$ is a presentation for the symmetric operad $I$ for sets.
\end{example}

\begin{example}\label{symm op for magmas} A presentation for the symmetric operad for magmas is given by $(J, \emptyset, !, !)$ where $J = \{ b \}$ is a single element set and $j:J \longrightarrow \mathbb{N}$ is defined by $j(b) = 2$. Then we have the diagram
\begin{center}
\begin{tikzpicture}[node distance=2.5cm, auto]

\node (A) [transform canvas={yshift=0.3ex}] {$I$};
\node (B) [right of=A] {$F(J)$};

\draw[transform canvas={yshift=0.75ex},->] (A) to node {} (B);
\draw[transform canvas={yshift=-0.25ex},->] (A) to node [swap] {} (B);

\end{tikzpicture}
\end{center}
in $\mathbf{SymmOp}$. Since $I$ is initial, the coequaliser of this diagram is the free symmetric operad $F(J)$, which may be regarded as (the symmetrisation of) the free \textit{plain} operad  $J^{\infty}$ on the collection $J \longrightarrow \mathbb{N}$. As we saw in Example \ref{free op for magmas}, an algebra for this operad on a set $X$ is precisely a magma with underlying set $X$.
\end{example}

\begin{notation} We denote by 0 the identity element of the symmetric group $S_n$.
\end{notation}

\begin{example}\label{pres for comm monoids} A presentation for the terminal symmetric operad for commutative monoids described in Example \ref{symm op for comm monoids} is given by the set $J = \{i,b\}$ of generators, where $j(i) = 0$ and $j(b) = 2$, and the set $R = \{ u, v, a, c \}$ of relations, where $r(u)=r(v)=1$, $r(a) = 3$ and $r(c) = 2$, together with morphisms 
\begin{center}
\begin{tikzpicture}[node distance=3.1cm, auto]

\node (A) [transform canvas={yshift=0.2ex}] {$R$};
\node (B) [right of=A] {};
\node (C) [node distance=2cm, right of=A] {};
\node () [node distance=0.1cm, below of=B] {$\underset{n \geqslant 0}{\mathlarger{\sum}} S_n \times J^{\infty}[n]$};

\draw[transform canvas={yshift=0.7ex},->] (A) to node {$e$} (C);
\draw[transform canvas={yshift=-0.3ex},->] (A) to node [swap] {$q$} (C);

\end{tikzpicture}
\end{center}
in $\mathbf{Set} / \mathbb{N}$ defined as follows.
\vspace{1mm}
\begin{enumerate}[$\bullet$]
\begin{minipage}{0.265\linewidth}
\item
$e(u) = ( 0, \, b \circ (\text{id},i) )$\\
$q(u) = ( 0, \, \text{id})$
\end{minipage}
\begin{minipage}{0.265\linewidth}
\item
$e(v) = ( 0, \, b \circ (i,\text{id}) )$\\ 
$q(v) = ( 0, \, \text{id} )$
\end{minipage}
\begin{minipage}{0.265\linewidth}
\item
$e(a) =( 0, \, b \circ (\text{id},b) )$\\
$q(a) = ( 0, \, b \circ (b, \text{id}) )$
\end{minipage}
\begin{minipage}{0.265\linewidth}
\item
$e(c) = ( 0, \, b )$\\ 
$q(c) = ( (1 \ 2), \, b )$
\end{minipage}
\end{enumerate}
\vspace{1mm}
The coequaliser of the diagram 
\begin{center}
\begin{tikzpicture}[node distance=2.7cm, auto]

\node (A) {$F(R)$};
\node (C) [right of=A] {$F(J)$};

\draw[transform canvas={yshift=0.7ex},->] (A) to node {$e$} (C);
\draw[transform canvas={yshift=-0.3ex},->] (A) to node [swap] {$q$} (C);

\end{tikzpicture}
\end{center}
in $\mathbf{SymmOp}$ is then the symmetric operad $C$ whose elements are the free operadic composites of elements $i \in C[0]$ and $b \in C[2]$, subject to the equalities below.
\vspace{1mm}
\begin{enumerate}[i)]
\begin{minipage}{0.25\linewidth}
\item $b \circ (\text{id},i) = \text{id}$
\end{minipage}
\begin{minipage}{0.25\linewidth}
\item $b \circ (i,\text{id}) = \text{id} $ 
\end{minipage}
\begin{minipage}{0.32\linewidth}
\item $b \circ (\text{id},b) =  b \circ (b, \text{id})$
\end{minipage}
\begin{minipage}{0.2\linewidth}
\item $b = (1 \ 2) \cdot b$ 
\end{minipage}
\end{enumerate}
\vspace{1mm}

An algebra for $C$ on a set $X$ is precisely a commutative monoid with underlying set $X$. The generator $b$ provides a binary operation on the elements of $X$, as in the previous example, while $i$ provides an identity element. Operadic composites of the generators $i$ and $b$ provide operations derived from these basic ones. For instance, if we write $b(x_1,x_2) = x_1x_2$, then given elements $x_1, x_2, x_3 \in X$ we have,
$$\big((1 \ 2 \ 3) \cdot (b \circ (\text{id}, b) ) \big) \, (x_1, x_2, x_3) = (b \circ (\text{id}, b)) \, (x_3,x_1,x_2) = b \, (\text{id}(x_3), \, b(x_1,x_2)) = x_3(x_1x_2).$$
Equalities i) - iv) yield left and right unit axioms, an associativity axiom, and a commutativity axiom, respectively. For example, by the fourth equality we have,
$$x_1x_2 = b(x_1,x_2) = ( (1 \ 2) \cdot b ) \, (x_1,x_2) = b(x_2,x_1) = x_2x_1.$$
\end{example}

From this point onwards, we will define symmetric operads by specifying the generators and relations of a presentation. As in the example above, the relations will be expressed as a list of equations between operadic composites of the generators. Below is an example of a presentation for the operad for monoids.

\begin{example}\label{pres for monoids} The operad $T$ for monoids consists of the operadic composites of elements $i \in T[0]$ and $b \in T[2]$ subject to the following equalities.
\vspace{1mm}
\begin{enumerate}[i)]
\begin{minipage}{0.3\linewidth}
\item $b \circ (\text{id},i) = \text{id}$
\item $b \circ (i,\text{id}) = \text{id} $
\item $b \circ (\text{id},b) =  b \circ (b, \text{id})$
\end{minipage}
\end{enumerate}
Observe that this is the same as the presentation for the symmetric operad $C$ for commutative monoids with a single exception - the relation yielding a commutativity axiom in the algebras for $C$. It follows that an algebra for $T$ is precisely a (not necessarily commutative) monoid.
\end{example}

The next lemma shows that it is sometimes possible to use a presentation for a symmetric operad to determine whether it is just (the symmetrisation of) a plain operad.

\begin{lemma}\label{prop - presentations for plain} Let $P=(J,R,e,q)$ be a presentation for a symmetric operad $A$. If for each $n \geqslant 0$ the functions $e_n, q_n:r^{-1}(n) \longrightarrow S_n \times J^{\infty}[n]$ are of the form $\langle 0, f_n \rangle$ and $\langle 0, g_n \rangle$, respectively, where $0:r^{-1}(n) \longrightarrow S_n$ denotes the constant function at the identity element $0$ of $S_n$, and $f_n$ and $g_n$ are functions $r^{-1}(n) \longrightarrow J^{\infty}[n]$, then $A$ is isomorphic to the symmetrisation of a plain operad.
\end{lemma}

In practice, this means that when have a presentation for a symmetric operad for which no non-identity elements of $S_n$ appear in the list of relations, the resulting operad is a plain operad. For instance, we can see from our last example that the operad for monoids is plain. We provide another example below.

\begin{example}\label{symm op for double monoids} A double monoid with shared unit is a set equipped with two monoid structures sharing the same identity element. The corresponding symmetric operad consists of the operadic composites of elements $i \in D[0]$ and $b, b' \in D[2]$, subject to the equalities below. 
\begin{enumerate}[i)]
\begin{minipage}{0.4\linewidth}
\item[\vspace{\fill}]
\item $b \circ (\text{id},i) = \text{id}$
\item $b \circ (i,\text{id}) =  \text{id}$
\item $b \circ (\text{id},b) = b \circ (b, \text{id})$
\end{minipage}
\begin{minipage}{0.4\linewidth}
\item[\vspace{\fill}]
\item $b' \circ (\text{id},i) = \text{id}$
\item $b' \circ (i,\text{id}) = \text{id}$
\item $b' \circ (\text{id},b') =b' \circ (b', \text{id})$
\end{minipage}
\end{enumerate}
\end{example}

Each symmetric operad has many presentations. To end this section, we provide an example of another presentation for the commutative monoid operad constructed using the Eckmann-Hilton argument, which states that given a set $X$ equipped with two unital binary operations $\diamond$ and $\smallsquare$ satisfying 
\begin{equation} \label{EckHil}
(w \smallsquare x) \diamond (y \smallsquare z) = (w \diamond y) \smallsquare (x \diamond z) \tag{$\star$}
\end{equation}
for all $w,x,y,z \in X$, the operations $\smallsquare$ and $\diamond$ are equal, and are associative and commutative. In other words, such a structure is just a commutative monoid.

\begin{example}\label{Eck-Hil Pres} The symmetric operad $C$ for commutative monoids consists of the operadic composites of elements $i,i' \in C[0]$ and $b,b' \in C[2]$, subject to the following equations.
\begin{enumerate}[i)]
\begin{minipage}{0.3\linewidth}
\item[\vspace{\fill}]
\item $b \circ (\text{id},i) = \text{id}$
\item $b \circ (i,\text{id}) = \text{id}$
\end{minipage}
\begin{minipage}{0.3\linewidth}
\item[\vspace{\fill}]
\item $b' \circ (\text{id},i') = \text{id}$
\item $b' \circ (i',\text{id}) = \text{id}$
\end{minipage}
\begin{minipage}{0.3\linewidth}
\item[\vspace{\fill}]
\item $b \circ (b',b') = (2 \ 3) \cdot ( b' \circ (b,b)) $
\item [\vspace{\fill}]
\end{minipage}
\end{enumerate}
Given an algebra for this operad on a set $X$, the generators $i$ and $i'$ each pick out a single element of $X$, and $b$ and $b'$ each provide a binary operation. Equations i) and ii) above state that the element specified by $i$ is an identity for the binary operation provided by $b$. Likewise, equations iii) and iv) state that the element specified by $i'$ is an identity for the binary operation provided by $b'$. The final equation yields an axiom corresponding to equation \eqref{EckHil} above. 

To see that this is the free commutative monoid operad, we have to show that 
$i = i'$ (the identities coincide),
$b = b'$ (the binary operations are equal),
$b = (1 \ 2) \cdot b$ (the binary operation is commutative), and
$b \circ (\text{id}, b) = b \circ (b, \text{id})$ (the binary operation is associative).
These equations can be seen to hold by the arguments outlined below, each of which makes use of the second equivariance axiom (\ref{SymmOpEq2}). To begin, note that for any $a \in C[n]$ we have,
$$a = \text{id} \circ (a) = (b \circ (\text{id},i)) \circ (a) = b \circ (\text{id} \circ (a),i) = b \circ (a,i)$$
and similarly, $a = b \circ (i,a)$, $a = b' \circ (a,i')$, and $a = b' \circ (i',a)$. Then,
\begin{align*}
i = b \circ \big( b' \circ (i',i), \, b' \circ (i,i') \big) =  \big( b \circ (b',b') \big)  \circ (i',i,i,i') &= \big( (2 \ 3) \cdot \big( b' \circ (b,b) \big) \big) \circ (i',i,i,i') \\
&= \big( b' \circ (b,b) \big) \circ (i'i,i,i') \\
&= \big( b' \circ (b \circ (i',i),b \circ (i,i')) \big) \\
&= i'.
\end{align*}
Now, using the fact that $i=i'$ we can similarly show that $b = b'$, which we can use in turn to show that $b = (1 \ 2) \cdot  b$, and finally that $b \circ (\text{id},b) = b \circ (b, \text{id})$.
\end{example}

\begin{remark}\label{Eck-Hil remark}
The Eckmann-Hilton argument shows that the 2-cells of a strict 2-category with a single 0-cell and a single 1-cell form a commutative monoid. Horizontal and vertical composition each define binary operations on the 2-cells, and the single identity 2-cell is a unit for both of these operations. The interchange law tells us that these operations satisfy Equation (\ref{EckHil}). More generally, by iterating the Eckmann-Hilton argument, we see that for $k \geqslant 2$, the $k$-cells of a strict higher category with a single $l$-cell for all $l<k$ form a commutative monoid.
\end{remark}

\section{Slices}\label{slices section}

In this section we provide a concrete definition of slices for globular operads in terms of presentations. Since each globular operad has many presentations, we show in Theorem \ref{slices don't depend on presentations} that the slices of a globular operad $\bs{G}$ do not depend on the choice of presentation for $\bs{G}$. While the statement of our definition requires several technical results and some additional machinery, it is fairly intuitive to apply in practice. Throughout Section \ref{slices subsection}, we use the example of the presentation for the globular operad for strict 3-categories given in Example \ref{pres for T_3} to illustrate the ideas, and show in Example \ref{slices of T_3} that the first slice of this operad is the plain operad for monoids, while the $k^{th}$ slice for all $k>1$ is the symmetric operad for commutative monoids, as desired. In Section \ref{examples of slices subsection}, we provide several additional examples for various theories of higher category.

\subsection{The Definition}\label{slices subsection} Our goal is to turn the $k$-cell generators and relations of a presentation for a globular operad into the generators and relations of a presentation for a symmetric operad. We begin by examining the structure of pasting diagrams in globular sets.

\begin{definition}\label{k-ball} Let $n$ be natural number. The $n$-\textit{ball} $\bs{B_n}$ is the globular set

\begin{center}
\begin{tikzpicture}[node distance=2.3cm, auto]

\node (A) {$\{0,1\}$};
\node (C) [left of=A] {$. . .$};
\node (D) [left of=C] {$\{0,1\}$};
\node (E) [left of=D] {$ \{ \filledstar \}$};
\node (F) [node distance=2.1cm, left of=E] {$\emptyset$};
\node (G) [node distance=2cm, left of=F] {$\emptyset$};
\node (H) [node distance=2cm, left of=G] {$. . .$};

\draw[transform canvas={yshift=0.5ex},->] (C) to node {\small ${0}$} (A);
\draw[transform canvas={yshift=-0.5ex},->] (C) to node [swap] {\small ${1}$} (A);

\draw[transform canvas={yshift=0.5ex},->] (D) to node {\small ${0}$} (C);
\draw[transform canvas={yshift=-0.5ex},->] (D) to node [swap] {\small ${1}$} (C);

\draw[transform canvas={yshift=0.5ex},->] (E) to node {\small ${0}$} (D);
\draw[transform canvas={yshift=-0.5ex},->] (E) to node [swap] {\small ${1}$} (D);

\draw[transform canvas={yshift=0.5ex},->] (F) to node {} (E);
\draw[transform canvas={yshift=-0.5ex},->] (F) to node [swap] {} (E);

\draw[transform canvas={yshift=0.5ex},->] (G) to node {} (F);
\draw[transform canvas={yshift=-0.5ex},->] (G) to node [swap] {} (F);

\draw[transform canvas={yshift=0.5ex},->] (H) to node {} (G);
\draw[transform canvas={yshift=-0.5ex},->] (H) to node [swap] {} (G);

\end{tikzpicture}
\end{center}
consisting of a single $n$-cell. Here the arrows labeled $0$ and $1$ represent the constant functions.
\end{definition}

\begin{definition}\label{arity function}
For each $n \in \mathbb{N}$ we define a function $| - |_n:\mathbf{GSet}(\bs{B_n}, \bs{1}^*) \longrightarrow \mathbb{N}$ from the set of $n$-pasting diagrams in the terminal globular set $\bs{1}$; see Notation \ref{TerminalGlobSet}. Given an $n$-pasting diagram $\pi$ in $\bs{1}$, $| \pi |_n$ is the number of copies of the unique $n$-cell of $\bs{1}$ contained in $\pi$. 
\end{definition}

We omit the subscript on $| - |_n$ when it can be understood from context.

\begin{example}\label{arity function examples}
For the following 2- and 3-pasting diagrams we have $|\tau|_2=2$ and $|\Psi|_3 = 10$, respectively.
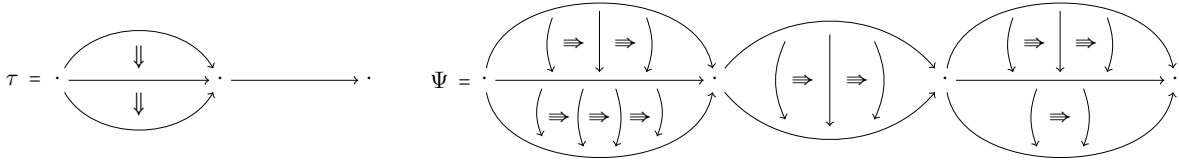
\begin{center}
\resizebox{1\textwidth}{!}{
\begin{tikzpicture}[node distance=3.4cm, auto]

\node (a) {$\cdot$};
\node (b) [node distance=2.4cm, right of=a] {$\cdot$};
\node (c) [node distance=2.2cm, right of=b] {$\cdot$};
\node () [node distance=0.5cm, left of=a] {$\tau \, =$};

\draw[->, bend left=60] (a) to node {} (b);
\draw[->] (a) to node [swap] {} (b);
\draw[->, bend right=60] (a) to node [swap] {} (b);

\node (x) [node distance=1.2cm, right of=a] {};
\node () [node distance=3.5mm, above of=x] {$\Downarrow $};
\node () [node distance=3.5mm, below of=x] {$\Downarrow $};

\draw[->] (b) to node {} (c);

\node (A) [node distance=1.7cm, right of=c] {$\cdot$};
\node (B) [right of=A] {$\cdot$};
\node (C) [right of=B] {$\cdot$};
\node (D) [right of=C] {$\cdot$};
\node () [node distance=0.5cm, left of=A] {$\Psi \, =$};

\draw[->, bend left=80] (A) to node {} (B);
\draw[->] (A) to node [swap] {} (B);
\draw[->, bend right=80] (A) to node [swap] {} (B);

\node (X') [node distance=1.05cm, right of=A] {};
\node (X'') [node distance=1.7cm, right of=A] {};
\node (X''') [node distance=2.35cm, right of=A] {};
\node (Y') [node distance=10.5mm, above of=X'] {};
\node (Y'') [node distance=11.5mm, above of=X''] {};
\node (Y''') [node distance=10.5mm, above of=X'''] {};

\draw[->, bend right=20] (Y') to node [swap] {} (X');
\draw[->] (Y'') to node {} (X'');
\draw[->, bend left=20] (Y''') to node {} (X''');

\node (Z) [node distance=0.55cm, above of=X''] {};
\node() [node distance=0.38cm, left of=Z] {$\Rrightarrow$};
\node() [node distance=0.38cm, right of=Z] {$\Rrightarrow$};

\node (P) [node distance=9mm, right of=A] {};
\node (P') [node distance=15mm, right of=A] {};
\node (P'') [node distance=19mm, right of=A] {};
\node (P''') [node distance=25mm, right of=A] {};
\node (Q) [node distance=9.5mm, below of=P] {};
\node (Q') [node distance=11mm, below of=P'] {};
\node (Q'') [node distance=11mm, below of=P''] {};
\node (Q''') [node distance=9.5mm, below of=P'''] {};

\draw[->, bend right=22] (P) to node [swap] {} (Q);
\draw[->, bend right=18] (P') to node {} (Q');
\draw[->, bend left=18] (P'') to node {} (Q'');
\draw[->, bend left=22] (P''') to node {} (Q''');

\node (R) [node distance=0.55cm, below of=X''] {$\Rrightarrow$};
\node() [node distance=0.6cm, left of=R] {$\Rrightarrow$};
\node() [node distance=0.6cm, right of=R] {$\Rrightarrow$};

\draw[->, bend left=50] (B) to node {} (C);
\draw[->, bend right=50] (B) to node {} (C);

\node (E) [node distance=1.1cm, right of=B] {};
\node (E') [node distance=1.7cm, right of=B] {};
\node (E'') [node distance=2.3cm, right of=B] {};
\node (F) [node distance=0.7cm, above of=E] {};
\node (F') [node distance=0.8cm, above of=E'] {};
\node (F'') [node distance=0.7cm, above of=E''] {};
\node (G) [node distance=0.7cm, below of=E] {};
\node (G') [node distance=0.8cm, below of=E'] {};
\node (G'') [node distance=0.7cm, below of=E''] {};

\draw[->, bend right=25] (F) to node [swap] {} (G);
\draw[->] (F') to node {} (G');
\draw[->, bend left=25] (F'') to node {} (G'');

\node () [node distance=1.3cm, right of=B] {$\Rrightarrow$};
\node () [node distance=2.1cm, right of=B] {$\Rrightarrow$};

\draw[->, bend left=80] (C) to node {} (D);
\draw[->] (C) to node [swap] {} (D);
\draw[->, bend right=80] (C) to node [swap] {} (D);

\node (U') [node distance=1.05cm, right of=C] {};
\node (U'') [node distance=1.7cm, right of=C] {};
\node (U''') [node distance=2.35cm, right of=C] {};
\node (V') [node distance=10.5mm, above of=U'] {};
\node (V'') [node distance=11.5mm, above of=U''] {};
\node (V''') [node distance=10.5mm, above of=U'''] {};

\draw[->, bend right=20] (V') to node [swap] {} (U');
\draw[->] (V'') to node {} (U'');
\draw[->, bend left=20] (V''') to node {} (U''');

\node (W) [node distance=0.55cm, above of=U''] {};
\node() [node distance=0.38cm, left of=W] {$\Rrightarrow$};
\node() [node distance=0.38cm, right of=W] {$\Rrightarrow$};

\node (N) [node distance=14mm, right of=C] {};
\node (N') [node distance=20mm, right of=C] {};
\node (M) [node distance=1.1cm, below of=N] {};
\node (M') [node distance=1.1cm, below of=N'] {};

\draw[->, bend right=20] (N) to node {} (M);
\draw[->, bend left=20] (N') to node {} (M');

\node () [node distance=0.55cm, below of=U''] {$\Rrightarrow$};

\end{tikzpicture}
}
\end{center}
\end{example}

\begin{lemma}\label{degenerate is 0 under arity function} An $n$-pasting diagram $\sigma$ in $\bs{1}$ is degenerate in the sense of Definition \ref{degenerate pd} if and only if $|\sigma|_n = 0$.
\end{lemma}

Recall from Section \ref{GlobOpSubSection} that given a composite $n$-cell $\Lambda \circ (\Lambda_1,...,\Lambda_m)$ in a globular operad $\bs{G}$, each $\Lambda_i$ is a $k$-cell of $\bs{G}$ for some $k \leqslant n$ and $\Lambda$ is an $n$-cell. In particular, $(\Lambda_1,...,\Lambda_m)$ is an $n$-pasting diagram in $\bs{G}$ of shape $g(\Lambda)$, where $g:\bs{G} \longrightarrow \bs{1}^*$ is the underlying collection map. The $n$-pasting diagram
$$g(\Lambda \circ (\Lambda_1,...,\Lambda_m)) = g(\Lambda) \circ (g(\Lambda_1),...,g(\Lambda_m))$$
in $\bs{1}$ is obtained by replacing the individual cells in $g(\Lambda)$ with the $g(\Lambda_i)$s. We may also view this $n$-pasting diagram as a composite in the globular operad $\bs{T}$ for strict $\omega$-categories; see Example \ref{GlobOp T for strict w cats}.

\begin{lemma}\label{arity in compite cells in a GlobOp} Let $\bs{G}$ be a  globular operad with underlying collection map $g$, let $\Lambda \circ (\Lambda_1,...,\Lambda_m)$ be a composite $n$-cell in $\bs{G}$, and let $(\Delta_1,...,\Delta_d)$ be the tuple consisting of just those $\Lambda_i$s given by $n$-cells of $\bs{G}$, as opposed to $k$-cells for $k < n$. Then $\left|g(\Lambda)\right|_n = d$, and 
$$\big| g(\Lambda \circ (\Lambda_1,...,\Lambda_m)) \big|_n = \ \big| g(\Lambda_1) \big|_n + ... + \big| g(\Lambda_m) \big|_n = \ \begin{cases}
0 & d = 0, \\
\big| g(\Delta_1) \big|_n + ... + \big| g(\Delta_d) \big|_n & d \neq 0.
\end{cases}$$
\end{lemma}

\begin{proof}
For each $\Lambda_i$ not contained in the tuple $(\Delta_1,...,\Delta_d)$, $g(\Lambda_i)$ must be a degenerate $n$-pasting diagram in $\bs{1}$, so by Lemma \ref{degenerate is 0 under arity function} we have $\left|g(\Lambda_i)\right|_n =0$. 
\end{proof}

We now describe a canonical strict total order on the $n$-cells contained in an $n$-pasting diagram in $\bs{1}$. For this, we use the fact that every $n$-pasting diagram in $\bs{1}$ containing exactly two copies of the single $0$-cell of $\bs{1}$ defines an $(n{-}1)$-pasting diagram. This lower dimensional diagram is obtained by omitting the 0-cells and representing every $k$-cell as a $(k{-}1)$-cell for all $k > 0$. For instance, the 3-pasting diagram on the left below corresponds to the 2-pasting diagram on the right.
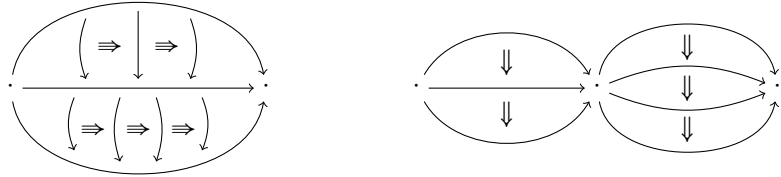
\begin{center}
\begin{tikzpicture}[node distance=3.4cm, auto]

\node (A) {$\cdot$};
\node (B) [right of=A] {$\cdot$};

\draw[->, bend left=80] (A) to node {} (B);
\draw[->] (A) to node [swap] {} (B);
\draw[->, bend right=80] (A) to node [swap] {} (B);

\node (X') [node distance=1.05cm, right of=A] {};
\node (X'') [node distance=1.7cm, right of=A] {};
\node (X''') [node distance=2.35cm, right of=A] {};
\node (Y') [node distance=10.5mm, above of=X'] {};
\node (Y'') [node distance=11.5mm, above of=X''] {};
\node (Y''') [node distance=10.5mm, above of=X'''] {};

\draw[->, bend right=20] (Y') to node [swap] {} (X');
\draw[->] (Y'') to node {} (X'');
\draw[->, bend left=20] (Y''') to node {} (X''');

\node (Z) [node distance=0.55cm, above of=X''] {};
\node() [node distance=0.38cm, left of=Z] {$\Rrightarrow$};
\node() [node distance=0.38cm, right of=Z] {$\Rrightarrow$};

\node (P) [node distance=9mm, right of=A] {};
\node (P') [node distance=15mm, right of=A] {};
\node (P'') [node distance=19mm, right of=A] {};
\node (P''') [node distance=25mm, right of=A] {};
\node (Q) [node distance=9.5mm, below of=P] {};
\node (Q') [node distance=11mm, below of=P'] {};
\node (Q'') [node distance=11mm, below of=P''] {};
\node (Q''') [node distance=9.5mm, below of=P'''] {};

\draw[->, bend right=22] (P) to node [swap] {} (Q);
\draw[->, bend right=18] (P') to node {} (Q');
\draw[->, bend left=18] (P'') to node {} (Q'');
\draw[->, bend left=22] (P''') to node {} (Q''');

\node (R) [node distance=0.55cm, below of=X''] {$\Rrightarrow$};
\node() [node distance=0.6cm, left of=R] {$\Rrightarrow$};
\node() [node distance=0.6cm, right of=R] {$\Rrightarrow$};

\node (a) [node distance=2cm, right of=B] {$\cdot$};
\node (b) [node distance=2.4cm, right of=a] {$\cdot$};
\node (c) [node distance=2.4cm, right of=b] {$\cdot$};

\draw[->, bend left=60] (a) to node {} (b);
\draw[->] (a) to node {} (b);
\draw[->, bend right=60] (a) to node [swap] {} (b);

\node (w) [node distance=1.2cm, right of=a] {};
\node () [node distance=3.5mm, above of=w] {$\Downarrow$};
\node () [node distance=3.5mm, below of=w] {$\Downarrow$};

\draw[->, bend left=80] (b) to node {} (c);
\draw[->, bend left=22] (b) to node {} (c);
\draw[->, bend right=22] (b) to node [swap] {} (c);
\draw[->, bend right=80] (b) to node [swap] {} (c);

\node (x) [node distance=1.2cm, right of=b] {$\Downarrow$};
\node () [node distance=5.5mm, above of=x] {$\Downarrow$};
\node () [node distance=5.5mm, below of=x] {$\Downarrow$};

\end{tikzpicture}
\end{center}

\begin{definition}\label{standard ordering} We define the \textit{standard order} on the copies of the unique $n$-cell of $\bs{1}$ contained in an $n$-pasting diagram in $\bs{1}$ by induction on $n$.
The standard order on the 1-cells contained in a 1-pasting diagram is given by numbering those cells from left to right, as illustrated below. 
\begin{center}
\begin{tikzpicture}[node distance=2.2cm, auto]

\node (A) {$\cdot$};
\node (B) [right of=A] {$\cdot$};
\node (C) [right of=B] {$\cdot$};
\node (D) [right of=C] {$\cdot$};
\node (E) [right of=D] {$\cdot$};

\draw[->] (A) to node {$1$} (B);
\draw[->] (B) to node {$2$} (C);
\draw[->] (C) to node {$3$} (D);
\draw[->] (D) to node {$4$} (E);

\end{tikzpicture}
\end{center}
Then for $n > 1$ and an $n$-pasting diagram $\pi$, we begin by considering the smaller $n$-pasting diagram contained between the two leftmost copies of the single 0-cell of $\bs{1}$.
By the reasoning above, this smaller $n$-pasting diagram defines a unique $(n{-}1)$-pasting diagram, so by induction we may number the cells accordingly.  Next, take the $n$-pasting diagram contained between the second and third copies of the single $0$-cell of $\bs{1}$ and continue the process, and so on until a natural number has been assigned to every $n$-cell contained in $\pi$.
\end{definition}

\begin{example} The standard order on the 3-cells contained in the 3-pasting diagram $\Psi$ used in Example \ref{arity function examples} is illustrated below.
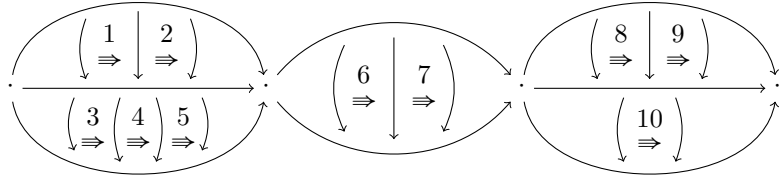
\begin{center}
\begin{tikzpicture}[node distance=3.4cm, auto]

\node (A) {$\cdot$};
\node (B) [right of=A] {$\cdot$};
\node (C) [right of=B] {$\cdot$};
\node (D) [right of=C] {$\cdot$};

\draw[->, bend left=80] (A) to node {} (B);
\draw[->] (A) to node [swap] {} (B);
\draw[->, bend right=80] (A) to node [swap] {} (B);

\node (X') [node distance=1.05cm, right of=A] {};
\node (X'') [node distance=1.7cm, right of=A] {};
\node (X''') [node distance=2.35cm, right of=A] {};
\node (Y') [node distance=10.5mm, above of=X'] {};
\node (Y'') [node distance=11.5mm, above of=X''] {};
\node (Y''') [node distance=10.5mm, above of=X'''] {};

\draw[->, bend right=20] (Y') to node [swap] {} (X');
\draw[->] (Y'') to node {} (X'');
\draw[->, bend left=20] (Y''') to node {} (X''');

\node (Z) [node distance=0.73cm, above of=X''] {};
\node() [node distance=0.38cm, left of=Z] {$1$};
\node() [node distance=0.38cm, right of=Z] {$2$};

\node (Z') [node distance=0.37cm, above of=X''] {};
\node() [node distance=0.38cm, left of=Z'] {$\Rrightarrow$};
\node() [node distance=0.38cm, right of=Z'] {$\Rrightarrow$};

\node (P) [node distance=9mm, right of=A] {};
\node (P') [node distance=15mm, right of=A] {};
\node (P'') [node distance=19mm, right of=A] {};
\node (P''') [node distance=25mm, right of=A] {};
\node (Q) [node distance=9.5mm, below of=P] {};
\node (Q') [node distance=11mm, below of=P'] {};
\node (Q'') [node distance=11mm, below of=P''] {};
\node (Q''') [node distance=9.5mm, below of=P'''] {};

\draw[->, bend right=22] (P) to node [swap] {} (Q);
\draw[->, bend right=18] (P') to node {} (Q');
\draw[->, bend left=18] (P'') to node {} (Q'');
\draw[->, bend left=22] (P''') to node {} (Q''');

\node (R) [node distance=0.38cm, below of=X''] {$4$};
\node() [node distance=0.6cm, left of=R] {$3$};
\node() [node distance=0.6cm, right of=R] {$5$};

\node (R') [node distance=0.72cm, below of=X''] {$\Rrightarrow$};
\node() [node distance=0.6cm, left of=R'] {$\Rrightarrow$};
\node() [node distance=0.6cm, right of=R'] {$\Rrightarrow$};

\draw[->, bend left=50] (B) to node {} (C);
\draw[->, bend right=50] (B) to node {} (C);

\node (E) [node distance=1.1cm, right of=B] {};
\node (E') [node distance=1.7cm, right of=B] {};
\node (E'') [node distance=2.3cm, right of=B] {};
\node (F) [node distance=0.7cm, above of=E] {};
\node (F') [node distance=0.8cm, above of=E'] {};
\node (F'') [node distance=0.7cm, above of=E''] {};
\node (G) [node distance=0.7cm, below of=E] {};
\node (G') [node distance=0.8cm, below of=E'] {};
\node (G'') [node distance=0.7cm, below of=E''] {};

\draw[->, bend right=25] (F) to node [swap] {} (G);
\draw[->] (F') to node {} (G');
\draw[->, bend left=25] (F'') to node {} (G'');

\node (x) [node distance=1.3cm, right of=B] {};
\node (y) [node distance=2.1cm, right of=B] {};

\node () [node distance=0.18cm, above of=x] {$6$};
\node () [node distance=0.18cm, below of=x] {$\Rrightarrow$};

\node () [node distance=0.18cm, above of=y] {$7$};
\node () [node distance=0.18cm, below of=y] {$\Rrightarrow$};

\draw[->, bend left=80] (C) to node {} (D);
\draw[->] (C) to node [swap] {} (D);
\draw[->, bend right=80] (C) to node [swap] {} (D);

\node (U') [node distance=1.05cm, right of=C] {};
\node (U'') [node distance=1.7cm, right of=C] {};
\node (U''') [node distance=2.35cm, right of=C] {};
\node (V') [node distance=10.5mm, above of=U'] {};
\node (V'') [node distance=11.5mm, above of=U''] {};
\node (V''') [node distance=10.5mm, above of=U'''] {};

\draw[->, bend right=20] (V') to node [swap] {} (U');
\draw[->] (V'') to node {} (U'');
\draw[->, bend left=20] (V''') to node {} (U''');

\node (W) [node distance=0.73cm, above of=U''] {};
\node() [node distance=0.38cm, left of=W] {$8$};
\node() [node distance=0.38cm, right of=W] {$9$};

\node (W') [node distance=0.37cm, above of=U''] {};
\node() [node distance=0.38cm, left of=W'] {$\Rrightarrow$};
\node() [node distance=0.38cm, right of=W'] {$\Rrightarrow$};

\node (N) [node distance=14mm, right of=C] {};
\node (N') [node distance=20mm, right of=C] {};
\node (M) [node distance=1.1cm, below of=N] {};
\node (M') [node distance=1.1cm, below of=N'] {};

\draw[->, bend right=20] (N) to node {} (M);
\draw[->, bend left=20] (N') to node {} (M');

\node (R) [node distance=0.4cm, below of=U''] {$10$};
\node (R') [node distance=0.74cm, below of=U''] {$\Rrightarrow$};

\end{tikzpicture}
\end{center}
\end{example}

\begin{remark}\label{order on general pds}
We have been implicitly using the standard order to denote pasting diagrams in globular sets. To see this, consider that an $n$-pasting diagram in a globular set $\bs{G}$ is given by the concatenation of $k$-cells of $\bs{G}$ along matching boundaries for all $k \leqslant n$. For example, a 2-pasting diagram 
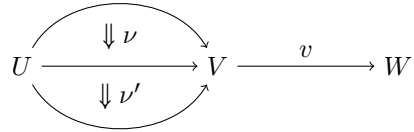
\begin{center}
\begin{tikzpicture}[node distance=2.6cm, auto]

\node (A) {$U$};
\node (B) [right of=A] {$V$};
\node (C) [node distance=2.4cm, right of=B] {$W$};

\draw[->, bend left=60] (A) to node {} (B);
\draw[->] (A) to node {} (B);
\draw[->, bend right=60] (A) to node [swap] {} (B);

\node (W) [node distance=1.3cm, right of=A] {};
\node () [node distance=0.4cm, above of=W] {$\Downarrow \nu$};
\node () [node distance=0.4cm, below of=W] {$\Downarrow \nu'$};

\draw[->] (B) to node {$v$} (C);

\end{tikzpicture}
\end{center}
arises from concatenating the 1-cell $v$ and the 2-cells $\nu$ and $\nu'$ along matching 0 and 1-cell boundaries. When we want to refer to a specific $n$-pasting diagram without having to draw it, we write a list of the $k$-cells making up the diagram. In the case of the 2-pasting above, for instance, we write $(\nu, \nu', v)$. 

However, in order to minimise ambiguity, we need to chose a consistent \textit{order} in which to write these cells. 
We begin by representing each of the $k$-cells in our list as an $n$-cell. For our example, this would give us the diagram on the left below.
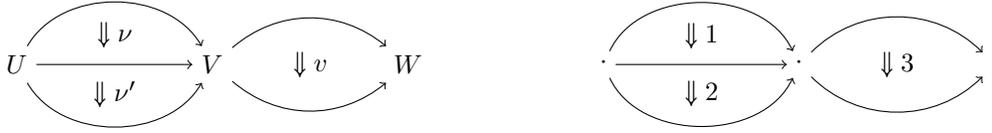
\begin{center}
\begin{tikzpicture}[node distance=2.6cm, auto]

\node (A) {$U$};
\node (B) [right of=A] {$V$};
\node (C) [right of=B] {$W$};

\draw[->, bend left=60] (A) to node {} (B);
\draw[->] (A) to node {} (B);
\draw[->, bend right=60] (A) to node [swap] {} (B);

\node (W) [node distance=1.3cm, right of=A] {};
\node () [node distance=0.4cm, above of=W] {$\Downarrow \nu$};
\node () [node distance=0.4cm, below of=W] {$\Downarrow \nu'$};

\draw[->, bend left=40] (B) to node {} (C);
\draw[->, bend right=40] (B) to node [swap] {} (C);

\node () [node distance=1.3cm, right of=B] {$\Downarrow v$};

\node (A') [right of=C] {$\cdot$};
\node (B') [ right of=A'] {$\cdot$};
\node (C') [right of=B'] {$\cdot$};

\draw[->, bend left=65] (A') to node {} (B');
\draw[->] (A') to node {} (B');
\draw[->, bend right=65] (A') to node [swap] {} (B');

\node (W') [node distance=1.3cm, right of=A'] {};
\node () [node distance=0.38cm, above of=W'] {$\Downarrow 1$};
\node () [node distance=0.38cm, below of=W'] {$\Downarrow 2$};

\draw[->, bend left=45] (B') to node {} (C');
\draw[->, bend right=45] (B') to node [swap] {} (C');

\node () [node distance=1.3cm, right of=B'] {$\Downarrow 3$};

\end{tikzpicture}
\end{center}
We then arrange our list of cells according to the standard order on the $n$-cells contained in the unique $n$-pasting diagram in $\bs{1}$ of the same shape. In our example, this is the order shown in the diagram on the right above, so we denote our original 2-pasting diagram by $(\nu, \nu', v)$, rather than $(\nu', \nu, v)$ or $(v, \nu, \nu')$, etc.
\end{remark}

Proposition \ref{non-standard ordering} describes another total order on the cells contained in pasting diagrams in $\bs{1}$ that, along with the standard order, will be essential for our construction of slices. We first require the following definition. 

\begin{definition} A presentation for a globular operad is a \textit{$k$-free presentation} if it contains no $n$-cell generators for any $n > k$, and no $n$-cell relations for any $n > k-1$. A globular operad is  \textit{$k$-free} if it has a $k$-free presentation.
\end{definition}

\begin{remarks} \label{k-free remarks}
\leavevmode
\begin{enumerate}
\item For $n>k$, the only $n$-cell of a $k$-free globular operad is the identity $n$-cell $\text{id}_n$.
\item Every presentation $P$ for a globular operad $\bs{G}$ contains a $k$-free presentation $\tilde{P}_k$, given by taking the $l$-cell generators and $(l-1)$-cell relations of $P$ for all $l \leqslant k$, and omitting everything else. The $k$-free operad $\bs{J_k}$ defined by $\tilde{P}_k$ is equal to $\bs{G}$ when truncated to $(k-1)$-dimensions, and has $k$-cells given by the free operadic composites of the $k$-cell generators.
\item Given a presentation $P$ for a globular operad $\bs{G}$, each $k$-cell relation in $P$ corresponds to an equality between a pair of parallel $k$-cells in the induced $k$-free operad $\bs{J_k}$. In particular, the set $R_k$ of $k$-cell relations is equipped with a pair of functions
\begin{center}
\begin{tikzpicture}[node distance=3.2cm, auto]

\node (A) {$R_k$};
\node (B) [right of=A] {$\mathbf{GSet}(\bs{B_k}, \bs{J_k})$};

\draw[transform canvas={yshift=0.75ex},->] (A) to node {$e_k$} (B);
\draw[transform canvas={yshift=-0.25ex},->] (A) to node [swap] {$q_k$} (B);

\end{tikzpicture}
\end{center}
into the $k$-cells of $\bs{J_k}$ such that for each $y \in R_k$, $e_k(y)$ and $q_k(y)$ are parallel $k$-cells of $\bs{J_k}$ satisfying $j(e_k(y)) = j(q_k(y))$, where $j:\bs{J_k} \longrightarrow \bs{1}^*$ denotes the underlying collection map of $\bs{J_k}$; see \cite{Me} for details. 
\end{enumerate}
\end{remarks}

\begin{example}\label{Strict J_2} Let $P$ be the presentation for the 3-globular operad $\bs{T_3}$ for strict 3-categories constructed in Example \ref{pres for T_3}. The 2-free globular operad $\bs{J_2}$ induced by $P$ has the same 0 and 1-cells as $\bs{T_3}$, and has 2-cells given by the free operadic composites of the 2-cell generators $i_2$, $h_2$ and $v_2$. The images of these 2-cells under the underlying collection map $j:\bs{J_2} \longrightarrow \bs{1}^*$ are illustrated below.
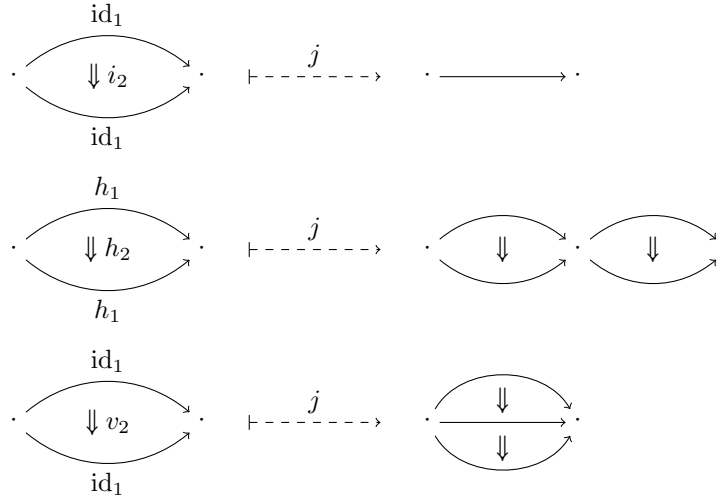
\begin{center}
\begin{tikzpicture}[node distance=2cm, auto]

\node (A) {$\cdot$};
\node (B) [node distance=2.5cm, right of =A] {$\cdot$};

\draw[->, bend left=40] (A) to node {$\text{id}_1$} (B);
\draw[->, bend right=40] (A) to node [swap] {$\text{id}_1$} (B);

\node () [node distance=1.25cm, right of =A] {$\Downarrow i_2$};

\node (P) [node distance=3cm, right of =B] {$\cdot$};
\node (Q) [right of =P] {$\cdot$};

\draw[->] (P) to node {} (Q);

\node (b) [node distance=0.5cm, right of=B] {};
\node (p) [node distance=0.5cm, left of=P] {};
\draw[|->, dashed] (b) to node {$j$} (p);

\node (A') [node distance=2.3cm, below of=A] {$\cdot$};
\node (B') [node distance=2.5cm, right of =A'] {$\cdot$};

\draw[->, bend left=40] (A') to node {$h_1$} (B');
\draw[->, bend right=40] (A') to node [swap] {$h_1$} (B');

\node () [node distance=1.25cm, right of =A'] {$\Downarrow h_2$};

\node (P') [node distance=3cm, right of =B'] {$\cdot$};
\node (Q') [right of =P'] {$\cdot$};
\node (R') [right of =Q'] {$\cdot$};

\draw[->, bend left=40] (P') to node {} (Q');
\draw[->, bend right=40] (P') to node {} (Q');
\draw[->, bend left=40] (Q') to node {} (R');
\draw[->, bend right=40] (Q') to node {} (R');

\node () [node distance=1cm, right of =P'] {$\Downarrow$};
\node () [node distance=1cm, right of =Q'] {$\Downarrow$};

\node (b') [node distance=0.5cm, right of=B'] {};
\node (p') [node distance=0.5cm, left of=P'] {};
\draw[|->, dashed] (b') to node {$j$} (p');

\node (A'') [node distance=2.3cm, below of=A'] {$\cdot$};
\node (B'') [node distance=2.5cm, right of =A''] {$\cdot$};

\draw[->, bend left=40] (A'') to node {$\text{id}_1$} (B'');
\draw[->, bend right=40] (A'') to node [swap] {$\text{id}_1$} (B'');

\node () [node distance=1.25cm, right of =A''] {$\Downarrow v_2$};

\node (P'') [node distance=3cm, right of =B''] {$\cdot$};
\node (Q'') [right of =P''] {$\cdot$};

\draw[->, bend left=60] (P'') to node {} (Q'');
\draw[->] (P'') to node {} (Q'');
\draw[->, bend right=60] (P'') to node {} (Q'');

\node (x) [node distance=1cm, right of =P''] {};
\node () [node distance=0.3cm, above of =x] {$\Downarrow$};
\node () [node distance=0.35cm, below of =x] {$\Downarrow$};

\node (b'') [node distance=0.5cm, right of=B''] {};
\node (p'') [node distance=0.5cm, left of=P''] {};
\draw[|->, dashed] (b'') to node {$j$} (p'');

\end{tikzpicture}
\end{center}
An algebra for $\bs{J_2}$ on a globular set $\bs{A}$ equips the 0 and 1-cells of $\bs{A}$ with the structure of a category, picks out a specified 2-cell $a \longrightarrow a$ for each 1-cell $a$ of $\bs{A}$, and provides binary horizontal and vertical composition operations on the the 2-cells of $\bs{A}$. However, composition of 2-cell satisfies no axioms.
\end{example}

\begin{remark}\label{The k-cells of J_k} 
Let $\bs{J_k}$ be a $k$-free globular operad equipped with a $k$-free presentation, so each $k$-cell of $\bs{J_k}$ is a unique free operadic composite of the $k$-cell generators. Using the associativity and identity axioms for operads, we can show by induction that every non-identity $k$-cell of $\bs{J_k}$ may be expressed in the form
$$Q \circ (\Lambda_1,...,\Lambda_m)$$
where $Q$ is some smaller operadic composite of the $k$-cell generators, each $\Lambda_i$ is either an $l$-cell generator of $P$ or an identity $l$-cell for some $l \leqslant k$, and $\Lambda_i$ is a non-identity cell for at least one $i$. Note that this is analogous to how we characterised the elements of a free plain operad in Example \ref{free plain operads}. 
\end{remark}

\begin{example}\label{rewriting cells of J_k}
The 2-cells $i_2$, $h_2$ and $v_2$ of the 2-free globular operad $\bs{J_2}$ presented in Example \ref{Strict J_2} may be written as $\text{id}_2 \circ (i_2)$, $\text{id}_2 \circ (h_2)$, and $\text{id}_2 \circ (v_2)$, respectively.
Meanwhile, we may use the identity and associativity axioms to rewrite the 2-cell $h_2 \circ (v_2, \, i_2 \circ (i_1))$ of $\bs{J_2}$ as
$$h_2 \circ (v_2, \, i_2 \circ (i_1)) = h_2 \circ (\text{id} \circ (v_2), \, i_2 \circ (i_1)) = (h_2 \circ (\text{id}_2, i_2)) \circ (v_2, i_1),$$
or equivalently, as $(h_2 \circ (v_2, i_2)) \circ (\text{id}_2, \text{id}_2, i_1)$.
\end{example}

\begin{proposition}\label{non-standard ordering} 
Let $\bs{J_k}$ be a $k$-free globular operad equipped with a $k$-free presentation, and let $j:\bs{J_k} \longrightarrow \bs{1}^*$ be the underlying collection map. Each $k$-cell $x$ of $\bs{J_k}$ induces a strict total order on the copies of the unique $k$-cell of $\bs{1}$ contained in $j(x)$, called the \textit{order induced by $x$}, defined as follows.
If $x=\text{id}_k$ then $j(\text{id}_k)$ is the $k$-pasting diagram in $\bs{1}$ made up of a single $k$-cell, so there is only one possible order. 
If $x$ is not the identity $k$-cell, first write $x = Q \circ (\Lambda_1,...,\Lambda_m)$ where $Q$ and the $\Lambda_i$s are as in Remark \ref{The k-cells of J_k}. Then,
$$j(x) = j(Q \circ (\Lambda_1,...,\Lambda_m)) = j(Q) \circ (j(\Lambda_1),...,j(\Lambda_m))$$
is the $k$-pasting diagram in $\bs{1}$ obtained by replacing the individual cells of $j(Q)$ with the $j(\Lambda_i)$s.
We use the standard order on the copies of the unique $k$-cell of $\bs{1}$ contained in $j(\Lambda_1)$, then continue with the $k$-cells contained in $j(\Lambda_2)$, and so on until we have assigned a natural number to each copy of the unique $k$-cell of $\bs{1}$ contained in the whole of $j(x)$.
\end{proposition}

\begin{example}\label{non-standard ordering examples} Let $\bs{J_2}$ be the 2-free globular operad given in Example \ref{Strict J_2}. The order induced by the 2-cell $h_2 \circ (v_2, v_2)$ of $\bs{J_2}$ on the 2-cells of $\bs{1}$ contained in $j(h_2 \circ (v_2, v_2)) = j(h_2) \circ (j(v_2), j(v_2))$ is the standard order:
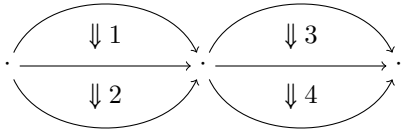
\begin{center}
\begin{tikzpicture}[node distance=2.6cm, auto]

\node (A) {$\cdot$};
\node (B) [right of=A] {$\cdot$};
\node (C) [right of=B] {$\cdot$};

\draw[->, bend left=65] (A) to node {} (B);
\draw[->] (A) to node {} (B);
\draw[->, bend right=65] (A) to node [swap] {} (B);

\node (X) [node distance=1.3cm, right of=A] {};
\node () [node distance=0.4cm, above of=X] {$\Downarrow 1$};
\node () [node distance=0.4cm, below of=X] {$\Downarrow 2$};

\draw[->, bend left=65] (B) to node {} (C);
\draw[->] (B) to node {} (C);
\draw[->, bend right=65] (B) to node [swap] {} (C);

\node (W) [node distance=1.3cm, right of=B] {};
\node () [node distance=0.4cm, above of=W] {$\Downarrow 3$};
\node () [node distance=0.4cm, below of=W] {$\Downarrow 4$};

\end{tikzpicture}
\end{center}
This is obtained by first applying the standard order to the 2-cells of $\bs{1}$ contained in the first copy of $j(v_2)$, and then continuing the process with the 2-cells in second copy of $j(v_2)$.
However, the order
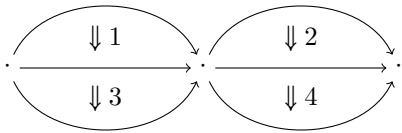
\begin{center}
\begin{tikzpicture}[node distance=2.6cm, auto]

\node (A) {$\cdot$};
\node (B) [right of=A] {$\cdot$};
\node (C) [right of=B] {$\cdot$};

\draw[->, bend left=65] (A) to node {} (B);
\draw[->] (A) to node {} (B);
\draw[->, bend right=65] (A) to node [swap] {} (B);

\node (X) [node distance=1.3cm, right of=A] {};
\node () [node distance=0.4cm, above of=X] {$\Downarrow 1$};
\node () [node distance=0.4cm, below of=X] {$\Downarrow 3$};

\draw[->, bend left=65] (B) to node {} (C);
\draw[->] (B) to node {} (C);
\draw[->, bend right=65] (B) to node [swap] {} (C);

\node (W) [node distance=1.3cm, right of=B] {};
\node () [node distance=0.4cm, above of=W] {$\Downarrow 2$};
\node () [node distance=0.4cm, below of=W] {$\Downarrow 4$};

\end{tikzpicture}
\end{center}
induced by $v_2 \circ (h_2,h_2)$, obtained by first applying the standard order to the 2-cells of $\bs{1}$ in the first copy of $j(h_2)$, and then continuing with the 2-cells in second copy of $j(h_2)$, is \textit{not} the standard order. 
\end{example}

\begin{proof}[proof of Propsition \ref{non-standard ordering examples}]
We need to show that the total order induced by $x$ does not depend on the choice of $Q$ or $\Lambda_i$s. Let $x=Q \circ (\Lambda_1,...,\Lambda_m) = Q' \circ (\Lambda'_1,...,\Lambda'_{m'})$ where $Q$, $\Lambda_i$, $Q'$ and $\Lambda'_i$ are as in Remark \ref{The k-cells of J_k}. Then 
$$Q \circ (\Lambda_1,...,\Lambda_m) 
= ( Q \circ (\Lambda_1,...,\Lambda_m)) \circ (\text{id}_{11},...,\text{id}_{1k_1},...,\text{id}_{m1},...,\text{id}_{mk_m}),$$
and the order induced on $j(x)$ by the expression on the left is clearly the same as the one induced by the expression on the right. The same is true for 
$$ ( Q \circ (\Lambda_1,...,\Lambda_m) ) \circ (\text{id}_{11},...,\text{id}_{1k_1},...,\text{id}_{m1},...,\text{id}_{mk_m}) = ( Q' \circ (\Lambda'_1,...,\Lambda'_{m'}) ) \circ (\text{id}_{11},...\text{id}_{1k_1},...,\text{id}_{m1},...,\text{id}_{mk_m}),$$
and finally, for 
$$\qquad \qquad \qquad ( Q' \circ (\Lambda'_1,...,\Lambda'_{m'}) ) \circ (\text{id}_{11},...\text{id}_{1k_1},...,\text{id}_{m1},...,\text{id}_{mk_m}) = Q' \circ (\Lambda'_1,...,\Lambda'_{m'}). \qquad \qquad \qquad \qedhere $$ 
\end{proof}

We now use Proposition \ref{non-standard ordering} together with the function $|-|:\mathbf{GSet}(\bs{B_k}, \bs{1}^*) \longrightarrow\mathbb{N}$ given in Definition \ref{arity function} to assign an element of a symmetric group to each $k$-cell of a $k$-free globular operad.

\begin{definition}\label{def of t_x} Let $\bs{J_k}$ be a $k$-free globular operad equipped with a $k$-free presentation, and let $j:\bs{J_k} \longrightarrow \bs{1}^*$ be the underlying collection map. 
The \textit{element $t_x$ of the symmetric group} $S_{|j(x)|}$
is the element that, when applied to the standard order on the copies of the unique $k$-cell of $\bs{1}$ contained in $j(x)$, produces the order induced by $x$.
\end{definition}

\begin{example}\label{examples of t_x} Once again, let $\bs{J_2}$ be the 2-free globular operad discussed in Example \ref{Strict J_2}. From Example \ref{non-standard ordering examples}, we see that 
\begin{itemize}
\item if $x = h_2 \circ (v_2, v_2)$ then $t_x = 0 \in S_4$, since the standard order and the order induced by $x$ agree in this case; and
\item if $x = v_2 \circ (h_2, h_2)$  then $t_x = (2 \ 3) \in S_4$, since in this case the order induced by $x$ on the copies of the unique 2-cell of $\bs{1}$ contained in $j(x)$ is the result of permuting the standard order using the element $(2 \ 3)$ of the symmetric group $S_4$.
\end{itemize}
\end{example}

\begin{definition}\label{the collection J_n} Let $\bs{J_k}$ be a $k$-free globular operad equipped with a $k$-free presentation, let $J_k$ be the set of $k$-cell generators in this presentation, and let $j:\bs{J_k} \longrightarrow \bs{1}^*$ be the underlying collection map. The \textit{plain collection $|j(-)|:J_k \longrightarrow \mathbb{N}$ associated to $J_k$} is the composite 
\begin{center}
\begin{tikzpicture}[node distance=2.8cm, auto]

\node (A) {$J_k$};
\node (B) [right of=A] {$\mathbf{GSet}(\bs{B_k}, \bs{J_k})$};
\node (C) [node distance=3.8cm, right of=B] {$\mathbf{GSet}(\bs{B_k}, \bs{1}^*)$};
\node (D) [right of=C] {$\mathbb{N}$};

\draw[right hook->] (A) to node {} (B);
\draw[->] (B) to node {} (C);
\draw[->] (C) to node {} (D);

\end{tikzpicture}
\end{center}
where the first map is the inclusion of the $k$-cell generators into the $k$-cells of $\bs{J_k}$, the second function is induced by $j:\bs{J_k} \longrightarrow \bs{1}^*$, and the third is as in Definition \ref{arity function}. 
\end{definition}

\begin{example}\label{corresponding collection for strict J_2} Let $\bs{J_2}$ be the 2-free globular operad given in Example \ref{Strict J_2}. Then $J_2 = \{ i_2, h_2, v_2 \}$ and the collection $|j(-)|:J_2 \longrightarrow \mathbb{N}$ is given by $|j(i_2)| =0$ and $|j(h_2)| = |j(v_2)| = 2$.
\end{example}

We are now ready to form a direct relationship between globular operads and symmetric operads using the language of presentations. First, recall the free-forgetful adjunction $F \dashv U: \mathbf{SymmOp} \longrightarrow \mathbf{Set}/\mathbb{N}$; see Definition \ref{free symmetric operad adjunction}. Given a plain collection $J_k \longrightarrow \mathbb{N}$, the collection $UF(J_k)$ is the function
\begin{center}
\begin{tikzpicture}[node distance=3cm, auto]

\node (A) {};
\node (B) [right of=A] {$\mathbb{N}$};
\node (a) [node distance=0.15cm, below of=A] {$\underset{n \geqslant 0}{\mathlarger{\sum}} S_n \times J_k^{\infty}[n]$};
\node (a'') [node distance=1.2cm, right of=A] {};

\draw[->] (a'') to node {} (B);

\end{tikzpicture}
\end{center}
for which the preimage of each $n \in \mathbb{N}$ is the set $S_n \times J^{\infty}_k[n]$, and $J_k^{\infty}$ denotes the free \textit{plain} operad on $J_k \longrightarrow \mathbb{N}$; see Example \ref{free plain operads}. Next, observe that taking only the last two maps in the composite in Definition \ref{the collection J_n} defines a collection $|j(-)|:\mathbf{GSet}(\bs{B_k}, \bs{J_k}) \longrightarrow \mathbb{N}$. Finally, recall that we denote the identity element of the symmetric group $S_n$ by 0.

\begin{proposition}\label{map to the corresponding symm op}  
Let $\bs{J_k}$ be a $k$-free globular operad equipped with a $k$-free presentation, let $J_k$ be the set of $k$-cell generators in this presentation, and let $j:\bs{J_k} \longrightarrow \bs{1}^*$ be the underlying collection map. There is a morphism 
\begin{center}
\begin{tikzpicture}[node distance=3.5cm, auto]

\node (A) {$\mathbf{GSet}(\bs{B_k}, \bs{J_k})$};
\node (B) [right of=A] {$UF(J_k)$};
\node (C) [node distance=2cm, right of=B] {};
\node () [node distance=1.5mm, below of=C] {$= \underset{n \geqslant 0}{\mathlarger{\sum}} S_n \times J_k^{\infty}[n]$};

\draw[->] (A) to node {$f$} (B);

\end{tikzpicture}
\end{center}
in the category $\mathbf{Set}/\mathbb{N}$ of plain collections where $f(\text{id}_k) = (0, \text{id})$, $f(\Lambda) = (0, \Lambda)$ for any $k$-cell generator $\Lambda \in J_k$, and 
for a $k$-cell of $\bs{J_k}$ of the form $x = Q \circ (\Lambda_1,..., \Lambda_m)$ where $Q$ and the $\Lambda_i$s are as in Remark \ref{The k-cells of J_k}, and $Q \neq \text{id}_k$,
$$f(Q \circ (\Lambda_1,..., \Lambda_m)) =  ( t_x, \Lambda_Q \circ (\Delta_1,..,\Delta_d))$$ 
where $t_x \in S_{|j(x)|}$ is as in Definition \ref{def of t_x}, $(\Delta_1,..,\Delta_d)$ is the tuple consisting of just those $\Lambda_i$s given by a $k$-cell generator or the identity $k$-cell $\text{id}_k$, and $f(Q) = (t_Q, \Lambda_Q)$ is defined by induction. 
\end{proposition}

\begin{proof}
First, we need to show that the function $f$ is a map of plain collections, that is, we need to show that $f$ sends a $k$-cell $x$ of $\bs{J_k}$ for which $|j(x)| = n$ to an element of the set $S_n \times J_k^{\infty}[n]$. This follows from the fact that
$$\Lambda_Q \circ (\Delta_1,...,\Delta_d) \in J_k^{\infty} \big[ | j(\Delta_1)| +...+|j(\Delta_d)| \big]$$
and, by Lemma \ref{arity in compite cells in a GlobOp},
$$\big| j(Q \circ (\Lambda_1,..., \Lambda_m)) \big| = | j(\Lambda_1) |+...+|j(\Lambda_m)| = \big| j(\Delta_1)\big| +...+\big|j(\Delta_d)\big|.$$
Next, we need to show that $\Lambda_Q \circ (\Delta_1,...,\Delta_d)$ is a composite in $J_k^{\infty}$, which is true if $\Lambda_Q \in J_k^{\infty} [d]$. Again, by Lemma \ref{arity in compite cells in a GlobOp}, we know that $|j(Q)| = d$, so since $f(Q) = (t_Q, \Lambda_Q)$, it follows from the induction hypothesis that $\Lambda_Q \in J_k^{\infty} [d]$. The fact $f$ is well-defined, i.e., that $f(x)$ does not depend on the choice of $Q$ or the $\Lambda_i$s, can be shown using a similar argument as in the proof of Proposition \ref{non-standard ordering}.
\end{proof}

Given a $k$-free globular operad $\bs{J_k}$, the morphism defined in Proposition \ref{map to the corresponding symm op} forgets the lower-dimensional data of the $k$-cells of $\bs{J_k}$ while still retaining some important structural information. The example below illustrates how this works in practice. 
First, recall that a pointed double magma is a set equipped with specified element, or a `point', and two binary operations, satisfying no axioms. 

\begin{example}\label{example of f for strict J_2} Let $\bs{J_2}$ be the 2-free globular operad given in Example \ref{Strict J_2}, so $J_2 = \{ i_2, h_2, v_2 \}$ and the collection $|j(-)|:J_2 \longrightarrow \mathbb{N}$ is given by $|j(i_2)| =0$ and $|j(h_2)| = |j(v_2)| = 2$. The free symmetric operad $F(J_2)$ consists of the operadic composites of elements $i_2 \in F(J_2)[0]$ and $h_2, v_2 \in F(J_2)[2]$, subject to no equations. An algebra for $F(J_2)$ on a set $X$ is then a pointed double magma with underlying set $X$; the binary operations are provided by $h_2$ and $v_2$, while $i_2$ picks out a point in $X$. 

The function 
\begin{center}
\begin{tikzpicture}[node distance=3.5cm, auto]

\node (A) {$\mathbf{GSet}(\bs{B_2}, \bs{J_2})$};
\node (B) [right of=A] {$UF(J_2)$};
\node (C) [node distance=2cm, right of=B] {};
\node () [node distance=1.5mm, below of=C] {$= \underset{n \geqslant 0}{\mathlarger{\sum}} S_n \times J_2^{\infty}[n]$};

\draw[->] (A) to node {$f$} (B);

\end{tikzpicture}
\end{center}
defined in Proposition \ref{map to the corresponding symm op} sends $\text{id}_2$ to $(0, \text{id}) \in S_1 \times J_2^{\infty}[1]$, and is defined on the 2-cell generators by $f(i_2) = (0,i_2) \in S_0 \times J_2^{\infty}[0]$, $f(h_2) = (0,h_2) \in S_2 \times J_2^{\infty}[2]$, and $f(v_2) = (0,v_2) \in S_2 \times J_2^{\infty}[2]$.
Of more interest are the values of $f$ at composite 2-cells of $\bs{J_2}$. From Example \ref{examples of t_x}, we see that
$$f(h_2 \circ (v_2, v_2)) = ( 0, \, h_2 \circ (v_2, v_2) ) \in S_4 \times J_2^{\infty}[4]$$
and
$$f(v_2 \circ (h_2, h_2)) = ( (2 \ 3), \, v_2 \circ (h_2, h_2)) \in S_4 \times J_2^{\infty}[4].$$
Informally, given a composite 2-cell $x$ in $\bs{J_2}$, $f$ records the value of $t_x$ and forgets the components of $x$ not coming from the identity 2-cell $\text{id}_2$ or one of the 2-cell generators $i_2$, $h_2$ and $v_2$. This is illustrated by the example below,
\begin{align*}
f \big( ( (h_2 \circ (v_2, i_2)) \circ (\text{id}_2, \text{id}_2, h_1) ) \circ (h_2, h_2, i_1, i_1) \big) &= \big( (2 \ 3), \, ( (h_2 \circ (v_2, i_2)) \circ (\text{id}_2, \text{id}_2) ) \circ (h_2, h_2)\big)\\
&= \left( (2 \ 3), \, (h_2 \circ (v_2, i_2)) \circ (h_2, h_2) \right) \\
&= \left( (2 \ 3), \, h_2 \circ (v_2 \circ  (h_2, h_2), i_2) \right) \in S_4 \times J_2^{\infty}[4].
\end{align*}
\end{example}

We are now ready to define, for any globular operad $\bs{G}$, an associated symmetric operad $S_k$. 

\begin{definition}
Let $P$ be a presentation for a globular operad $\bs{G}$, let $\tilde{P}_k$ be the $k$-free presentation contained in $P$, and let $\bs{J_k}$ be the $k$-free globular operad defined by $\tilde{P}_k$, as in part (2) of Remarks \ref{k-free remarks}. The $k^{th}$ \textit{slice of} $\bs{G}$ is the symmetric operad $S_k$ with presentation $(J_k,R_k,\tilde{e}_k,\tilde{q}_k)$ where
\begin{enumerate}[$\bullet$]

\item $J_k$ and $R_k$ are the sets of $k$-cell generators and $k$-cell relations in $P$, respectively;

\item The collection $|-|:J_k \longrightarrow \mathbb{N}$ is as in Definition \ref{the collection J_n};

\item The maps $\tilde{e}_k$ and $\tilde{q}_k$ are the composites
\begin{center}
\begin{tikzpicture}[node distance=3.5cm, auto]

\node (A) {$R_k$};
\node (B) [right of=A] {$\mathbf{GSet}(\bs{B_k}, \bs{J_k})$};
\node (C) [right of=B] {$UF(J_k)$};

\draw[transform canvas={yshift=0.75ex},->] (A) to node {$e_k$} (B);
\draw[transform canvas={yshift=-0.25ex},->] (A) to node [swap] {$q_k$} (B);
\draw[->] (B) to node {$f$} (C);

\end{tikzpicture}
\end{center}
where $e_k$ and $q_k$ are as in part (3) of Remarks \ref{k-free remarks} and $f$ is defined as in Proposition \ref{map to the corresponding symm op}. Since $j(e_k(y)) = j(q_k(y))$ for any $y \in R_k$, composing either $e_k$ or $q_k$ with $|j(-)|:\mathbf{GSet}(\bs{B_k}, \bs{J_k}) \longrightarrow \mathbb{N}$ defines a collection $R_k \longrightarrow \mathbb{N}$, and the maps $\tilde{e}_k$ and $\tilde{q}_k$ are maps of plain collections.
\end{enumerate}
\end{definition}

Although the definition of slices is fairly technical, it is intuitive to apply in practice. As a first example, we calculate the slices of the 3-globular operad for strict 3-categories. Recall that the $k^{th}$ slice of a globular operad for some theory of higher category should be the symmetric operad for the kind of algebraic structure formed by the $k$-cells of those higher categories whenever they contain only one $l$-cell for all $l < k$. In the case of strict higher categories, we would therefore expect the first slice to be the plain operad for monoids, and all higher slices to be the symmetric operad for commutative monoids. 

Before proceeding, note that the set of 0-cells of a higher category simply form a set, rather than a monoid or some other more complex structure. As such, we have the following lemma.

\begin{lemma} The $0^{th}$ slice of a globular operad for higher categories is the plain operad for sets.
\end{lemma}

\begin{proof}
A globular operad $\bs{G}$ for higher categories has exactly one 0-cell, the identity $\text{id}_0$; see Notation \ref{0-cells in higher cats} and \cite[Definition 8.9]{Me}. It follows that there exists a presentation for $\bs{G}$ whose sets $J_0$ and $R_0$ of 0-cell generators and relations, respectively, are the empty sets. The $0^{th}$ slice of $\bs{G}$ is then the symmetric operad with presentation $(\emptyset, \emptyset, !, !)$. From Example \ref{Presentation for I}, we see that this is just the symmetric operad for sets. 
\end{proof}

\begin{example}\label{slices of T_3}
We calculate the first and second slices of the 3-globular operad $\bs{T_3}$ for strict 3-categories using the presentation given in Example \ref{pres for T_3}. From there, we see that the first slice of $\bs{T_3}$ is the symmetric operad $T$ consisting of the operadic composites of elements $i_1 \in T[0]$ and $h_1 \in T[2]$ subject to the following equalities.
\vspace{1mm}
\begin{enumerate}[i)]
\begin{minipage}{0.5\linewidth}
\item $h_1 \circ (\text{id},i_1) = \text{id}$
\item $h_1 \circ (i_1,\text{id}) = \text{id}$
\item $h_1 \circ (\text{id},h_1) = h_1 \circ (h_1, \text{id})$
\end{minipage}
\end{enumerate}
So the first slice is the operad for monoids; see Example \ref{pres for monoids}.

Similarly, the second slice of $\bs{T_3}$ is the symmetric operad $C$ consisting the operadic composites of elements $i_2 \in C[0]$ and $h_2, v_2 \in C[2]$ subject to the equalities below.
\vspace{1mm}
\begin{enumerate}[i')]
\begin{minipage}{0.4\linewidth}
\item $h_2 \circ (\text{id}, i_2 ) = \text{id}$
\item $h_2 \circ (i_2, \text{id}) = \text{id}$
\item $v_2 \circ (\text{id}, i_2) = \text{id}$
\item $v_2 \circ (i_2, \text{id}) = \text{id}$
\end{minipage}
\begin{minipage}{0.4\linewidth}
\item $h_2 \circ (\text{id}, h_2) = h_2 \circ (h_2, \text{id})$
\item $v_2 \circ (\text{id}, v_2) = v_2 \circ (v_2, \text{id})$
\item $h_2 \circ (v_2,v_2) = (2 \ 3) \cdot (v_2 \circ (h_2,h_2))$
\item $h_2 \circ (i_2,i_2) = i_2$
\end{minipage}
\end{enumerate}
Note that equality viii') follows from both i') and ii'). Example \ref{Eck-Hil Pres} shows that equalities v') and vi') follow from equalities i') - iv') and vii'), and that the second slice of $\bs{T_3}$ is precisely the symmetric operad for commutative monoids. 
\end{example}

Strict $(n+1)$-categories are the same as strict $n$-categories when truncated to $n$-dimensions, so the globular operads $\bs{T_n}$ and $\bs{T_{n+1}}$ for strict $(n+1)$-categories and strict $n$-categories, respectively, have the the same $k$-cells for all $k \leqslant n$. This means that a presentation for $\bs{T_{n+1}}$ contains a presentation for $\bs{T_n}$; we just omit the $(n+1)$-cell generators and relations. As a result, the $k^{th}$ slices of $\bs{T_n}$ and $\bs{T_{n+1}}$ coincide for all $k \leqslant n$. 

\begin{theorem}
The first slice of the $n$-globular operad $\bs{T_n}$ for strict $n$-categories is the plain operad for monoids, and  $k^{th}$ slice for all $k>1$ is the symmetric operad for commutative monoids.
\end{theorem}

\begin{proof}
The presentation for $\bs{T_3}$ given in Example \ref{pres for T_3} can be generalised to construct a presentation for $\bs{T_n}$ wherein the set $J_k$ of $k$-cell generators consists of an element $i_k$ providing the $k$-cell identities, and an element $b_{kl}$ for each $0 \leqslant l < k$ providing an operation composing a pair of $k$-cells along an $l$-cell boundary. The $k$-cell relations will be those yielding the various unit, associativity and interchange axioms.
The $k^{th}$ slice of $\bs{T_n}$ is then the symmetric operad $C$ consisting of the operadic composites of elements $i_k \in C[0]$ and $b_{k1},...,b_{kk-1} \in C[2]$, subject to the following equalities for all $0 \leqslant l < k$.
\vspace{1mm}
\begin{enumerate}[i)]
\begin{minipage}{0.8\linewidth}
\item $b_{kl} \circ (i_k, \, \text{id}) = \text{id}$
\item $b_{kl} \circ (\text{id}, \, i_k) = \text{id}$
\item $b_{kl} \circ (\text{id}, \, b_{kl}) = b_{kl} \circ (b_{kl}, \, \text{id})$
\item $b_{kl} \circ (b_{ki} , \, b_{ki} ) = (2 \ 3) \cdot ( b_{ki}  \circ (b_{kl}, \, b_{kl}))$ for all $i \neq l$
\end{minipage}
\end{enumerate}
Invoking the Eckmann-Hilton argument once again, we see that the $k^{th}$ slice of $\bs{T_n}$ is the plain operad for monoids if $k=1$, and is the symmetric operad for commutative monoids if $k>1$.
\end{proof}

In order to complete our definition of slices, we need to prove that the slices of a globular operad $\bs{G}$ do not depend on the choice of presentation for $\bs{G}$. For this, we first require some additional details about presentations for globular operads; once again, we note that a more comprehensive overview can be found in \cite{Me}. 

\begin{definition}\label{truncationfunctor} Let $k$ be an integer satisfying $-1 \leqslant k \leqslant n$. The $k^{th}$ \textit{truncation functor} $Tr_k:\mathbf{GOp} \longrightarrow \mathbf{GOp}$ is the functor sending a globular operad $\bs{G}$ to the globular operad whose $l$-cells are those of $\bs{G}$ for all $l \leqslant k$, and whose only $n$-cell is the identity $n$-cell $\text{id}_n$ for all $n > k$.
\end{definition}

\begin{notation}\label{counit} Let $\epsilon_k:Tr_k \Rightarrow 1$ denote the natural transformation whose components are the inclusions.
\end{notation}

\begin{definition} A presentation for a globular operad is a \textit{$k$-presentation} if it contains no $n$-cell generators, and therefore no $n$-cell relations, for all $n>k$. 
\end{definition}

In \cite{Me} presentations are defined by constructing a category $\bs{k} \mhyphen \bf{Pres}$ of $k$-presentations for globular operads for each integer $k \geqslant -1$, together with an adjunction
\begin{center}
\begin{tikzpicture}[node distance=3cm, auto]

\node (A) {$\bs{k} \mhyphen \mathbf{Pres}$};
\node (B) [right of=A] {$\mathbf{GOp}$};

\draw[->, bend left=35] (A) to node {$F_k$} (B);
\draw[->, bend left=35] (B) to node {$V_k$} (A);

\node (W) [node distance=1.5cm, right of=A] {$\perp$};

\end{tikzpicture}
\end{center}
satisfying $F_kV_k=Tr_k$, and with counit $\epsilon_k$. These definitions are recursive; for $k=-1$, we define $\bs{k} \mhyphen \mathbf{Pres}$ to be the terminal category and $F_{-1}$ to be the functor picking out the initial globular operad $\bs{1}$ consisting only of the identity cells. For $k \geqslant 0$, we construct $\bs{k} \mhyphen \bf{Pres}$ and the accompanying adjunction using the category $(\bs{k-1}) \mhyphen \mathbf{Pres}$ and the adjunction $F_{k-1} \dashv V_{k-1}$. 

Recall the $k$-ball $\bs{B_k}$ from Definition \ref{k-ball}. The free functor $F_k$ is defined using the fact that a function $J_k \longrightarrow \mathbf{GSet} (\bs{B_k}, \bs{1}^*)$ corresponds to a morphism $J_k \cdot \bs{B_k} \longrightarrow \bs{1}^*$ of globular sets, where $J_k \cdot \bs{B_k}$ denotes the coproduct of $|J_k|$ copies of  $\bs{B_k}$, under the adjunction
\begin{center}
\begin{tikzpicture}[node distance=3cm, auto]

\node (A) {$\mathbf{Set}$};
\node (B) [right of=A] {$\mathbf{GSet}$.};

\draw[->, bend left=35] (A) to node {$- \cdot \bs{B_k}$} (B);
\draw[->, bend left=35] (B) to node {$\mathbf{GSet}(\bs{B_k}, -)$} (A);

\node (W) [node distance=1.5cm, right of=A] {$\perp$};

\end{tikzpicture}
\end{center}

\begin{definition}\label{The free pres functor.} 
The \textit{free functor} $F_k: \bs{k} \mhyphen \mathbf{Pres} \longrightarrow \mathbf{GOp}$ sends a $k$-presentation $P_k$ with sets $J_k$ and $R_k$ of $k$-cell generators and relations, respectively, to the coequaliser
\begin{center}
\begin{tikzpicture}[node distance=2.5cm, auto]

\node (A) {$F_{k}(P_k)$};
\node (B) [left of=A] {$\bs{J_k}$};
\node (C) [node distance=2.8cm, left of=B] {$R_k \cdot F(\bs{B_k})$};

\draw[transform canvas={yshift=0.5ex},->] (C) to node {$\hat{e}_k$} (B);
\draw[transform canvas={yshift=-0.5ex},->] (C) to node [swap] {$\hat{q}_k$} (B);

\draw[->] (B) to node {} (A);

\end{tikzpicture}
\end{center}
in $\mathbf{GOp}$, where $e_k, q_k : R_k \longrightarrow \mathbf{GSet}(\bs{B_k}, \bs{J_k})$ are as in part (3) of Remarks \ref{k-free remarks} and $F$ is the free functor $\mathbf{GColl} \longrightarrow \mathbf{GOp}$ from the category of globular collections defined in the proof of Theorem \ref{T-op adjunction}.
\end{definition}

\begin{definition}\label{k-pres for G}
A $k$-\textit{presentation} for a globular operad $\bs{G}$ is a $k$-presentation $P_k$ together with an isomorphism $F_k(P_k) \longrightarrow Tr_k(\bs{G})$.
\end{definition}

Observe that since $F_kV_k = Tr_k$, it follows from Definition \ref{k-pres for G} that given a globular operad $\bs{G}$, $V_k(\bs{G})$ is a $k$-presentation for $\bs{G}$.

\begin{theorem}\label{slices don't depend on presentations} The slices of a globular operad $\bs{G}$ are independent of the choice of presentation for $\bs{G}$.
\end{theorem}

\begin{proof} Let $P_k$ be a $k$-presentation for $\bs{G}$. We show that the $k^{th}$ slice defined by $P_k$ is isomorphic to the $k^{th}$ slice defined by the $k$-presentation $V_kF_k(P_k) = V_k(\bs{G})$. Let $J_{(\bs{G},\bs{k})}$ denote the set of $k$-cell generators of $V_k(\bs{G})$, and let $\bs{J}_{(\bs{G},\bs{k})}$ denote the $k$-free globular operad equal to $\bs{G}$ in dimensions $<k$, and with $k$-cells given by the free operadic composites of the set $J_{(\bs{G},\bs{k})}$ of $k$-cell generators. Similarly, let $R_{(\bs{G},k)}$ denote the set of $k$-cell relations in $V_k(\bs{G})$, and let $e_{(\bs{G},k)}$ and $q_{(\bs{G},k)}$ denote the parallel morphisms
\begin{center}
\begin{tikzpicture}[node distance=4.5cm, auto]

\node (A') {$R_{(\bs{G},k)}$};
\node (B') [right of=A'] {$\mathbf{GSet} ( \bs{B_k}, \bs{J}_{(\bs{G,\bs{k})}} )$};

\draw[transform canvas={yshift=0.75ex},->] (A') to node {$e_{(\bs{G},k)}$} (B');
\draw[transform canvas={yshift=-0.25ex},->] (A') to node [swap] {$q_{(\bs{G},k)}$} (B');

\end{tikzpicture}
\end{center}
The unit of the adjunction $F_k \dashv V_k$ at $P_k$ is made up in part of inclusion functions $J_k \longrightarrow J_{(\bs{G},k)}$ and $R_k \longrightarrow R_{(\bs{G},k)}$, the former of which defines a morphism $\bs{J_k} \longrightarrow \bs{J}_{(\bs{G},\bs{k})}$ of globular operads, such that the induced diagram below commutes both with respect to $e_k$ and $e_{(\bs{G},k)}$, and with respect to $q_k$ and $q_{(\bs{G},k)}$.
\begin{center}
\begin{tikzpicture}[node distance=4.8cm, auto]

\node (A) {$R_k$};
\node (B) [right of=A] {$\mathbf{GSet}(\bs{B_k}, \bs{J_k})$};
\node (C) [node distance=4.2cm, right of=B] {$UF(J_k)$};

\draw[transform canvas={yshift=0.75ex},->] (A) to node {$e_k$} (B);
\draw[transform canvas={yshift=-0.25ex},->] (A) to node [swap] {$q_k$} (B);
\draw[->] (B) to node {$f$} (C);

\node (A') [node distance=1.95cm, below of=A] {$R_{(\bs{G},k)}$};
\node (B') [right of=A'] {$\mathbf{GSet}(\bs{B_k}, \bs{J}_{(\bs{G},\bs{k})})$};
\node (C') [node distance=4.2cm, right of=B'] {$UF(J_{(\bs{G},k)})$};

\draw[transform canvas={yshift=0.75ex},->] (A') to node {$e_{(\bs{G},k)}$} (B');
\draw[transform canvas={yshift=-0.25ex},->] (A') to node [swap] {$q_{(\bs{G},k)}$} (B');
\draw[->] (B') to node [swap] {$f$} (C');

\draw[->] (A) to node {} (A');
\draw[->] (B) to node {} (B');
\draw[->] (C) to node {} (C');

\end{tikzpicture}
\end{center}
Consider the commutative diagram,
\begin{center}
\begin{tikzpicture}[node distance=4.2cm, auto]

\node (A) {$R_k \cdot F(\bs{B_k})$};
\node (B) [right of=A] {$\bs{J_k}$};

\draw[transform canvas={yshift=0.75ex},->] (A) to node {$\hat{e}_k$} (B);
\draw[transform canvas={yshift=-0.25ex},->] (A) to node [swap] {$\hat{q}_k$} (B);

\node (A') [node distance=1.95cm, below of=A] {$R_{(\bs{G},k)} \cdot F(\bs{B_k})$};
\node (B') [right of=A'] {$\bs{J}_{(\bs{G},\bs{k})}$};

\draw[transform canvas={yshift=0.75ex},->] (A') to node {$\hat{e}_{(\bs{G},k)}$} (B');
\draw[transform canvas={yshift=-0.25ex},->] (A') to node [swap] {$\hat{q}_{(\bs{G},k)}$} (B');

\draw[->] (A) to node {} (A');
\draw[->] (B) to node {} (B');

\end{tikzpicture}
\end{center}
in $\mathbf{GOp}$. Since $P_k$ and $V_k(\bs{G})$ are both presentations for $\bs{G}$, the coequaliser of both $\hat{e}_k$ and $\hat{q}_k$, and of $\hat{e}_{(\bs{G},k)}$ and $\hat{q}_{(\bs{G},k)}$, is $Tr_k(\bs{G})$. Moreover, the induced morphism $Tr_k(\bs{G}) \longrightarrow Tr_k(\bs{G})$ is the identity. 
It follows that the morphism $S_k \longrightarrow S_{(\bs{G},k)}$ induced from the coequaliser of $\tilde{e}_k$ and $\tilde{q}_k$ to the coequaliser of $\tilde{e}_{(\bs{G},k)}$ and $\tilde{q}_{(\bs{G},k)}$ in the diagram 
\begin{center}
\begin{tikzpicture}[node distance=4cm, auto]

\node (A) {$F(R_k)$};
\node (B) [right of=A] {$F(J_k)$};

\draw[transform canvas={yshift=0.75ex},->] (A) to node {$\tilde{e}_k$} (B);
\draw[transform canvas={yshift=-0.25ex},->] (A) to node [swap] {$\tilde{q}_k$} (B);

\node (A') [node distance=1.95cm, below of=A] {$F(R_{(\bs{G},k)})$};
\node (B') [right of=A'] {$F(J_{(\bs{G},k)})$};

\draw[transform canvas={yshift=0.75ex},->] (A') to node {$\tilde{e}_{(\bs{G},k)}$} (B');
\draw[transform canvas={yshift=-0.25ex},->] (A') to node [swap] {$\tilde{q}_{(\bs{G},k)}$} (B');

\draw[->] (A) to node {} (A');
\draw[->] (B) to node {} (B');

\end{tikzpicture}
\end{center}
in $\mathbf{SymmOp}$ in an isomorphism. Since $S_k$ is the $k^{th}$ slice of $\bs{G}$ arising from $P_k$, and $S_{(\bs{G},k)}$ is the $k^{th}$ slice of $\bs{G}$ arising from $V_k(\bs{G})$, we have proved our result.
\end{proof}

\subsection{Examples}\label{examples of slices subsection}

In \cite{batanin}, Batanin hypothesises that the slices of the globular operad for a theory of higher category can tell us whether those higher categories are equivalent to fully weak ones. In particular, he gives the following conjecture.

\begin{conj}\label{Conjecture}\cite{batanin} Let $\bs{G_n}$ be a contractible $n$-globular operad. The algebras for $\bs{G_n}$ are equivalent to fully weak $n$-categories if the $k^{th}$ slice of $\bs{G_n}$ is a plain operad for all $k<n$.
\end{conj}

In this section, we calculate the slices of the globular operad for weak higher categories, as well as the slices of the globular operads for two theories of semi-strict higher category. The first of these are $n$-categories with weak interchange laws, and the second are $n$-categories with weak identities in dimensions $\leqslant n-2$. Note that for $n=2$, both of these are precisely strict 2-categories. 

\begin{example}\label{slices for W_3} We calculate the first and second slices of the 3-globular operad $\bs{W_3}$ for tricategories using the presentation given in Example \ref{W_3}. From there, we see that the first slice of $\bs{W_3}$ is the symmetric operad $W_1$ consisting of the free operadic composites of elements $i_1 \in W_1[0]$ and $h_1 \in W_1[2]$. Thus, the first slice of $\bs{W_3}$ is the plain operad for pointed magmas. This is what we would expect, as the 1-cells of a tricategory with a single object form a pointed magma; the binary operation is given by 1-cell composition and the single identity 1-cell is the point.

Similarly, the second slice of $\bs{W_3}$ is the symmetric operad $W_2$ consisting of the free operadic composites of elements $i_2$, $l_2$, $l'_2$, $r_2$, $r'_2$, $a_2$, $a'_2 \in W_2[0]$ and $h_2$, $v_2 \in W_2[2]$. Thus, the second slice of $\bs{W_3}$ is the plain operad for 7-pointed double magmas. Again, this is what we would expect, as the 2-cells of a tricategory with a single object and a single 1-cell, which is necessarily the identity 1-cell, form a 7-pointed double magma.
\end{example}

In dimensions $\leqslant 1$, bicategories and tricategories are identical, so the $n$-globular operads $\bs{W_2}$ and $\bs{W_3}$ for bicategories and tricategories, respectively, have the same 0 and 1-cells. More generally, weak $n$-categories and weak $(n{+}1)$-categories should be identical in dimensions $< n$, so the operads $\bs{W_n}$ and $\bs{W_{n+1}}$ for weak $n$-categories and weak $(n{+}1)$-categories, respectively, should have the same $k$-cells for all $k < n$. This means that the $k^{th}$ slices of the globular operads $\bs{W_n}$ and $\bs{W_{n+1}}$ coincide for all $k < n$.

The highest dimension for which we can give a precise presentation for the globular operad $\bs{W_n}$ for fully weak biased\footnote{A biased $n$-category is one where the composition operations on $k$-cells are binary operations; we take two $k$-cells and compose them into another $k$-cell. An unbiased higher category, meanwhile, has a specified operation composing any pasting diagram of $k$-cells directly into a single $k$-cell; see, for example, \cite[\S 7]{Me}.} $n$-categories is $n=4$; this is done in \cite[Definition 9.6]{Me}. The obstacle when trying to construct a presentation for $\bs{W_n}$ for $n>4$ is that it is unknown what kinds of coherence $k$-cells should be present in weak $n$-categories for any $k>3$, and our presentations are constructed by specifying a $k$-cell generator for each binary composition operation on $k$-cells and for each kind of coherence $k$-cell. However, as demonstrated in Example \ref{slices for W_3} above, when taking the slice of a globular operad for higher categories, each kind of coherence cell satisfying no axioms with respect to composition simply becomes a point. 

\begin{definition} A $k$-tuple magma is a set equipped with $k$ binary operations, satisfying no axioms.
\end{definition}
 
\begin{theorem}\label{slices for W_n}
For $k<n$, the $k^{th}$ slice of the $n$-globular operad $\bs{W_n}$ for weak $n$-categories is the plain operad for many-pointed $k$-tuple magmas.
\end{theorem}

\begin{proof}
The presentation for $\bs{W_3}$ given in Example \ref{W_3} can be generalised to construct a presentation for $\bs{W_n}$ wherein for all $k<n$, the set $J_k$ of $k$-cell generators consists of an element for each kind of coherence $k$-cell present in fully weak $n$-categories, together with an element $b_{kj}$ for each $0 \leqslant l < k$ providing an operation composing a pair of $k$-cells along an $l$-cell boundary; and the set $R_k$ of $k$-cell relations is the empty set.
The $k^{th}$ slice of $\bs{W_n}$ is then the free symmetric operad $W_k$ on specified elements in $W_k[0]$, and elements $b_{k1},...,b_{kk-1} \in W_k[2]$. This is precisely the plain operad for many-pointed $k$-tuple magmas.
\end{proof}

We now turn our attention to higher categories with weak interchange laws, but that are strictly associative and unital. Note that 3-categories with weak interchange laws are not the same as Gray categories. Gray categories are strict 3-categories with no direct horizontal composite for 2-cells. On the other hand, 3-categories with weak interchange laws have direct horizontal composites, but the 2-dimensional interchange law holds only up to coherence 3-cells. See \cite{RG1} for a presentation for the 3-globular operad for Gray categories, as well as a comparsion to the operad for 3-categories with weak interchange laws. 

Recall that by Lemma \ref{contractible n-cells}, if an $n$-globular operad $\bs{G_n}$ is contractible, then there is no need to specify any $n$-cell generators or relations in a presentation for $\bs{G_n}$. For the remainder of this section, when constructing a presentation for some $n$-globular operad $\bs{G_n}$ we will declare $\bs{G_n}$ to be contractible. The proofs that the globular operads here are actually contractible can be found in \cite[\S 9.2]{Me}

\begin{definition}\label{H_4} The 4-globular operad $\bs{H_4}$ for 4-categories with weak interchange laws is the \textit{contractible} 4-globular operad with the same 0 and 1-cells as $\bs{T_3}$ (see Example \ref{pres for T_3}), and

\begin{list}{•}{}

\item 2-cells consisting of the operadic composites of 2-cells $i_2$, $h_2$ and $v_2$ whose images under the underlying collection map are as in the presentation for $\bs{T_3}$, subject to the following equalities,
\vspace{1mm}
\begin{enumerate}[i')]
\begin{minipage}{0.5\linewidth}
\item $h_2 \circ (\text{id}_2, i_2 \circ (i_1)) = \text{id}_2$
\item $h_2 \circ (i_2 \circ (i_1), \text{id}_2) = \text{id}_2$
\item $v_2 \circ (\text{id}_2,i_2) = \text{id}_2$
\item $v_2 \circ (i_2, \text{id}_2) = \text{id}_2$
\end{minipage}
\begin{minipage}{0.5\linewidth}
\item $h_2 \circ (\text{id}_2, h_2) = h_2 \circ (h_2, \text{id}_2)$
\item $v_2 \circ (\text{id}_2, v_2) = v_2 \circ (v_2, \text{id}_2)$
\item $h_2 \circ (i_2,i_2) = i_2 \circ (h_1)$;
\item[\vspace{\fill}]
\end{minipage}
\end{enumerate}

\item 3-cells consisting of the operadic composites of 3-cells $i_3$, $h_3$, $v_3$, $c_3$, $s_3$, and $s'_3$ whose images under the underlying collection map are,
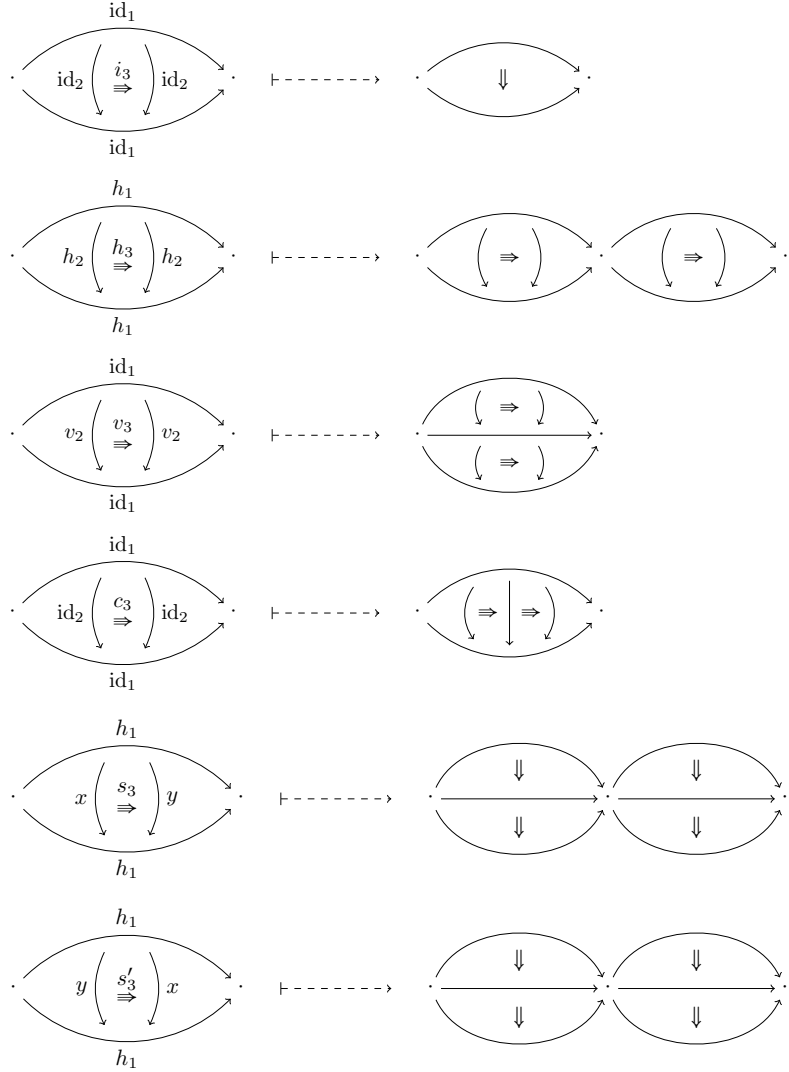
\begin{center}
\resizebox{0.67\textwidth}{!}{
\begin{tikzpicture}[node distance=2cm, auto]

\node (A) {$\cdot$};
\node (B) [node distance=3.6cm, right of=A] {$\cdot$};
\node (a) [node distance=1.8cm, right of=A] {};
\node () [node distance=0.15cm, above of=a] {$i_3$};
\node () [node distance=0.15cm, below of=a] {$\Rrightarrow$};

\node (c) [node distance=1.5cm, right of=A] {};
\node (e) [node distance=0.7cm, above of=c] {};
\node (f) [node distance=0.7cm, below of=c] {};
\node (d) [node distance=2.1cm, right of=A] {};
\node (g) [node distance=0.7cm, above of=d] {};
\node (h) [node distance=0.7cm, below of=d] {};

\draw[->, bend left=45] (A) to node {$\text{id}_1$} (B);
\draw[->, bend right=45] (A) to node [swap] {$\text{id}_1$} (B);
\draw[->, bend right=25] (e) to node [swap] {$\text{id}_2$} (f);
\draw[->, bend left=25] (g) to node {$\text{id}_2$} (h);

\node (P) [node distance=3cm, right of =B] {$\cdot$};
\node (Q) [node distance=2.8cm, right of =P] {$\cdot$};
\node () [node distance=1.4cm, right of=P] {$\Downarrow$};

\draw[->, bend left=40] (P) to node {} (Q);
\draw[->, bend right=40] (P) to node {} (Q);

\node (x) [node distance=0.5cm, right of=B] {};
\node (y) [node distance=0.5cm, left of=P] {};
\draw[|->, dashed] (x) to node {} (y);

\node (A') [node distance=2.9cm, below of=A] {$\cdot$};
\node (B') [node distance=3.6cm, right of=A'] {$\cdot$};
\node (a') [node distance=1.8cm, right of=A'] {};
\node () [node distance=0.15cm, above of=a'] {$h_3$};
\node () [node distance=0.15cm, below of=a'] {$\Rrightarrow$};

\node (c') [node distance=1.5cm, right of=A'] {};
\node (e') [node distance=0.7cm, above of=c'] {};
\node (f') [node distance=0.7cm, below of=c'] {};
\node (d') [node distance=2.1cm, right of=A'] {};
\node (g') [node distance=0.7cm, above of=d'] {};
\node (h') [node distance=0.7cm, below of=d'] {};

\draw[->, bend left=45] (A') to node {$h_1$} (B');
\draw[->, bend right=45] (A') to node [swap] {$h_1$} (B');
\draw[->, bend right=25] (e') to node [swap] {$h_2$} (f');
\draw[->, bend left=25] (g') to node {$h_2$} (h');

\node (P') [node distance=3cm, right of=B'] {$\cdot$};
\node (Q') [node distance=3cm, right of=P'] {$\cdot$};
\node (R') [node distance=3cm, right of=Q'] {$\cdot$};
\node () [node distance=1.5cm, right of=P'] {$\Rrightarrow$};
\node () [node distance=1.5cm, right of=Q'] {$\Rrightarrow$};

\node (u') [node distance=1.2cm, right of=P'] {};
\node (i') [node distance=0.6cm, above of=u'] {};
\node (j') [node distance=0.6cm, below of=u'] {};
\node (v') [node distance=1.8cm, right of=P'] {};
\node (n') [node distance=0.6cm, above of=v'] {};
\node (m') [node distance=0.6cm, below of=v'] {};

\node (w') [node distance=1.2cm, right of=Q'] {};
\node (k') [node distance=0.6cm, above of=w'] {};
\node (l') [node distance=0.6cm, below of=w'] {};
\node (z') [node distance=1.8cm, right of=Q'] {};
\node (s') [node distance=0.6cm, above of=z'] {};
\node (t') [node distance=0.6cm, below of=z'] {};

\draw[->, bend left=45] (P') to node {} (Q');
\draw[->, bend right=45] (P') to node [swap] {} (Q');
\draw[->, bend right=30] (i') to node [swap] {} (j');
\draw[->, bend left=30] (n') to node {} (m');

\draw[->, bend left=45] (Q') to node {} (R');
\draw[->, bend right=45] (Q') to node [swap] {} (R');
\draw[->, bend right=30] (k') to node [swap] {} (l');
\draw[->, bend left=30] (s') to node {} (t');

\node (x') [node distance=0.5cm, right of=B'] {};
\node (y') [node distance=0.5cm, left of=P'] {};
\draw[|->, dashed] (x') to node {} (y');

\node (A'') [node distance=2.9cm, below of=A'] {$\cdot$};
\node (B'') [node distance=3.6cm, right of=A''] {$\cdot$};
\node (a'') [node distance=1.8cm, right of=A''] {};
\node () [node distance=0.15cm, above of=a''] {$v_3$};
\node () [node distance=0.15cm, below of=a''] {$\Rrightarrow$};

\node (c'') [node distance=1.5cm, right of=A''] {};
\node (e'') [node distance=0.7cm, above of=c''] {};
\node (f'') [node distance=0.7cm, below of=c''] {};
\node (d'') [node distance=2.1cm, right of=A''] {};
\node (g'') [node distance=0.7cm, above of=d''] {};
\node (h'') [node distance=0.7cm, below of=d''] {};

\draw[->, bend left=45] (A'') to node {$\text{id}_1$} (B'');
\draw[->, bend right=45] (A'') to node [swap] {$\text{id}_1$} (B'');
\draw[->, bend right=25] (e'') to node [swap] {$v_2$} (f'');
\draw[->, bend left=25] (g'') to node {$v_2$} (h'');

\node (P'') [node distance=3cm, right of=B''] {$\cdot$};
\node (Q'') [node distance=3cm, right of=P''] {$\cdot$};

\draw[->, bend left=65] (P'') to node {} (Q'');
\draw[->] (P'') to node {} (Q'');
\draw[->, bend right=65] (P'') to node [swap] {} (Q'');

\node (u'') [node distance=1.5cm, right of=P''] {};
\node (i'') [node distance=0.45cm, above of=u''] {$\Rrightarrow$};
\node (j'') [node distance=0.45cm, below of=u''] {$\Rrightarrow$};

\node (n'') [node distance=0.4cm, left of=i'', above of=i''] {};
\node (m'') [node distance=0.4cm, left of=i'', below of=i''] {};
\node (k'') [node distance=0.4cm, right of=i'', above of=i''] {};
\node (l'') [node distance=0.4cm, right of=i'', below of=i''] {};
\node (w'') [node distance=0.4cm, left of=j'', above of=j''] {};
\node (z'') [node distance=0.4cm, left of=j'', below of=j''] {};
\node (s'') [node distance=0.4cm, right of=j'', above of=j''] {};
\node (t'') [node distance=0.4cm, right of=j'', below of=j''] {};

\draw[->, bend right=30] (n'') to node [swap] {} (m'');
\draw[->, bend left=30] (k'') to node {} (l'');
\draw[->, bend right=30] (w'') to node [swap] {} (z'');
\draw[->, bend left=30] (s'') to node {} (t'');

\node (x'') [node distance=0.5cm, right of=B''] {};
\node (y'') [node distance=0.5cm, left of=P''] {};
\draw[|->, dashed] (x'') to node {} (y'');

\node (A''') [node distance=2.9cm, below of=A''] {$\cdot$};
\node (B''') [node distance=3.6cm, right of=A'''] {$\cdot$};
\node (a''') [node distance=1.8cm, right of=A'''] {};
\node () [node distance=0.15cm, above of=a'''] {$c_3$};
\node () [node distance=0.15cm, below of=a'''] {$\Rrightarrow$};

\node (c''') [node distance=1.5cm, right of=A'''] {};
\node (e''') [node distance=0.7cm, above of=c'''] {};
\node (f''') [node distance=0.7cm, below of=c'''] {};
\node (d''') [node distance=2.1cm, right of=A'''] {};
\node (g''') [node distance=0.7cm, above of=d'''] {};
\node (h''') [node distance=0.7cm, below of=d'''] {};

\draw[->, bend left=45] (A''') to node {$\text{id}_1$} (B''');
\draw[->, bend right=45] (A''') to node [swap] {$\text{id}_1$} (B''');
\draw[->, bend right=25] (e''') to node [swap] {$\text{id}_2$} (f''');
\draw[->, bend left=25] (g''') to node {$\text{id}_2$} (h''');

\node (P''') [node distance=3cm, right of=B'''] {$\cdot$};
\node (Q''') [node distance=3cm, right of=P'''] {$\cdot$};

\node (u''') [node distance=1cm, right of=P'''] {};
\node (i''') [node distance=1.5cm, right of=P'''] {};
\node (j''') [node distance=2cm, right of=P'''] {};
\node (n''') [node distance=0.55cm, above of=u'''] {};
\node (m''') [node distance=0.55cm, below of=u'''] {};
\node (k''') [node distance=0.65cm, above of=i'''] {};
\node (l''') [node distance=0.65cm, below of=i'''] {};
\node (w''') [node distance=0.55cm, above of=j'''] {};
\node (z''') [node distance=0.55cm, below of=j'''] {};
\node () [node distance=1.15cm, right of=P'''] {$\Rrightarrow$};
\node () [node distance=1.85cm, right of=P'''] {$\Rrightarrow$};

\draw[->, bend left=45] (P''') to node {} (Q''');
\draw[->, bend right=45] (P''') to node [swap] {} (Q''');
\draw[->, bend right=35] (n''') to node [swap] {} (m''');
\draw[->] (k''') to node [swap] {} (l''');
\draw[->, bend left=35] (w''') to node {} (z''');

\node (x''') [node distance=0.5cm, right of=B'''] {};
\node (y''') [node distance=0.5cm, left of=P'''] {};
\draw[|->, dashed] (x''') to node {} (y''');

\end{tikzpicture}
}
\end{center}

\begin{center}
\resizebox{0.67\textwidth}{!}{
\begin{tikzpicture}[node distance=2cm, auto]

\node (A*) {$\cdot$};
\node (B*) [node distance=3.6cm, right of=A*] {$\cdot$};
\node (a*) [node distance=1.8cm, right of=A*] {};
\node () [node distance=0.15cm, above of=a*] {$s_3$};
\node () [node distance=0.15cm, below of=a*] {$\Rrightarrow$};

\node (c*) [node distance=1.5cm, right of=A*] {};
\node (e*) [node distance=0.7cm, above of=c*] {};
\node (f*) [node distance=0.7cm, below of=c*] {};
\node (d*) [node distance=2.1cm, right of=A*] {};
\node (g*) [node distance=0.7cm, above of=d*] {};
\node (h*) [node distance=0.7cm, below of=d*] {};

\draw[->, bend left=45] (A*) to node {$h_1$} (B*);
\draw[->, bend right=45] (A*) to node [swap] {$h_1$} (B*);
\draw[->, bend right=25] (e*) to node [swap] {$x$} (f*);
\draw[->, bend left=25] (g*) to node {$y$} (h*);

\node (P*) [node distance=3cm, right of=B*] {$\cdot$};
\node (Q*) [node distance=2.8cm, right of=P*] {$\cdot$};
\node (R*) [node distance=2.8cm, right of=Q*] {$\cdot$};

\node (i*) [node distance=1.4cm, right of=P*] {};
\node () [node distance=0.45cm, above of=i*] {$\Downarrow$};
\node () [node distance=0.45cm, below of=i*] {$\Downarrow$};
\node (j*) [node distance=1.4cm, right of=Q*] {};
\node () [node distance=0.45cm, above of=j*] {$\Downarrow$};
\node () [node distance=0.45cm, below of=j*] {$\Downarrow$};

\draw[->, bend left=65] (P*) to node {} (Q*);
\draw[->] (P*) to node {} (Q*);
\draw[->, bend right=65] (P*) to node [swap] {} (Q*);
\draw[->, bend left=65] (Q*) to node {} (R*);
\draw[->] (Q*) to node {} (R*);
\draw[->, bend right=65] (Q*) to node [swap] {} (R*);

\node (x*) [node distance=0.5cm, right of=B*] {};
\node (y*) [node distance=0.5cm, left of=P*] {};
\draw[|->, dashed] (x*) to node {} (y*);

\node (A**) [node distance=3cm, below of=A*] {$\cdot$};
\node (B**) [node distance=3.6cm, right of=A**] {$\cdot$};
\node (a**) [node distance=1.8cm, right of=A**] {};
\node () [node distance=0.15cm, above of=a**] {$s'_3$};
\node () [node distance=0.15cm, below of=a**] {$\Rrightarrow$};

\node (c**) [node distance=1.5cm, right of=A**] {};
\node (e**) [node distance=0.7cm, above of=c**] {};
\node (f**) [node distance=0.7cm, below of=c**] {};
\node (d**) [node distance=2.1cm, right of=A**] {};
\node (g**) [node distance=0.7cm, above of=d**] {};
\node (h**) [node distance=0.7cm, below of=d**] {};

\draw[->, bend left=45] (A**) to node {$h_1$} (B**);
\draw[->, bend right=45] (A**) to node [swap] {$h_1$} (B**);
\draw[->, bend right=25] (e**) to node [swap] {$y$} (f**);
\draw[->, bend left=25] (g**) to node {$x$} (h**);

\node (P**) [node distance=3cm, right of=B**] {$\cdot$};
\node (Q**) [node distance=2.8cm, right of=P**] {$\cdot$};
\node (R**) [node distance=2.8cm, right of=Q**] {$\cdot$};

\node (i**) [node distance=1.4cm, right of=P**] {};
\node () [node distance=0.45cm, above of=i**] {$\Downarrow$};
\node () [node distance=0.45cm, below of=i**] {$\Downarrow$};
\node (j**) [node distance=1.4cm, right of=Q**] {};
\node () [node distance=0.45cm, above of=j**] {$\Downarrow$};
\node () [node distance=0.45cm, below of=j**] {$\Downarrow$};

\draw[->, bend left=65] (P**) to node {} (Q**);
\draw[->] (P**) to node {} (Q**);
\draw[->, bend right=65] (P**) to node [swap] {} (Q**);
\draw[->, bend left=65] (Q**) to node {} (R**);
\draw[->] (Q**) to node {} (R**);
\draw[->, bend right=65] (Q**) to node [swap] {} (R**);

\node (x**) [node distance=0.5cm, right of=B**] {};
\node (y**) [node distance=0.5cm, left of=P**] {};
\draw[|->, dashed] (x**) to node {} (y**);

\end{tikzpicture}
}
\end{center}
where $x=h_2 \circ (v_2,v_2)$ and $y=v_2 \circ (h_2,h_2)$, subject to the equalities below.
\vspace{1mm}
\begin{enumerate}[i'')]
\begin{minipage}{0.5\linewidth}
\item $h_3 \circ (\text{id}_3, i_3 \circ (i_2 \circ (i_1))) = \text{id}_3$
\item $h_3 \circ ( i_3 \circ (i_2 \circ (i_1)), \text{id}_3) = \text{id}_3$
\item $v_3 \circ (\text{id}_3, i_3 \circ (i_2)) = \text{id}_3$
\item $v_3 \circ (i_3 \circ (i_2), \text{id}_3) = \text{id}_3$
\item $c_3 \circ (\text{id}_3, i_3) = \text{id}_3$
\item $c_3 \circ (i_3, \text{id}_3) = \text{id}_3$
\end{minipage}
\begin{minipage}{0.5\linewidth}
\item $h_3 \circ (\text{id}_3, h_3) = h_3 \circ (h_3,\text{id}_3)$
\item $v_3 \circ (\text{id}_3, v_3) = v_3 \circ (v_3, \text{id}_3)$
\item $c_3 \circ (\text{id}_3, c_3) = c_3 \circ (c_3, \text{id}_3)$
\item $h_3 \circ (i_3, i_3) = i_3 \circ (h_2)$ 
\item $v_3 \circ (i_3, i_3) = i_3 \circ (v_2)$ 
\item[\hspace{\fill}]
\end{minipage}
\end{enumerate}

\end{list}

An algebra for $\bs{H_4}$ on a 4-globular set $\bs{A}$ is a 4-category with weak interchange laws with underlying 4-globular set $\bs{A}$. Note that in dimensions $\leqslant 2$, our presentation for $\bs{H_4}$ is the same as the presentation for the 2-globular operad $\bs{T_3}$ for strict 3-categories given in Example \ref{pres for T_3}, with the exception of a single relation; $\bs{H_4}$ does not have the 2-cell relation yielding a strict interchange law. Thus, we have 3-cell generators $s_3$ and $s'_3$ providing interchange coherence 3-cells and their weak inverses, respectively. The 3-cell relations yield unit axioms for composition of 3-cells along matching 0, 1, and 2-cell boundaries, three associativity axioms, and axioms stating the composite of two identity 3-cells along a 0 or 1-cell is another identity 3-cell.
\end{definition}

\begin{example}
Since the 1-cells of $\bs{H_4}$ are the same as the 1-cells of $\bs{T_n}$, the first slices of these operads coincide, so the first slice of $\bs{H_4}$ is the plain operad for monoids. Using the presentation above, we see that the second slice is the symmetric operad $H_2$ consisting of the operadic composites of elements $i_2 \in H_2[0]$ and $h_2, v_2 \in H_2[2]$ subject to the equalities below.
\vspace{1mm}
\begin{enumerate}[i')]
\begin{minipage}{0.4\linewidth}
\item $h_2 \circ (\text{id}, i_2 ) = \text{id}$
\item $h_2 \circ (i_2, \text{id}) = \text{id}$
\item $v_2 \circ (\text{id}, i_2) = \text{id}$
\item $v_2 \circ (i_2, \text{id}) = \text{id}$
\end{minipage}
\begin{minipage}{0.4\linewidth}
\item $h_2 \circ (\text{id}, h_2) = h_2 \circ (h_2, \text{id})$
\item $v_2 \circ (\text{id}, v_2) = v_2 \circ (v_2, \text{id})$
\item $h_2 \circ (i_2,i_2) = i_2$
\end{minipage}
\end{enumerate}
Note that equality vii') follows from both i') and ii'), so comparing with Example \ref{symm op for double monoids}, we see that the second slice of $\bs{H_4}$ is the plain operad for double monoids with shared unit.

The third slice of $\bs{H_4}$ is the symmetric operad $H_3$ consisting of the operadic composites of elements $i_3, s_3, s'_3 \in H_3[0]$ and $h_3, v_3, c_3 \in H_3[2]$ subject to the equalities below.
\vspace{1mm}
\begin{enumerate}[i'')]
\begin{minipage}{0.5\linewidth}
\item $h_3 \circ (\text{id}, i_3) = \text{id}$
\item $h_3 \circ (i_3, \text{id}) = \text{id}$
\item $v_3 \circ (\text{id}, i_3) = \text{id}$
\item $v_3 \circ (i_3, \text{id}) = \text{id}$
\item $c_3 \circ (\text{id}, i_3) = \text{id}$
\item $c_3 \circ (i_3, \text{id}) = \text{id}$
\end{minipage}
\begin{minipage}{0.5\linewidth}
\item $h_3 \circ (\text{id}, h_3) = h_3 \circ (h_3,\text{id})$
\item $v_3 \circ (\text{id}, v_3) = v_3 \circ (v_3, \text{id})$
\item $c_3 \circ (\text{id}, c_3) = c_3 \circ (c_3, \text{id})$
\item $h_3 \circ (i_3, i_3) = i_3$ 
\item $v_3 \circ (i_3, i_3) = i_3$ 
\item[\hspace{\fill}]
\end{minipage}
\end{enumerate}
Equations x'') and xi'') follow from the first four equations, so we may disregard them from our list. The third slice of $H_4$ is then the plain operad for two-pointed triple monoids with shared unit. The generators $h_3$, $v_3$ and $c_3$ each provide a binary operation, and $i_3$ provides the shared unit. The points are given by $s_3$ and $s'_3$.
\end{example}

Observe that $n$-categories with weak interchange laws are the same as $(n+1)$-categories with weak interchange laws when truncated to $n-1$ dimensions. The corresponding globular operads $\bs{H_n}$ and $\bs{H_{n+1}}$ are therefore equal when truncated to any dimension $<n$, so the $k^{th}$ slices of these operads coincide for all $k<n$. 

\begin{theorem}\label{slices_for_weak_interchange} For $k<n$, the $k^{th}$ slice of the $n$-globular operad $\bs{H_n}$ for $n$-categories with weak interchange laws is the plain operad $H_k$ for (many-pointed) $k$-tuple monoids with shared unit.
\end{theorem}

\begin{proof}
The presentation for $\bs{H_4}$ given in Definition \ref{H_4} can be generalised to construct a presentation for $\bs{H_n}$ wherein for all $k<n$, the set $J_k$ of $k$-cell generators consists of an element for each kind of coherence $k$-cell related to interchange, an element $i_k$ providing the identity $k$-cells, and an element $b_{kl}$ for each $0 \leqslant l < k$ providing an operation composing a pair of $k$-cells along an $l$-cell boundary. The set $R_k$ of $k$-cell relations are those yielding unit and associativity axioms, as well as those axioms stating that the composite of two identity $k$-cells along an $l$-dimensional boundary is another identity cell. The later of these kinds of relations become redundant when taking the slice, and may therefore by disregarded. 
The $k^{th}$ slice of $\bs{H_n}$ is then the symmetric operad $H_k$ consisting of the free operadic composites of specified elements in $H_k[0]$, including an element $i_k \in H_k[0]$, and elements $b_{k1},...,b_{kk-1} \in H_k[2]$, subject to the following equalities for all $0 \leqslant l < k$,
\vspace{1mm}
\begin{enumerate}[i)]
\begin{minipage}{0.8\linewidth}
\item $b_{kl} \circ (\text{id}, \, i_k) = \text{id}$
\item $b_{kl} \circ (i_k, \, \text{id}) = \text{id}$
\item $b_{kl} \circ (\text{id}, \, b_{kl}) = b_{kl} \circ (b_{kl}, \, \text{id})$.
\end{minipage}
\end{enumerate}
Thus, $H_k$ is the operad for many-pointed $k$-tuple monoids with shared unit, and Lemma \ref{prop - presentations for plain} tells us that this is a plain operad.
\end{proof}

As a final example, we look at the slices of the $n$-globular operad for $n$-categories with weak units in dimensions $< n-1$. 

\begin{definition} \label{E_3}
The 3-globular operad $\bs{E_3}$ for 3-categories with weak units in low dimensions is the \textit{contractible} 3-globular operad with

\begin{list}{$\bullet$}{}

\item a single 0-cell, the identity $\text{id}_0$; 

\item 1-cells consisting of the operadic composites of 1-cells $i_1$ and $h_1$ whose images under the underlying collection map are as in the presentation for $\bs{W_3}$, subject to the equality 
\begin{enumerate}[i)]
\item $h_1 \circ (\text{id}_1, h_1) = h_1 \circ (h_1, \text{id}_1)$; and
\end{enumerate}

\item 2-cells consisting of the operadic composites of 2-cells $i_2$, $h_2$, $v_2$, $l_2$, $l'_2$, $r_2$ and $r'_2$ whose images under the underlying collection map are as in the presentation for $\bs{W_3}$ (see Example \ref{W_3}), subject to the following equalities
\vspace{1mm}
\begin{enumerate}[i')]
\begin{minipage}{0.3\linewidth}
\item $v_2 \circ (\text{id}_2,i_2) = \text{id}_2$
\item $v_2 \circ (i_2, \text{id}_2) = \text{id}_2$
\end{minipage}
\begin{minipage}{0.4\linewidth}
\item $h_2 \circ (\text{id}_2, h_2) = h_2 \circ (h_2, \text{id}_2)$
\item $v_2 \circ (\text{id}_2, v_2) = v_2 \circ (v_2, \text{id}_2)$
\end{minipage}
\begin{minipage}{0.4\linewidth}
\item $h_2 \circ (v_2,v_2) = v_2 \circ (h_2,h_2)$
\item $h_2 \circ (i_2,i_2) = i_2 \circ (h_1)$
\end{minipage}
\end{enumerate}
\end{list}

An algebra for $\bs{E_4}$ on a 3-globular set $\bs{A}$ is a 3-category with weak 1-cell identities whose underlying 3-globular set is $\bs{A}$. The single 1-cell relation yields an associativity axiom for 1-cell composition. Thus, the 2-cell generators $a_2$ and $a_2'$ for associativity coherence 2-cells that appear in the presentation for $\bs{W_3}$ do not appear here, but every other 2-cell generator does. The 2-cell relations yield unit axioms for \textit{vertical} composition of 2-cells, associativity axioms for both horizontal and vertical composition of 2-cells, respectively, an interchange law, and an axiom stating that the horizontal composite of two 2-cell identities is an identity 2-cell.
\end{definition}

Recall that a semigroup is a set equipped with an associative binary operation.

\begin{example} Using the presentation above, we see that the first slice of $\bs{E_3}$ is the symmetric operad $E_1$ consisting of the operadic composites of elements $i_1 \in E_1[0]$ and $h_1 \in E_1[2]$ subject to the equality $h_1 \circ (\text{id}, h_1) = h_1 \circ (h_1, \text{id})$. Thus, the first slice is the plain operad for pointed semigroups.

The second slice of $\bs{E_3}$ is the symmetric operad $\tilde{E}_2$ consisting of the operadic composites of elements $i_2$, $l_2$, $l'_2$, $r_2$, $r'_2 \in \tilde{E}_2[0]$ and $h_2, v_2 \in \tilde{E}_2[2]$ subject to the following equalities.
\vspace{1mm}
\begin{enumerate}[i')]
\begin{minipage}{0.26\linewidth}
\item $v_2 \circ (\text{id},i_2) = \text{id}$
\item $v_2 \circ (i_2, \text{id}) = \text{id}$
\end{minipage}
\begin{minipage}{0.36\linewidth}
\item $h_2 \circ (\text{id}, h_2) = h_2 \circ (h_2, \text{id})$
\item $v_2 \circ (\text{id}, v_2) = v_2 \circ (v_2, \text{id})$
\end{minipage}
\begin{minipage}{0.4\linewidth}
\item $h_2 \circ (v_2,v_2) = (2 \ 3) \cdot (v_2 \circ (h_2,h_2))$
\item $h_2 \circ (i_2,i_2) = i_2$
\end{minipage}
\end{enumerate}
Given an algebra for $\tilde{E}_2$ on a set $X$, the generators $i_2$, $l_2$, $l'_2$, $r_2$, and $r'_2$ each specify an element of $X$, while $h_2$ and $v_2$ provide binary operations on $X$, denoted $\ast$ and $\cdot$, respectively, such that
\begin{itemize}
\item $(X, \ast)$ is a semigroup; 
\item $(X, \cdot, 1_X)$ is a monoid whose identity $1_X$ is the element specified by $i_2$;
\item $1_X \ast 1_X = 1_X$; and
\item $(w \cdot x) \ast (y \cdot z) = (w \ast y) \cdot (x \ast z)$ for all $w,x,y,z \in X$.
\end{itemize}
\end{example}

The slices for $\bs{E_3}$ have a nice diagrammatic representation. Consider the first slice - the plain operad for pointed semigroups. If we represent an element $x$ in a pointed semigroup $X$ by,
\begin{center}
\begin{tikzpicture}[node distance=1cm, auto]

\node (A) {};
\node (B) [right of=A] {$\bullet$};
\node (C) [right of=B] {};

\node () [node distance=3mm, above of=B] {$x$};

\draw[-] (A) to node {} (C);

\end{tikzpicture}
\end{center}
then the binary operation may be represented by concatenation. For example, if we denote our binary operation by $\ast$, then $x \ast y$ is represented by the diagram below.
\begin{center}
\begin{tikzpicture}[node distance=1cm, auto]

\node (A) {};
\node (B) [right of=A] {$\bullet$};
\node (C) [right of=B] {$\bullet$};
\node (D) [right of=C] {};

\node () [node distance=3mm, above of=B] {$x$};
\node () [node distance=3mm, above of=C] {$y$};

\draw[-] (A) to node {} (D);

\end{tikzpicture}
\end{center}
The fact that this binary operation is associative means we may represent composites of any length without ambiguity, so a diagram like
\begin{center}
\begin{tikzpicture}[node distance=1cm, auto]

\node (A) {};
\node (B) [right of=A] {$\bullet$};
\node (C) [right of=B] {$\bullet$};
\node (D) [right of=C] {$\bullet$};
\node (E) [right of=D] {$\bullet$};
\node (F) [right of=E] {};

\node () [node distance=3mm, above of=B] {$w$};
\node () [node distance=3mm, above of=C] {$x$};
\node () [node distance=3mm, above of=D] {$y$};
\node () [node distance=3mm, above of=E] {$z$};

\draw[-] (A) to node {} (F);

\end{tikzpicture}
\end{center}
represent a unique element of $X$, since by the associativity axiom, any way of composing the ordered list $(w,x,y,z)$ of elements in $X$ is equal. To keep track of the point $\pentagram \in X$, we use the notation 
\begin{center}
\begin{tikzpicture}[node distance=1cm, auto]

\node[circle,draw, scale=0.4] (X) at (0,0) {};
\draw(-1,0) to (X);
\draw(X) to (1,0);

\node () [node distance=3mm, above of=X] {};

\end{tikzpicture}
\end{center}
so the element $x \ast \pentagram \ast y$ is represented by the diagram below.

\begin{center}
\begin{tikzpicture}[node distance=1cm, auto]

\node[circle,draw, scale=0.4] (X) at (0,0) {};
\node(A) at (-1,0) {$\bullet$};
\node(B) at (1,0) {$\bullet$};
\draw(-2,0) to (X);
\draw(X) to (2,0);

\node () [node distance=3mm, above of=A] {$x$};
\node () [node distance=3mm, above of=B] {$y$};

\end{tikzpicture}
\end{center}

Next, consider the second slice of $\bs{E_3}$. 
If we represent an element $x$ in an algebra for $\tilde{E}_2$ as,
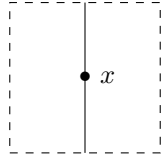
\begin{center}
\begin{tikzpicture}[node distance=1cm, auto]

\draw[dashed] (-1,-1) rectangle (1, 1);

\node (X) at (0,0) {$\bullet$};

\draw(0,-1) to (0,1);

\node () [node distance=3mm, right of=X] {$x$};

\end{tikzpicture}
\end{center}
then the binary operations $\ast$ and $\cdot$ can be represented by horizontal and vertical concatenation, respectively. For example, the elements $x \ast y$ and $x \cdot y$ are represented by the diagrams below.
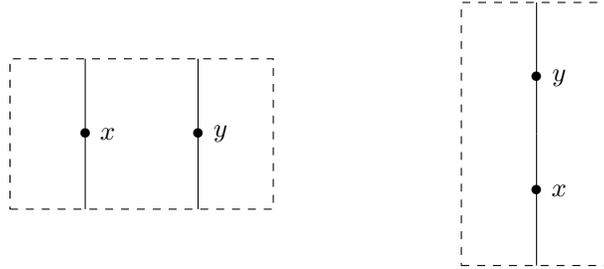
\begin{center}
\begin{tikzpicture}[node distance=1cm, auto]

\draw[dashed] (-1,-1) rectangle (2.5, 1);
\node (X) at (0,0) {$\bullet$};
\draw(0,-1) to (0,1);
\node () [node distance=3mm, right of=X] {$x$};

\node (Y) at (1.5,0) {$\bullet$};
\draw(1.5,-1) to (1.5,1);
\node () [node distance=3mm, right of=Y] {$y$};

\draw[dashed] (5,-1.75) rectangle (7, 1.75);
\node (X') at (6,-0.75) {$\bullet$};
\draw(6,-1.75) to (6,0);
\node () [node distance=3mm, right of=X'] {$x$};

\node (Y') at (6,0.75) {$\bullet$};
\draw(6,0) to (6,1.75);
\node () [node distance=3mm, right of=Y'] {$y$};

\end{tikzpicture}
\end{center}
The fact that the latter operation has an identity $1_X$ can be incorporated into our diagrams by representing $1_X$ simply as 
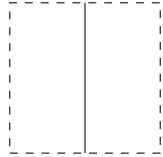
\begin{center}
\begin{tikzpicture}[node distance=1cm, auto]

\draw[dashed] (-1,-1) rectangle (1, 1);

\draw(0,-1) to (0,1);

\node () [node distance=3mm, right of=X] {};

\end{tikzpicture}
\end{center}
so the equations $x \cdot 1_X = x = 1_X \cdot x$ are expressed diagrammatically as: 
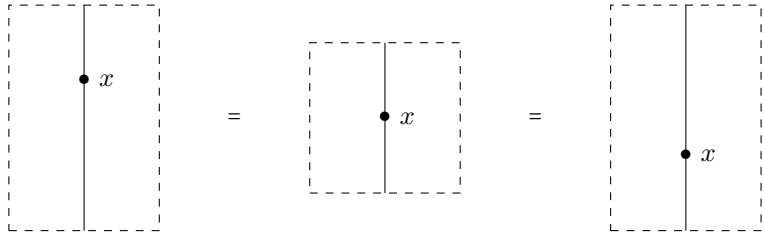
\begin{center}
\begin{tikzpicture}[node distance=1cm, auto]

\draw[dashed] (-1,-0.5) rectangle (1, 2.5);
\draw(0,-0.5) to (0,1);

\node (Y) at (0,1.5) {$\bullet$};
\draw(0,1) to (0,2.5);
\node () [node distance=3mm, right of=Y] {$x$};

\node () at (2,1) {$=$};

\draw[dashed] (3,0) rectangle (5, 2);
\node (X') at (4,1) {$\bullet$};
\draw(4,0) to (4,2);
\node () [node distance=3mm, right of=X'] {$x$};

\node () at (6,1) {$=$};

\draw[dashed] (7,-0.5) rectangle (9, 2.5);
\node (X'') at (8,0.5) {$\bullet$};
\draw(8,-0.5) to (8,2.5);
\node () [node distance=3mm, right of=X''] {$x$};

\end{tikzpicture}
\end{center}
The axiom $1_X \ast 1_X = 1_X$ can be represented as the following equality of diagrams.
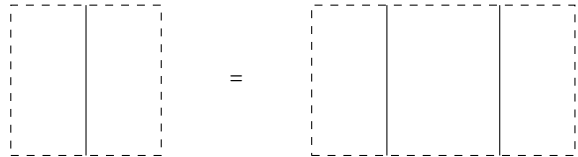
\begin{center}
\begin{tikzpicture}[node distance=1cm, auto]

\draw[dashed] (-1,-1) rectangle (1, 1);

\draw(0,-1) to (0,1);

\node () [node distance=3mm, right of=X] {};

\node () at (2,0) {$=$};

\draw[dashed] (3,-1) rectangle (6.5, 1);

\draw(4,-1) to (4,1);

\draw(5.5,-1) to (5.5,1);

\node () [node distance=3mm, right of=X] {};

\end{tikzpicture}
\end{center}
As a result, diagrams like
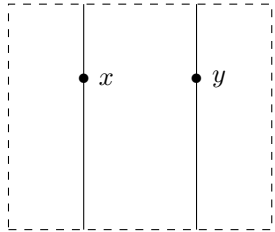
\begin{center}
\begin{tikzpicture}[node distance=1cm, auto]

\draw[dashed] (0,0) rectangle (3.5, 3);
\draw(1,0) to (1,3);
\draw(2.5,0) to (2.5,3);

\node (X) at (1,1) {};
\node (A) at (1,2) {$\bullet$};

\node (Y) at (2.5,1) {};
\node (B) at (2.5,2) {$\bullet$};

\node () [node distance=3mm, right of=A] {$x$};
\node () [node distance=3mm, right of=B] {$y$};

\end{tikzpicture}
\end{center}
represent a unique element of $X$, since every way of reading this diagram produces the same result, for example: $x \ast y = (x \ast y) \cdot 1_X = (x \ast y) \cdot (1_X \ast 1_X) = (x \cdot 1_X) \ast (y \cdot 1_X)$.

Furthermore, the fact that the binary operations are both associative, and satisfy $(w \cdot x) \ast (y \cdot z) = (w \ast y) \cdot (x \ast z)$ for any $w,x,y,z$, means that \textit{any} diagram represent a unique element of $X$, for example, the diagram 
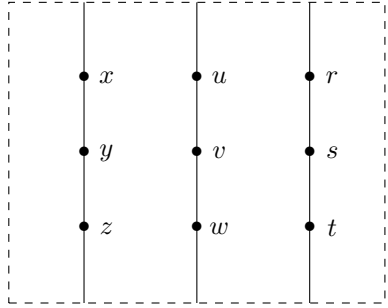
\begin{center}
\begin{tikzpicture}[node distance=1cm, auto]

\draw[dashed] (0,0) rectangle (5, 4);
\draw(1,0) to (1,4);
\draw(2.5,0) to (2.5,4);
\draw(4,0) to (4,4);

\node (X) at (1,1) {$\bullet$};
\node (A) at (1,2) {$\bullet$};
\node (U) at (1,3) {$\bullet$};

\node (Y) at (2.5,1) {$\bullet$};
\node (B) at (2.5,2) {$\bullet$};
\node (V) at (2.5,3) {$\bullet$};

\node (Z) at (4,1) {$\bullet$};
\node (C) at (4,2) {$\bullet$};
\node (W) at (4,3) {$\bullet$};

\node () [node distance=3mm, right of=A] {$y$};
\node () [node distance=3mm, right of=B] {$v$};
\node () [node distance=3mm, right of=C] {$s$};

\node () [node distance=3mm, right of=X] {$z$};
\node () [node distance=3mm, right of=Y] {$w$};
\node () [node distance=3mm, right of=Z] {$t$};

\node () [node distance=3mm, right of=U] {$x$};
\node () [node distance=3mm, right of=V] {$u$};
\node () [node distance=3mm, right of=W] {$r$};

\end{tikzpicture}
\end{center}
represents the composite $(x \ast u \ast r) \cdot (y \ast v \ast s) \cdot (z \ast w \ast t) = \hdots = (x \cdot y \cdot z) \ast (u \cdot v \cdot w) \ast (r \cdot s \cdot t)$. The four points provided the the generators $l_2$, $l'_2$, $r_2$, and $r'_2$, can be distinguished by varying the notation used for them, for example, our four point could be represented by the diagrams below.

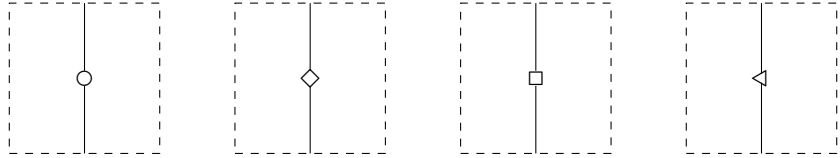
\begin{center}
\begin{tikzpicture}[node distance=1cm, auto]

\draw[dashed] (-1,-1) rectangle (1, 1);
\node (A) at (0,0) {$\medcircle$};
\draw(0,-1) to (0,-0.1);
\draw(0,0.1) to (0,1);

\draw[dashed] (2,-1) rectangle (4, 1);
\node (X) at (3,0) {$\meddiamond$};
\draw(3,-1) to (3,-0.1);
\draw(3,0.1) to (3,1);

\draw[dashed] (5,-1) rectangle (7, 1);
\node (X) at (6,0) {$\square$};
\draw(6,-1) to (6,-0.1);
\draw(6,0.1) to (6,1);

\draw[dashed] (8,-1) rectangle (10, 1);
\node (X) at (9,0) {$\triangleleft$};
\draw(9,-1) to (9,-0.06);
\draw(9,0.06) to (9,1);

\end{tikzpicture}
\end{center}

The proof of the following theorem is similar to the proofs of Theorems \ref{slices for W_n} and \ref{slices_for_weak_interchange}, and hence omitted.

\begin{theorem}\label{slices_for_weak_units}
For $k<n$, the $k^{th}$ slice of the $n$-globular operad $\bs{E_n}$ for $n$-categories with weak units in low dimensions is the symmetric operad for many-pointed $k$-tuple semigroups, of one which is a monoid if $k=n-1$, such that any two of the semigroup operations $\ast$ and $\cdot$ satisfy $(w \cdot x) \ast (y \cdot z) = (w \ast y) \cdot (x \ast z)$, and if we denote the identity of the monoid by $1$, then $1 \ast 1 = 1$ for any semigroup operation $\ast$.
\end{theorem} 

The slices of $\bs{E_n}$ are not plain operads, so Batanin's Conjecture \ref{Conjecture} does not apply directly to higher categories with weak units. However, the slices of $\bs{E_n}$ have a diagrammatic representation analogous to the slices of $\bs{E_3}$, and we may view these diagrams and their associated graphical calculus as decribing the string diagrams for higher categories with weak units. Recall from the introduction that Batanin's conjecture came in two parts: The first part states that if the slices of an $n$-globular operad $\bs{G}$ are plain operads for all $k<n$, then the string diagrams associated to the algebras for $\bs{G}$ form a presheaf category. The second hypothesises is that if the string diagrams associated to a higher category are presheaves, then those higher categories are equivalent to fully weak ones. Thus, if we can use the slices here to show directly that our string diagrams are presheaf categories, then first part of the hypothesis can be circumvented.

\end{document}